\documentclass{amsart}

\usepackage{amsmath,amsthm,amsfonts,amssymb,mathrsfs}
\usepackage[margin=1in]{geometry}
\usepackage[T1]{fontenc}
\usepackage[bookmarks]{hyperref}
\usepackage{verbatim}
\usepackage{todonotes}
\usepackage{appendix}
\usepackage{tikz}
\usepackage{xspace}
\usetikzlibrary{matrix,arrows}
\usepackage{fancyvrb }
\usepackage{color}
\usetikzlibrary{arrows,matrix}

\newtheorem{theorem}{Theorem}[section]

\newtheorem{proposition}[theorem]{Proposition}

\newtheorem*{conj}{Conjecture}

\newtheorem{corollary}[theorem]{Corollary}

\newtheorem{lemma}[theorem]{Lemma}
\input xy
\xyoption{all}

\theoremstyle{definition}

\newtheorem{example}[theorem]{Example}

\newtheorem{remark}[theorem]{Remark}

\newtheoremstyle{named}{}{}{\itshape}{}{\bfseries}{.}{.5em}{\thmnote{#3}}
\theoremstyle{named}
\newtheorem*{namedtheorem}{Theorem}

\DeclareMathOperator{\im}{Im}

\DeclareMathOperator{\codim}{codim}

\DeclareMathOperator{\reg}{reg}

\DeclareMathOperator{\Lie}{Lie}

\DeclareMathOperator{\GL}{GL}

\DeclareMathOperator{\SL}{SL}
\DeclareMathOperator{\sl2}{\mathfrak{sl}_2}
\DeclareMathOperator{\Sp}{Sp}
\DeclareMathOperator{\SO}{SO}
\DeclareMathOperator{\ad}{ad}

\def\mini{\mathrm{min}}
\def\special{\mathrm{sp}}

\newcommand{\calP}{\mathcal{P}}

\newcommand{\fg}{\mathfrak{g}}

\newcommand{\g}{\mathfrak{g}}

\newcommand{\cg}{\mathfrak{c}}

\newcommand{\Ad}{\rm Ad}

\newcommand{\C}{\mathbb{C}}

\newcommand{\0}{\mathcal O}

\title{Local geometry of special pieces of nilpotent orbits}

\author{Baohua Fu, Daniel Juteau, Paul Levy and Eric Sommers}

\date{\today}

\begin{document}

\begin{abstract}
The nilpotent cone of a simple Lie algebra is partitioned into locally closed subvarieties called special pieces, each containing exactly one special orbit. Lusztig conjectured that each special piece is the quotient of some smooth variety by a precise finite group $H$, a result proved for the classical types by Kraft and Procesi.
The present work is about exceptional types.
Our main result is a local version of Lusztig's conjecture: the intersection of a special piece with a Slodowy slice transverse to the minimal orbit in the piece is isomorphic to the quotient of a vector space by $H$.
Along the way,  we complete our previous work on the generic singularities of nilpotent orbit closures, by providing proofs for the last two `exotic' singularities.
Four further, non-isolated, exotic singularities are studied: we show that quotients $\overline{\0_{\mini}(\mathfrak{so}_8)}/\mathfrak{S}_4$, $S^2({\mathbb C}^2/\mu_3)$, $S^3({\mathbb C}^2/\mu_2)$ and $\overline{\0_{\mini}(\mathfrak{sl}_3)}/\mathfrak{S}_4$ occur as Slodowy slice singularities between nilpotent orbits in types $F_4$, $E_6$, $E_7$ and $E_8$ respectively.
We also extend, to fields other than ${\mathbb C}$, the results of Brylinski and Kostant on shared orbit pairs.
In the course of our analysis, we discover a shared pair which is missing from Brylinski and Kostant's classification.
\end{abstract}

\maketitle

\setcounter{section}{-1}

\section{Introduction}
\subsection{Geometry of special pieces}

Let $G$ be a simple algebraic group over $\C$, with Lie algebra $\g$.
For a nilpotent orbit $\0$ in $\g$, we denote by $A(\0)$ the component group of the centralizer of an element of $\0$ in the adjoint group of $G$.
The Springer correspondence \cite{Springer:weyl} is a bijection between irreducible representations of the Weyl group $W$ of $\fg$ and certain pairs $(\0, \rho)$, where $\rho$ is an irreducible representation of  $A(\0)$.
Lusztig \cite{Lusztig:special} introduced special representations of $W$ and observed that they are assigned via the Springer correspondence to certain nilpotent orbits with trivial $\rho$.
This defines a subset of the nilpotent orbits, called the special nilpotent orbits. (Spaltenstein gave an alternative definition of the special orbits in \cite[Chap. III]{Spaltenstein}.)
Although this notion is of key importance in representation theory, its underlying geometric significance has remained somewhat mysterious.
However, Lusztig made a conjecture in \cite{Lusztig:unipotent}, which we will now explain.

The set ${\mathcal N}(\g)/G$ of nilpotent orbits in $\g$ comes equipped with the closure partial ordering: $\0'\leq \0$ if and only if $\0'\subseteq\overline\0$.
For a special nilpotent orbit $\0$,  the {\em special piece} $\calP(\0)$ containing $\0$ is the locally closed subvariety of $\mathfrak{g}$ consisting of the union of the nilpotent  orbits in $\overline{\0}$ which are not contained in the closure of any special nilpotent orbits $\0' <\0$. 
Specifically, we have
$${\mathcal P}(\0):=\overline\0 -\bigcup_{ \0' < \0 \text{ special}} \overline{\0'}.$$

The special pieces give a partition of ${\mathcal N}(\g)$ (this is clear in Spaltenstein's approach \cite{Spaltenstein}).
In \cite{Lusztig:green}, Lusztig conjectured that all special pieces ${\mathcal P}(\0)$ are rationally smooth, that is, the intersection cohomology complex of $\overline\0$ (with characteristic $0$ coefficients), restricted to ${\mathcal P}(\0)$, coincides with the constant sheaf.
Rational smoothness is trivially true for special pieces consisting of only one nilpotent orbit; this covers type $A_n$, where all orbits are special.
For the exceptional types, the conjecture was established by computational methods (see the history in \cite{Lusztig:unipotent}).
In the remaining classical types, Kraft and Procesi \cite{Kraft-Procesi:special} established it as a consequence of the stronger statement that ${\mathcal P}(\0)$ is a quotient of a smooth variety by a finite group.
This led to the conjecture \cite{Lusztig:unipotent}:

\begin{conj}[``Lusztig's special pieces conjecture''] \label{LusztigConj}
Any special piece $\calP(\0)$ is a quotient $Z/H$ of a smooth variety $Z$ by a finite group $H$.
Points of $Z$ have conjugate $H$-stabilizers if and only if their images in $Z/H$ lie in the same nilpotent orbit of $\calP(\0)$.
\end{conj}

Lusztig's statement \cite{Lusztig:unipotent} of the conjecture predicts both the group  $H$, 
which is a subgroup of $\bar{A}(\0)$, {\it Lusztig's canonical quotient} of $A(\0)$, as well as the correspondence between stabilizer subgroups of $H$ up to conjugacy and nilpotent orbits in the special piece.
The group $H$ is also a product of symmetric groups.
By \cite{Kraft-Procesi:special}, the conjecture remains open only for exceptional types.
Achar and Sage proposed \cite{Achar-Sage} a candidate for $Z$ by using the Deligne-Bezrukavnikov theory of perverse coherent sheaves.  
However, smoothness of Achar-Sage's $Z$ seems difficult to prove.

\subsection{Local geometry of special pieces}

Here we take a somewhat different perspective on the conjecture.
In each special piece, there is a unique minimal orbit; we study the transverse slice to that orbit in the piece, which turns out to have a surprisingly simple form:  it is the quotient of a vector space by $H$.  

Let $\mathfrak{h}_{n-1}$ be the $(n-1)$-dimensional reflection representation
of the symmetric group $\mathfrak{S}_{n}$.
Taking into account the symplectic structure of the slice, we have the following result.


\begin{namedtheorem}[Main Theorem] \label{main}
	Consider a special piece in a simple Lie algebra.
	A Slodowy transverse slice to the minimal orbit in the piece is isomorphic to a product of $r$ varieties of the form
	$$
	(\mathfrak{h}_{n-1} \oplus \mathfrak{h}_{n-1}^*)^k/\mathfrak{S}_{n}
	$$
	where $r$ is the number of irreducible factors in the reflection group $H$.
	In the exceptional groups, $r=1$ always, so that $H=\mathfrak{S}_{n}$ and there is a single value of $k$.
	In the classical groups $n=2$ always and the value of $k$ varies across the product.
\end{namedtheorem}

It is easily verified in each case that the above isomorphism is ${\mathbb C}^\times$-equivariant, where the action on the Slodowy slice is given in \S \ref{transverse}.
It therefore follows from \cite[Thm. 3.1]{Nam2} that this is an isomorphism of symplectic singularities.

In the exceptional groups the number $n$ for each special orbit is given in \cite[\S 6]{Lusztig:unipotent}.
The value of $k$ then follows from the dimension of the transverse slice.
Note that the stabilizer of a point of $({\mathfrak h}_{n-1}\oplus{\mathfrak h}_{n-1}^*)^k$ is a parabolic subgroup of $\mathfrak{S}_n$, and points have conjugate stabilizers if and only if their images in the quotient lie in the same nilpotent orbit.
One thus obtains an order-reversing bijection between the parabolic subgroups of $\mathfrak{S}_n$ up to conjugacy and the nilpotent orbits lying in the special piece.
This is identical with the bijection given in \cite[\S 6]{Lusztig:unipotent}.
(See Remark \ref{orderremark}.)

If $n=2$, then the special piece consists of two orbits, and in this case the Main Theorem follows from \cite{FJLS}, since the Slodowy slice is the generic singularity, which in these cases is always isomorphic to the closure of a minimal orbit in a symplectic Lie algebra, i.e., to $\C^{2k}/\mathfrak{S}_2$.
(Note that it is not true, as stated in the paragraph after \cite[Thm. 1.3]{FJLS}, that we always have $k=1$ in exceptional types, see the Table in \S \ref{Section:locallspconj}.)



There are precisely two special orbits with $n>3$: $F_4(a_3)$ in type $F_4$ (with $n=4$) and $E_8(a_7)$ in type $E_8$ (with $n=5$); in both cases $k=1$.
There are seven further special orbits with $n=3$ (see the Table in \S \ref{Section:locallspconj}).
Among these, we have $k=1$ except one case: the special piece of $D_4(a_1)+A_1$ in $E_8$.
We note in particular that this implies that the closure of $D_4(a_1)+A_1$ does not admit a symplectic resolution.
It is interesting to remark that for all cases with $n \geq 3$ and $k=1$, the special orbit $\0$ is even, hence its closure admits a symplectic resolution, given by a Springer map
$T^*(G/P) \to \overline{\0}$.
Restricting to the transverse slice, this gives a symplectic resolution of $(\mathfrak{h}_{n-1} \oplus \mathfrak{h}_{n-1}^*)/\mathfrak{S}_{n}$.
On the other hand, it is known \cite{Fu-Namikawa} that the latter admits a unique symplectic resolution given by the Hilbert-Chow morphism from the Hilbert scheme of points.  In particular, this shows that the Hilbert-Chow resolution is embedded in the Springer resolution, at least for $\mathfrak{S}_3, \mathfrak{S}_4$ and $\mathfrak{S}_5$.
This is a very interesting example, although a similar but simpler phenomenon appears in classical types, where the Hilbert-Chow resolution of $S^2(\mathbb{C}^2/\{\pm 1\})$ appears as a generalized Springer resolution of the Slodowy slice singularity from $\0_{[2,2,2]}$ to $\overline{\0_{[4,2]}}$ in $\mathfrak{sp}_6$ \cite{Fu:wreath}.
(See also \S \ref{otherquotsubsec}.)

In classical types $H = \mathfrak{S}_2\ldots\mathfrak{S}_2$ (with $r$ factors) is the subgroup of $\bar{A}(\0)$ 
defined by Lusztig.  There are positive integers $k_1, k_2, \dots, k_r$ so that the isomorphism in the Main Theorem is to
$$\left(\prod_{i=1}^r({\mathfrak h}_1\oplus{\mathfrak h}_1^*)^{k_i})\right)/\mathfrak{S}_2\ldots\mathfrak{S}_2,$$
with each copy of $\mathfrak{S}_2$ acting diagonally on the respective term in the product.
The proof comes down, as in the $\mathfrak{S}_2$ case above in the exceptional groups, to the observation that the intersection of the Slodowy slice with the special orbit is equal to the product of the closures of the minimal orbits in $\mathfrak{sp}_{2k_i}$ for $i =1, \dots, r$.  Note that we misstated the result in the classical types in  \cite[Thm. 1.3]{FJLS}. 



As a direct consequence of the Main Theorem, we obtain, as a consequence of unibranchness, normality of special pieces in the exceptional types.
This is an important subconjecture of Lusztig's special pieces conjecture (as observed in \cite[Conj. 1.2]{Achar-Sage}).
Similarly (assuming unibranchness), we also obtain a new `geometric' proof of rational smoothness in exceptional types.
On the other hand, we need to use that the special pieces are normal in our proof of the Main Theorem in classical types.

\begin{corollary}
Every special piece in the exceptional types is rationally smooth and normal.
\end{corollary}

\subsection{Exotic singularities of minimal degenerations}

A pair $\0>\0'$ of nilpotent orbits in $\g$ is called a {\it degeneration}; if $\0$ and $\0'$ are adjacent in the partial order, then the pair is a {\it minimal degeneration}.
Associated to any degeneration $\0>\0'$ is a smooth equivalence class of singularities ${\rm Sing}(\0,\0')$, represented by the intersection ${\mathcal S}\cap\overline\0$, where ${\mathcal S}$ is a Slodowy slice at an element of $\0'$.
We call the intersection ${\mathcal S}\cap\overline\0$ a {\it Slodowy slice singularity}.
The minimal degenerations are precisely those degenerations for which the Slodowy slice singularity is isolated.
Two particular cases lie at opposite extremes in the poset ${\mathcal N}(\g)/G$:
\begin{itemize}
    \item  if {$\0=\0_\mini$} is the minimal nilpotent orbit and $\0'=\{ 0\}$ then the associated singularity is called a {\it minimal singularity}, denoted $a_n$, $b_n$, etc.
    \item  if $\0$ is the regular and $\0'$ is the subregular nilpotent orbit then, according to the Brieskorn-Slodowy-Grothendieck theorem, the associated singularity is a simple singularity of the same type as $\g$, denoted $A_n$, $B_n$, etc.
\end{itemize}

In the 1980s, the seminal papers \cite{Kraft-Procesi:GLn,Kraft-Procesi:classical} of Kraft and Procesi completely classified ${\rm Sing}(\0,\0')$ for minimal degenerations $\0>\0'$ in classical Lie algebras.
It turns out that these singularities are simple, minimal, or a union of two simple singularities of type {$A_{2k-1}$}, meeting transversely at the common singular point.
More recently, work of the authors \cite{FJLS} answered this question for exceptional Lie algebras (up to normalization in a few cases).
There are several more exotic possibilities in exceptional types, including the following, each appearing once:
\begin{itemize}
    \item the quotient ${\mathbb C}^4/\mu_3$ where $\mu_3$ is cyclic of order 3;
    \item a partial resolution of the quotient ${\mathbb C}^4/\Gamma$ where $\Gamma$ is a dihedral group of order 10. (This very special example led to the discovery of a
    new family of isolated symplectic singularities with trivial local
    fundamental group, see \cite{BBFJLS}.); 
    \item a non-normal singularity $\mu$ with normalization a simple singularity of type $A_3$;
    \item the quotient $a_2/\mathfrak{S}_2$, where $\mathfrak{S}_2$ acts via an outer involution of $\mathfrak{sl}_3$.
\end{itemize}

In \cite{FJLS} we stated, but did not prove, that the last two singularities are associated to the pairs $(D_7(a_1),E_8(b_6))$ in $E_8$ and $(A_4{+}A_1,A_3{+}A_2{+}A_1)$ in $E_7$.
One of the main goals of this paper is to provide proofs of these statements, completing the picture for minimal degenerations (up to normalization) in exceptional types.
The singularity $\mu$ is explained in terms of a more eye-opening result: it is isomorphic to the singular locus of 
a quotient $a_2/\mathfrak{S}_4$, that we show (Thm. \ref{a2S4thm}) is associated to the non-minimal degeneration $E_8(a_6) > E_8(b_6)$, which contains $D_7(a_1)$ as the unique intermediate orbit.
(This action of $\mathfrak{S}_4$ is generated by outer involutions and is defined in \S \ref{outerspace}.)

\subsection{The quotient $d_4/\mathfrak{S}_4$}

Another (non-isolated) exotic singularity plays a starring role here: we prove in Thm. \ref{d4S4thm} that the singularity associated to the degeneration $F_4(a_3)>A_2$ of special orbits in type $F_4$ is $d_4/\mathfrak{S}_4$, for an action of $\mathfrak{S}_4$ by ``affine diagram automorphisms''.
Note that this degeneration subsumes the exceptional special piece in type $F_4$.
This result will be important in subsequent work \cite{JLS:Duality}, where the singularities associated to minimal degenerations of special orbits (in classical or exceptional types) will be classified.
A version of Lusztig-Spaltenstein duality will also be observed for such singularities.

Thm. \ref{d4S4thm} has an interesting interpretation in terms of the universal cover of the special orbit $\0=\0_{F_4(a_3)}$.
Recall from \cite{Brylinski-Kostant:JAMS} that this universal cover has a canonical affinization $\widetilde{\0}$, which is equipped with an action of $\mathfrak{S}_4$ such that the projection to $\overline\0$ is the quotient by $\mathfrak{S}_4$.
In subsequent work \cite{Fu-Juteau-Levy-Sommers:GeomSP}, the second, third and fourth authors will deduce from our Main Theorem (specifically, from the computations required to obtain it) that $\widetilde{\0}$ is smooth over the special piece containing $\0$; Theorem \ref{d4S4thm} says on the other hand that $\widetilde\0$ has a $d_4$ singularity at points lying above $\0_{A_2}$.
The former statement is Lusztig's special pieces conjecture in this case (which will be addressed in complete generality in \cite{Fu-Juteau-Levy-Sommers:GeomSP}).


We remark the following interesting consequence of the results concerning $d_4/\mathfrak{S}_4$ and $a_2/\mathfrak{S}_4$.
Since $F_4(a_3)$ and $E_8(b_6)$ are even orbits, their closures admit symplectic resolutions.
By restricting to the Slodowy slice, we therefore obtain symplectic resolutions of $d_4/\mathfrak{S}_4$ and $a_2/\mathfrak{S}_4$, a highly non-trivial result.

\subsection{Two other quotient singularities}\label{otherquotsubsec}

Two cases in the Main Theorem, namely those leading to the singularities $({\mathfrak h}_3\oplus {\mathfrak h}_3^*)/\mathfrak{S}_4$ and $({\mathfrak h}_4\oplus{\mathfrak h}_4^*)/\mathfrak{S}_5$, proved more difficult than the others, and were only solved with the aid of computer calculations.
It is natural to ask whether any other symplectic quotient singularities $(V\oplus V^*)/\Gamma$ (where $\Gamma$ is a complex reflection group of rank greater than 1) appear as slices between nilpotent orbits.
Based on our earlier work \cite{FJLS} on generic singularities, we identified two likely candidates (one in type $E_6$ and one in type $E_7$).
It turns out that the computational methods applied to the $({\mathfrak h}_3\oplus{\mathfrak h}_3^*)/\mathfrak{S}_4$ and $({\mathfrak h}_4\oplus{\mathfrak h}_4^*)/\mathfrak{S}_5$ cases can be easily adapted to these other cases.

For the first example, let ${\mathcal S}$ be the Slodowy slice to an element of the nilpotent orbit in $E_7$ with Bala-Carter label $A_4+A_2$, and let $\0$ be the orbit labelled $E_7(a_5)$.
We show that ${\mathcal S}\cap\overline\0$ is isomorphic to $S^3({\mathbb C}^2/\{ \pm 1\})$, equivalently, to $({\mathfrak h}_{B_3}\oplus {\mathfrak h}_{B_3}^*)/W(B_3)$, where ${\mathfrak h}_{B_3}$ denotes a Cartan subalgebra of $\mathfrak{so}_7$.
This generalizes the example of $S^2({\mathbb C}^2/\{ \pm 1\})$ discussed in \cite{Fu:wreath}.
Note, however, that there exists a singularity between adjacent strata in $S^3({\mathbb C}^2/\{ \pm 1\})$ which is equivalent to the non-normal unibranch singularity denoted $m$ in \cite{FJLS}.
(Specifically, these are the strata corresponding to the orbits $D_6(a_2)>A_5+A_1$ in $E_7$; see Rk. \ref{WB3remark}(a) for discussion of the strata.)
Hence, $S^3({\mathbb C}^2/\{ \pm 1\})$ cannot appear in classical types, where failure of normality always comes down to branching in codimension 2 \cite{Kraft-Procesi:classical}.

Secondly, let ${\mathcal S}'$ be the Slodowy slice to an element of the $A_4+A_1$ orbit in $E_6$, and let $\0'$ be the orbit labelled $E_6(a_3)$.
We show that ${\mathcal S}'\cap\overline{\0'}$ is isomorphic to $S^2({\mathbb C}^2/\mu_3)$ (where $\mu_3$ denotes a cyclic subgroup of ${\rm SL}_2$ of order 3).
Note that this is a quotient singularity ${\mathbb C}^4/\Gamma$, where $\Gamma$ is the complex reflection group $G(3,1,2)=\mu_3\wr\mathfrak{S}_3$.
We know of no other nilpotent slice singularities which are quotient singularities for {complex} (not real) reflection groups.

Similarly to \cite{Fu:wreath}, the Hilbert-Chow resolutions of $S^3({\mathbb C}^2/\mu_2)$ and $S^2({\mathbb C}^2/\mu_3)$ are included in the appropriate generalized Springer resolutions of $E_7(a_5)$ and $E_6(a_3)$ (which exist because these orbits are distinguished, therefore even).

We note that for each simply-connected simple algebraic group of exceptional type, there is exactly one nilpotent orbit $\0$ supporting a cuspidal local system, namely  $G_2(a_1)$, $F_4(a_3)$, $E_6(a_3)$, $E_7(a_5)$ and $E_8(a_7)$ respectively \cite{Lusztig:icc}.
For each of those, we found a large slice isomorphic to a symplectic quotient singularity $U/\Gamma$.
Furthermore, the fundamental group $\pi_1(\0)$ plays a role: we have $\Gamma=\pi_1(\0)$ in types $G_2$, $F_4$, $E_8$, which are precisely the cases covered by our Main Theorem; for $E_6(a_3)$ and $E_7(a_5)$, the centre of $G$ acts non-trivially on the cuspidal local system, and $\pi_1(\0)$ is a natural quotient of $\Gamma$.

\subsection{Shared orbit pairs}

Almost all of the results mentioned above involve finite quotients of nilpotent orbit closures or symplectic vector spaces.
These quotients can be understood via elementary invariant-theoretic methods.
As a testing ground, we first apply these methods to the shared orbit pairs classified in \cite{Brylinski-Kostant:JAMS}.
Brylinski and Kostant identified the pairs $\g_0\subset\g$ of Lie algebras with $\g$ simple for which there exist nilpotent orbits $\0_0\subset\g_0$ and $\0\subset\g$ such that there is a finite $G_0$-equivariant morphism from $\overline\0$ to $\overline{\0_0}$.
It was observed in \cite[Prop. 5.1]{Brylinski-Kostant:JAMS} that, when classifying the pairs $(\g_0,\g)$, it suffices to consider the minimal nilpotent orbit in $\g$.
Here we will fill in the gaps and classify all pairs $(\0_0,\0)$.
(We note here that Panyushev \cite{Panyushev:shared-orbits} has also recently, by rather different methods, provided a new confirmation in characteristic zero of those shared orbits which arise as $\mathfrak{S}_2$-quotients of minimal nilpotent orbit closures.)
Our arguments are relatively elementary, so these results carry over to positive characteristic, under mild restrictions.
The complete list is provided in Prop. \ref{p.sharedorbits}.
Surprisingly, we have discovered one pair $(\g_0,\g)$ which is missing from \cite{Brylinski-Kostant:JAMS}, as explained in the following example.

\begin{example}
Let $\g^\dag=\mathfrak{so}_8\subset\g'=\mathfrak{so}_9$ in the standard way.
Up to $\SO_8$-conjugacy, there are three subalgebras of $\g^\dag$ isomorphic to $\mathfrak{so}_7$; these are permuted transitively by triality automorphisms of $\g^\dag$.
Let $\g$ be one of the non-standard subalgebras isomorphic to $\mathfrak{so}_7$ (e.g. the fixed points for the graph involution swapping $\alpha_1$ and $\alpha_3$, using Bourbaki's numbering of simple roots).
Note that $\g^\dag=\g\oplus V$, where $V$ is the natural module for $\g$, but $\g'=\g\oplus V\oplus V(\varpi_3)$.
(This example contradicts the assertion in the penultimate paragraph of the proof of \cite[Thm. 5.9]{Brylinski-Kostant:JAMS}, that no triple $\g\subsetneq\g^\dag\subsetneq\g'$ exists with ${\rm rank}(\g^\dag)={\rm rank}(\g')$; Brylinski and Kostant considered the triples $\mathfrak{so}_n\subset\mathfrak{so}_{n+1}\subset\mathfrak{so}_{n+2}$ but omitted to account for the case $n=7$ and a non-standard embedding.)
Let $\0'$ (resp. $\0^\dag$, $\0$) be the minimal nilpotent orbit in $\g'$ (resp. the orbit with partition $[3,1^5]$ in $\g^\dag$, $[3,2^2]$ in $\g$).
It was proved by Brylinski and Kostant that the projection $\g'\rightarrow\g^\dag$ identifies $\overline{\0^\dag}$ with $\overline{\0'}/\mathfrak{S}_2$.
On the other hand, we will show later (see Example \ref{sonexamples}(a)) that the projection $\g^\dag\rightarrow \g$ identifies $\overline{\0}^\dag$ with the normalization of $\overline{\0}$.
Hence $(\g,\g')$ is missing from the list in \cite[Thm. 5.9]{Brylinski-Kostant:JAMS}.
\end{example}

Our approach to shared orbit pairs also helps to shed light on the unique case in the Main Theorem with $n>2$ and $k>1$.
Specifically, let $\0$ be the orbit in type $E_8$ with label $D_4(a_1)+A_1$ and let ${\mathcal S}$ be a Slodowy slice to an element of the orbit $2A_2$.
We show in \S \ref{doubled} that ${\mathcal S}\cap\overline\0$ is isomorphic to $(d_4\times d_4)/\mathfrak{S}_3$ (a kind of `doubled' version of $g_2^{\special}=d_4/\mathfrak{S}_3$).
As a by-product, one obtains a new proof of the branched singularity $2g_2^{\special}$ associated to the degeneration $(D_4(a_1), 2A_2)$.
It also provides a new explanation for the occurrence of the four-dimensional non-normal singularity, denoted $m'$ in \cite{FJLS}.

\subsection{Structure of the paper}
The paper is organised as follows.
After some preliminaries in \S \ref{prelim}, in \S \ref{sharedorbits} we reprove the results of \cite{Brylinski-Kostant:JAMS} on existence of shared orbit pairs, and we establish the appearance of the `doubled version' $(d_4\times d_4)/\mathfrak{S}_3$ of $g_2^{\special}$ in type $E_8$.
In \S \ref{almostsharedorbits} we develop an approach to studying {\it almost shared orbits}: we define a nilpotent $G$-orbit $\0$ in $\g$ to be {\it almost shared} with a subgroup $G_0$ if there is a $G_0$-orbit in $\0$ of codimension 1.
Our main trick is to consider the expanded group $G_0\times{\mathbb G}_m$ (where ${\mathbb G}_m$ acts by scaling), which has an open orbit in $\0$.
This gives us a means to understand certain quotients $\overline\0/\Gamma$, in a similar vein to our arguments for shared orbit pairs.
We apply this to produce concrete descriptions of (i) $a_2/\mathfrak{S}_2$; (ii) ${\mathfrak h}_2^{\oplus r}/\mathfrak{S}_3$, for actions defined in \S \ref{outersln} and \S \ref{S3subsec}.
In \S \ref{F4sec} we further modify the previous approach in order to understand the quotient $d_4/\mathfrak{S}_4$.
Although we cannot describe it as the closure of an orbit, we do find (after some work) an explicit description as the closure of the image of a certain morphism.
A quick computation then verifies that the Slodowy slice singularity associated to the exceptional special degeneration is isomorphic to $d_4/\mathfrak{S}_4$.  
In \S \ref{compsec}, five further singularities are studied using computational arguments: (i) the singularities for the degenerations $F_4(a_3)>A_2+\tilde{A}_1$ and
$E_8(a_7)>A_4+A_3$ are shown to be $({\mathfrak h}_3\oplus{\mathfrak
h}^*_3)/\mathfrak{S}_4$ and $({\mathfrak h}_4\oplus{\mathfrak
h}^*_4)/\mathfrak{S}_5$ respectively; (ii) the singularities for the degenerations $E_7(a_5)>A_4+A_2$ in $E_7$ and $E_6(a_3)>A_4+A_1$ in $E_6$ are shown to be $S^3({\mathbb C}^2/\{ \pm 1\})$ and $S^2({\mathbb C}^2/\mu_3)$ respectively;
(iii) the exotic quotient $a_2/\mathfrak{S}_4$ is identified as the singularity of $E_8(a_6)>E_8(b_6)$.
In \S \ref{Section:locallspconj} we complete the proof of the Main Theorem.
\vspace{0.3 cm}

{\em Acknowledgements:}
Several of these results required computer verification.
We made extensive use of GAP, Magma and Singular in preparing the paper; all of the essential computational results were ultimately verified in GAP, with one exception (see the paragraph before Thm. \ref{a2S4thm}).
Our GAP files are available from the authors on request. 
We thank Chris Parker for suggesting to do the hardest computer calculations column by column, Ellen Goldstein for correspondence relating to normality of nilpotent orbit closures in positive characteristic, Harm Derksen for some background on rings of invariants, and Amihay Hanany, Deshuo Liu, Antoine Bourget and Julius Grimminger for discussions of the connection to the mathematical physics literature.
We are grateful to Gwyn Bellamy for informative discussions about rational Cherednik algebras in the wreath product case.
B. F. is supported by the NSFC grant no. 12288201 and CAS Project for Young Scientists in Basic Research, grant no. YSBR-033.

\section{Preliminaries}\label{prelim}

\subsection{Assumptions in positive characteristic}

Let $\Phi$ be an irreducible root system with basis of simple roots $\{ \alpha_1,\ldots ,\alpha_r\}$ and let $\hat\alpha=\sum_i m_i\alpha_i$ be the highest root.
Recall that the prime $p$ is said to be {\it good} for $\Phi$ if $p>m_i$ for all $i$, and is {\it very good} if in addition it doesn't divide the index of the root lattice in the weight lattice.
Specifically, $2$ is bad (i.e. not good) outside type $A$; $3$ is bad for all exceptional types; $5$ is bad in type $E_8$; if $p$ is good but not very good then $\Phi$ has type $A_n$ and $p|(n+1)$.
A prime is good (resp. very good) for a reducible root system if it is good (resp. very good) for every irreducible component.
If $G$ is a reductive group over $k$ then we say that $p$ is (very) good for $G$ if it is (very) good for the root system of $G$.

It is now fairly well known (see e.g.~\cite{Premet}) that if the characteristic is very good for $G$ then the classification of the nilpotent orbits in $\g$ (including information on dimensions and centralizers, for example summarized in \cite{Carter}) is identical with the complex Lie algebra of the same type.
Furthermore, the closure relation on nilpotent orbits is unchanged in very good characteristic or characteristic zero \cite[\S 5(E)]{Beynon-Spaltenstein}.
We will use these facts without any further reference.
Our assumptions on positive characteristic $p$ will always ensure that $p$ is very good for $G$.

\subsection{Transverse slices}\label{transverse}
Let $G$ be a reductive group acting on its Lie algebra $\g$ and let $x\in\g$.
By a {\it transverse slice} at $x$, we mean an affine linear subvariety ${\mathcal S}:=x+{\mathfrak v}\subset\g$ such that ${\mathfrak v}\oplus [\g,x]=\g$.
This condition ensures that the morphism $G\times {\mathcal S}\rightarrow \g$, $(g,s)\mapsto g\cdot s$ is smooth at $x$, and the dimension of ${\mathcal S}$ is minimal subject to this condition.
In other words, ${\mathcal S}$ is a transverse slice in $\g$ as per the definition in \cite[\S 5.1]{Slodowy:book}.
By base change, ${\mathcal S}\cap X$ is a transverse slice in $X$ for any $G$-stable closed subset of $\g$ containing $x$.

If ${\rm char}\, k=0$ and $x=f$ is a nilpotent element of $\g$, then the standard choice of transverse slice at $f$ is the {\it Slodowy slice} $f+\g^e$, where $\{ h,e,f\}$ is an $\mathfrak{sl}_2$-triple.
In positive characteristic, the Slodowy slice is not in general transverse.
We first note that the canonical inclusion $[\g,f]\subset T_f(\0_f)$ can be proper.
Equality holds if and only if the centralizer $G^f$ is smooth, which is always the case in very good characteristic (see the discussion in \cite[\S I.5]{sands}).
This will apply to all of our orbits, hence the dimension of a transverse slice at $f$ is equal to $\dim G^f$.

The second problem to overcome concerns the choice of $\mathfrak{sl}_2$-triple containing $f$, since two such triples no longer need be conjugate.
Recall that there is a canonical {\it $p$-operation} $x\mapsto x^{[p]}$ on $\g$, and an element $x$ is nilpotent if and only if $x^{[p^N]}=0$ for $N$ large enough.
Our assumptions on $p$ will ensure that all of the nilpotent orbits we consider are contained in the {\it restricted nullcone} ${\mathcal N}_1:=\{ x\in\g :x^{[p]}=0\}$.
Then a natural choice of $\mathfrak{sl}_2$-triple can be obtained via the differential of an {\it optimal $\SL_2$-homomorphism} \cite{McNinchoptimal}, or equivalently by reducing modulo $p$ from an $\mathfrak{sl}_2$-triple over ${\mathbb Z}$.
(The equivalence of these is a consequence of the classification of nilpotent orbits in very good characteristic, see \cite{Premet}.)
The condition $x\in{\mathcal N}_1$ entails that all Jordan blocks of $\ad(x)$ are of size $\leq p$.
Let $\phi:\SL_2\rightarrow G$ be an optimal homomorphism for $f$ and let ${\mathfrak l}=\Lie(\im\phi)$.
Then $\g$ is a rational representation for $\SL_2$, by means of $\phi$.
Let $\lambda(t)=\phi({\rm diag} (t,t^{-1}))$.
Consider the stronger property that all weights of $\Ad\, \lambda(t)$ on $\g$ are strictly less than $p$.
This entails that every irreducible $\SL_2$-stable factor in $\g$ is of the form $V(i)$ with $i<p$, hence is tilting.
It follows that $\g$ is a semisimple module for $\SL_2$, and thus the same argument as in characteristic zero establishes that $f+\g^e$ is transverse at $f$.

In this paper we will always make this assumption on the characteristic, so that the Slodowy slice is a {\it bona fide} transverse slice.
(Of course, we also require $p$ to be very good.)
We note that if $\lambda(k^\times)$ is represented by a weighted Dynkin diagram then this condition is equivalent to: $\langle \lambda,\hat\alpha\rangle<p$.
We can therefore verify the necessary assumptions: (i) $p>3$ for the orbit $2A_2$ in type $E_8$ (\S \ref{doubled}, where in any case we require $p>7$); (ii) $p>5$ for the orbit $A_3+A_2+A_1$ in type $E_7$ (\S \ref{outersln}, where in any case we require $p>7$); (iii) $p>3$ for the minimal nilpotent orbit in type $G_2$ (\S \ref{S3smallsing}); (iv) $p>5$ for the orbit labelled $2A_2+2A_1$ in type $E_8$ (\S \ref{S3bigsing}); (v) $p>3$ for the orbits $A_2$ and $A_2+\tilde{A}_1$ in type $F_4$ (\S \ref{F4sec}).

The Slodowy slice (in arbitrary characteristic) comes equipped with a ${\mathbb C}^\times$-action, given by:
$t\cdot y = t^2 {\rm Ad}(\lambda(t))(y)$ where $\lambda$ is an optimal cocharacter as above.
In what follows (for example, throughout \S \ref{almostsharedorbits}) we will use ${\mathbb G}_m$ to denote the multiplicative group producing this action on the Slodowy slice, reserving the notation ${\mathbb C}^\times$ for subgroups of $G$.
Note that the ${\mathbb G}_m$-action preserves the intersection of the slice with any nilpotent orbit.
In characteristic zero, writing $x\in\g^e$ as $\sum_{i\geq 0} x_i$ where $[h,x_i]=i x_i$, the action is given by:
$$t\cdot \left(f+\sum_{i\geq 0} x_i\right) = f+\sum_{i\geq 0} t^{i+2} x_i.$$

\subsection{Branching}\label{branching}

Recall that an irreducible variety is {\it unibranch} at $x$ if the normalization is locally a homeomorphism at $x$.
As observed in \cite[\S 2.4]{FJLS}, the branches of a Slodowy slice singularity are equal to its irreducible components.
The number of branches can be determined from tables of Green functions, as explained in \cite[\S 5(F)]{Beynon-Spaltenstein}.
This calculation is completely uniform in very good characteristic, so the number of branches will be the same when we consider Slodowy slice singularities in positive characteristic (since we always assume the characteristic is very good).

Excluding the minimal degenerations already studied in \cite{FJLS}, there is one minimal special degeneration in exceptional types which is not unibranch: in type $E_8$, $\0_{D_4(a_1)}$ has two branches at $\0_{2A_2}$.
As established in Cor. \ref{2g2spcor}, the corresponding Slodowy slice singularity is a union of two minimal special singularities $g_2^{\special}$, permuted transitively by the component group $\mathfrak{S}_2$ of the chosen element of $\0_{2A_2}$.
All other degenerations considered in this paper (including those in \S \ref{compsec} which are not minimal special degenerations) are unibranch.
We will routinely use this fact, as follows: if $\overline\0$ is unibranch at $f\in\0'$ then the Slodowy slice singularity at $f$ is irreducible, hence it is equal to the closure of any subset of ${\mathcal S}\cap\overline\0$ of dimension $\codim_{\overline\0}\0'$.

\section{Shared orbit pairs}\label{sharedorbits}

In this and the next two sections we will study various quotients for actions of finite groups on simple Lie algebras.
In the first instance, we provide new proofs of Brylinski and Kostant's results \cite{Brylinski-Kostant:JAMS} on shared orbit pairs.
(See Prop. \ref{p.sharedorbits} for the classification.)
We assume throughout this section that the ground field $k$ is algebraically closed of characteristic $\neq 2$.

\subsection{An involution of $\mathfrak{so}_{n+1}$}
\label{sonsubsec}

In this section we denote by ${\rm O}_n$ (resp. $\SO_n$) the group of $n\times n$ invertible matrices (resp. matrices of determinant 1) preserving the standard bilinear form on $k^n$, i.e. satisfying $gg^T=I_n$, and by $\mathfrak{so}_n$ the Lie algebra of $\SO_n$, i.e. the Lie algebra of anti-symmetric $n\times n$ matrices.
For any $m>0$ we include ${\rm O}_n$ in ${\rm O}_{n+m}$ via the map $g\mapsto \begin{pmatrix} g & 0 \\ 0 & I_{m} \end{pmatrix}$, and similarly for $\mathfrak{so}_n$.
Let ${\mathfrak g}=\mathfrak{so}_{n+1}$.
Conjugation by the diagonal matrix ${\rm diag}(1,1,\ldots ,1,-1)$ determines an involution $\theta$ of ${\rm O}_{n+1}$ (resp. ${\mathfrak g}$) with fixed points $\pm{\rm O}_n$ (resp. $\g_0:=\mathfrak{so}_n$).
From now on, ${\mathfrak{S}}_2$ will always act on $\mathfrak{so}_{n+1}$ via $\langle\theta\rangle$ (or perhaps some conjugate), unless otherwise specified.
Clearly, any element of $\mathfrak{so}_{n+1}$ is of the form:
$$\begin{pmatrix} M & u \\ -u^T & 0 \end{pmatrix}$$
where $u\in k^n$ is a column vector and $M$ is in $\mathfrak{so}_n$.
Identifying ${\mathfrak g}$ with the set of pairs $(M,u)$, we see that $\theta$ maps $(M,u)\mapsto (M,-u)$.
With this set-up we have an immediate description of the quotient $\mathfrak{so}_{n+1}/\mathfrak{S}_2$.
Let $\Sigma_n$ denote the vector space of symmetric $n\times n$ matrices.
Clearly, $\Sigma_n$ is stable under conjugation by elements of ${\rm O}_n$.

\begin{lemma}\label{solemma}
The map $\varphi:(M,u)\mapsto (M,uu^T)$ induces an ${\rm O}_n$-equivariant isomorphism of $\mathfrak{so}_{n+1}/\mathfrak{S}_2$ with $\mathfrak{so}_n\times X$ where $X$ is the cone of elements of $\Sigma_n$ of rank at most 1.
\end{lemma}

\begin{proof}
Choose a basis $\{ v_1,\ldots ,v_{n(n+1)/2}\}$ for ${\mathfrak g}$ consisting of eigenvectors for $\theta$.
We can assume that $\theta(v_i)=v_i$ for $i\leq n(n-1)/2$ and $\theta(v_i)=-v_i$ for $i>n(n-1)/2$.
Let $\{ x_1,\ldots ,x_{n(n+1)/2} \}$ be a dual basis for $\g^*$.
Then $k[\g]^{\mathfrak{S}_2}$ is clearly generated by the coordinate functions $x_i$ with $1\leq i\leq n(n-1)/2$, along with all products $x_ix_j$ where $n(n-1)/2< i\leq j\leq n(n+1)/2$.
\end{proof}

We note the further consequence of this description:

\begin{corollary}\label{nec?}
If $Y$ is any $\theta$-stable closed subset of ${\mathfrak g}$ then the restriction of $\varphi$ to $Y$ induces an isomorphism of $Y/\mathfrak{S}_2$ with $\varphi(Y)$.
If $Y$ is ${\rm O}_n$-stable, then this isomorphism is ${\rm O}_n$-equivariant.
\end{corollary}

\begin{proof}
Let $J\subset k[{\mathfrak g}]$ be the defining ideal of $Y$.
Since $J$ is $\theta$-stable and therefore spanned by eigenvectors for $\theta$, an element of $k[Y]=k[{\mathfrak g}]/J$ is fixed by $\theta$ if and only if it has a lift in $k[{\mathfrak g}]$ which is fixed by $\theta$.
This shows that a set of generators for $k[{\mathfrak g}]^{\mathfrak{S}_2}$ maps to a set of generators for $k[Y]^{\mathfrak{S}_2}$.
Hence the restriction of $\varphi$ to $Y$ induces a quotient morphism too.
\end{proof}

\begin{remark}
Another way of stating the main conclusion of Cor. \ref{nec?} is that $k[\g]^{\mathfrak{S}_2}\rightarrow k[Y]^{\mathfrak{S}_2}$ is surjective for any closed $\theta$-stable subvariety (in fact, subscheme) of $\g$.
The argument clearly generalises to the following (presumably well-known) statement: let $\Gamma$ be a finite group acting on an affine variety $U$ such that $k[U]$ is a completely reducible $k\Gamma$-module, let $\pi:U\rightarrow U/\Gamma$ be the quotient and let $V$ be a $\Gamma$-stable closed subset of $U$.
Then the restriction of $\pi$ to $V$ induces an isomorphism $V/\Gamma\cong\pi(V)$.
We will frequently use this more general statement.
(To see that the assumption of complete reducibility is necessary, let $k$ have characteristic $p>0$, let $\Gamma$ be cyclic of order $p$, let $U=k\Gamma$ with the action of $\Gamma$ by left multiplication and let $V$ be any proper non-trivial $\Gamma$-stable subspace of $U$.)
\end{remark}

\begin{lemma}\label{squarezero}
If ${\mathcal O}$ is an ${\rm O}_{n+1}$-orbit of elements of $\g$ with square zero, then the restriction to $\overline{\mathcal O}$ of the projection map $\pi:{\mathfrak g}\rightarrow{\mathfrak g}_0$ induces an ${\rm O}_n$-equivariant isomorphism
$$\overline{\mathcal O}/\mathfrak{S}_2 \cong \pi(\overline{\mathcal O}).$$
\end{lemma}

\begin{proof}
We remark first of all that ${\mathcal O}$ is $\theta$-stable by definition.
Since $$\begin{pmatrix} M & u \\ -u^T & 0\end{pmatrix}^2 = \begin{pmatrix} M^2-uu^T & Mu \\ -u^TM & -u^Tu\end{pmatrix},$$
then any element $(M,u)\in\overline{\mathcal O}$ satisfies $M^2=uu^T$.
Thus the restriction of $\varphi$ to $\overline\0$ has the form $(M,u)\mapsto (M,M^2)$, hence factors through the projection to ${\mathfrak g}_0$.
We now use Corollary \ref{nec?} to establish the result.
\end{proof}

\begin{corollary}\label{mincor}
a) For any $r\leq \frac{n+1}{4}$, the projection $\mathfrak{so}_{n+1}\rightarrow\mathfrak{so}_n$ induces an ${\rm O}_n$-equivariant isomorphism:
$$\overline{\0_{[2^{2r},1^{a}]}}/\mathfrak{S}_2\rightarrow \overline{\0_{[3,2^{2r-2},1^a]}}$$
where $a=n+1-4r$.

b) Let $\0=\0_{[2^{2r}]}\subset\mathfrak{so}_{4r}$ and let $\0^{I}$, $\0^{II}$ denote the distinct $\SO_{4r}$-orbits in $\0$.
Then the composition $$\overline{\0}^I\hookrightarrow\overline\0\rightarrow \overline\0/\mathfrak{S}_2\cong \overline{\0_{[3,2^{2r-2}]}}$$
is the normalization map of the orbit closure $\overline{\0_{[3,2^{2r-2}]}}$.
This orbit closure is non-normal if $r>1$.
\end{corollary}

\begin{proof}
For (a), we can describe $\overline{\mathcal O}=\overline{\0_{[2^{2r},1^a]}}$ as the set of elements of ${\mathfrak g}\subset\mathfrak{gl}_{n+1}$ with rank at most $2r$ and square zero.
If $(M,u)$ determines an element of $\overline{\mathcal O}$ then clearly ${\rm rank}\, M\leq 2r$ too.
Furthermore, we have $M^2=uu^T$ and $Mu=0$, so that $M^2$ has rank at most 1 and $M^3=0$.
This shows that $M\in\overline{{\mathcal O}_0}=\overline{\0_{[3,2^{2r-2},1^a]}}$.
By equality of dimensions, we must have $\pi(\overline{\mathcal O})=\overline{{\mathcal O}_0}$.

For (b), we recall by \cite[17.3]{Kraft-Procesi:classical} that $\overline{\0}^{I}$ and $\overline\0^{II}$ are normal.
(The proof of this in odd positive characteristic goes through exactly as in \cite{Kraft-Procesi:classical}, using the crucial fact in Goldstein's PhD thesis \cite[Prop. 5.2a]{Goldstein-PhD} that a regular function on $\0$ extends to $\overline\0$ if and only if it extends to the union of $\0$ and the orbits of codimension 2.
Note that here there is a unique ${\rm O}_{4r}$-orbit of codimension 2 in $\overline\0$.)
Since $\0/\mathfrak{S}_2\cong\0^I$, it follows that $\overline{\0}^I$ is the normalization of $\overline\0/\mathfrak{S}_2$.
The non-normality of $\overline{\0_{[3,2^{2r-2}]}}$ for $r>1$ follows from the fact \cite{Kraft-Procesi:classical} that the singularity of its degeneration to $\0_{[3,2^{2r-4},1^4]}$ is $2A_1$.
(This argument also establishes non-normality in odd positive characteristic, by \cite[Thm. 3.3]{Xiao-Shu}.)
\end{proof}

\begin{remark}
a) Cor. \ref{mincor} gives a generalization of \cite[Cor. 5.4(ii) and Thm. 5.9(ii)]{Brylinski-Kostant:JAMS} even for $k=\mathbb{C}$. This gives new examples of shared nilpotent orbits (while in \cite{Brylinski-Kostant:JAMS}, only covers from minimal nilpotent orbit closures are considered).

b) It is well known that the projection $\mathfrak{so}_7 \to \fg_2$ induces an isomorphism of $\overline{\0}_{\mini}$ with the normalization of the 8-dimensional nilpotent orbit closure in $\fg_2$, see \S \ref{G2case}.
It is somewhat surprising that the similar example appearing in Cor. \ref{mincor}(b) has (as far as we are aware) not been studied before.
In contrast with the $G_2$ case however, here the normalization map has degree 2 on the boundary.
It turns out that there is one final example of a shared orbit pair $(\0,\0_0)$ with $\overline{\0_0}$ non-normal, as discussed in the following example.
\end{remark}

\begin{example}\label{sonexamples}
a) Let $\sigma$ be a triality (i.e. outer of order 3) automorphism of $\mathfrak{so}_8$ and let $\0$ be the nilpotent orbit with partition $[3,1^5]$.
Then $\sigma^{-1}(\0)$ is one of the $\SO_8$-orbits $\0_{[2^4]}^{I}, \0_{[2^4]}^{II}$.
(This can be seen from the fact that $e_{\alpha_3}+e_{\alpha_4}\in\0$, or from weighted Dynkin diagrams.)
Hence, the projection from $\mathfrak{so}_8$ to $\sigma(\mathfrak{so}_7)$ identifies $\overline\0$ with the normalization of the orbit in $\sigma(\mathfrak{so}_7)$ with partition $[3,2^2]$.
Combining this with the $\mathfrak{S}_2$-quotient $\overline{\0_{[2^2,1^5]}}\rightarrow\overline\0$, this establishes the claim in the Introduction: the pair $(\sigma(\mathfrak{so}_7),\mathfrak{so}_9)$ is missing from the classification in \cite{Brylinski-Kostant:JAMS}.
Note that this is the only example (see Prop. \ref{p.sharedorbits}) of connected shared orbits $(\0,\0_0)$ for which $\overline{\0_0}$ is non-normal but the map $\overline\0\rightarrow \overline{\0_0}$ is not birational.

b) Here we give a geometric explanation of some special cases of Cor. \ref{mincor}. Let ${\rm OG}(k,n)$ be the orthogonal Grassmannian, which parametrises  isotropic $k$-dimensional vector sub-spaces in an orthogonal $n$-dimensional vector space. It is well known that ${\rm OG}(n,2n)$ has two isomorphic components, denoted ${\rm OG}^{\pm}(n,2n)$. Moreover we have a natural isomorphism
 ${\rm OG}(n, 2n+1) \simeq {\rm OG}^\pm(n+1, 2n+2)$. There are two cases:

(i)  $n=2r$ is even. The Springer maps associated to ${\rm OG}(n, 2n+1) \cong {\rm OG}^\pm(n+1, 2n+2)$ are given by 
$$
\mu: T^*({\rm OG}(2r, 4r+1)) \to  \overline{\0}_{[3,2^{2r-2},1^2]} \subset \mathfrak{so}_{4r+1}, \quad \quad \nu: T^*({\rm OG}^\pm(2r+1, 4r+2)) \to  \overline{\0}_{[2^{2r},1^2]} \subset \mathfrak{so}_{4r+2}.
$$
One computes that $\mu$ is of degree 2 while $\nu$ is birational. This gives the degree 2 cover: $\overline{\0}_{[2^{2r},1^2]} \to  \overline{\0}_{[3,2^{2r-2},1^2]}$, which also induces
covers of orbits in its boundary.

(ii)  $n=2r+1$ is odd. The Springer maps associated to ${\rm OG}(n, 2n+1) \cong {\rm OG}^\pm(n+1, 2n+2)$ are given by 
$$
\mu: T^*({\rm OG}(2r+1, 4r+3)) \to  \overline{\0}_{[3,2^{2r}]} \subset \mathfrak{so}_{4r+3}, \quad \quad \nu: T^*({\rm OG}^\pm(2r+2, 4r+4)) \to  \overline{\0}_{[2^{2r+2}]}^{I,II} \subset \mathfrak{so}_{4r+4}.
$$
One computes that $\mu$ and $\nu$ are both birational. This gives the normalization map: $\overline{\0}_{[2^{2r+2}]}^{I,II} \to  \overline{\0}_{[3,2^{2r}]}$, which is not an isomorphism as the latter is not normal.
It is interesting to note that this normalization map induces degree 2 covers for orbits in its boundary.
\end{example}

\subsection{An involution of $\mathfrak{gl}_{2n}$}\label{spsubsec}

We began with the pair $(\mathfrak{so}_n,\mathfrak{so}_{n+1})$ because that case is especially straightforward to work out.
But we can apply similar arguments to establish the other shared orbit pairs.
For all but one of these, $\g_0$ is the fixed point subalgebra for a finite group of automorphisms of $\g$.
We first consider the case of a cyclic group of order two, which covers three of 
the remaining pairs $(\g_0,\g)$.
Let $G$ be an algebraic group equipped with an involutive automorphism $\theta$, let $\g=\Lie(G)$, $G_0=G^\theta$, $\g_0=\Lie(G_0)$.
By abuse of notation we write $\theta$ for the associated automorphism of $\g$.
Let $U$ be the $(-1)$-eigenspace for $\theta$, so that $\g={\mathfrak g}_0\oplus U$.
Then we can clearly repeat the arguments of Lemma \ref{solemma} and Cor. \ref{nec?} to obtain:

\begin{lemma}\label{genlem}
The morphism $\varphi:\g={\mathfrak g}_0\oplus U\rightarrow {\mathfrak g}_0\times S^2U$, $$x+y\mapsto (x, y\otimes y)$$ induces a $G_0$-equivariant isomorphism of $\g/\langle \theta\rangle$ with ${\mathfrak g}_0\times X$, where $X$ is the cone of elements of $S^2U\cong\Sigma_{\dim U}$ of rank at most $1$.
Furthermore, if $Y$ is any closed $\theta$-stable subvariety of $\g$ then $Y/\langle \theta\rangle\cong \varphi(Y)$.
\end{lemma}

In each case we have to establish an analogue of Lemma \ref{squarezero}.
A statement equivalent to Lemma \ref{squarezero} is that there exists a linear (${\rm O}_n$-equivariant) map $\varepsilon:S^2\mathfrak{so}_n\rightarrow\Sigma_n=S^2 U$ (given by $M\otimes N\mapsto \frac{1}{2}(MN+NM)$) such that for $(M,u)$ belonging to a square zero orbit $\0$ in $\mathfrak{so}_{n+1}$, $\varepsilon(M\otimes M)=u\otimes u$.
Then the restriction of the quotient morphism $(M,u)\mapsto (M,u\otimes u)$ to  $\overline\0$ is of the form $(M,u)\mapsto (M,\varepsilon(M\otimes M))$, hence factors through the projection to $\g_0$.
We want to construct an analogous map $\varepsilon$ for each of the other cases.

Let $J'=\begin{pmatrix} 0 & J \\ -J & 0 \end{pmatrix}$, where $J$ is the $n\times n$ matrix with $1$ on the anti-diagonal and 0 elsewhere, let $\g=\mathfrak{gl}_{2n}$ and let $\theta:\g\rightarrow\g$, $x\mapsto J'x^TJ'$.
Then $\g_0=\mathfrak{sp}_{2n}$ is the Lie subalgebra of $\mathfrak{gl}_{2n}$ preserving the symplectic form determined by $J'$.
From now on we fix this $\mathfrak{S}_2$-action on $\mathfrak{gl}_{2n}$.
The map $v\otimes w\mapsto (vw^T+wv^T)J'$ defines an isomorphism of $\g_0$-modules from $S^2 V$ to $\g_0$, where $V$ is the natural module for $\g_0$.
Similarly, the $-1$ eigenspace for $\theta$ is isomorphic to $\Lambda^2 V$.
By Lemma \ref{genlem}, the quotient morphism takes the form $\varphi:\g\rightarrow \g_0\times S^2(\Lambda^2 V)$.

\begin{lemma}\label{splem}
Suppose ${\rm char}\, k\neq 2,3$.
Let $(\g_0,\g)=(\mathfrak{sp}_{2n},\mathfrak{gl}_{2n})$ (where $n\geq 2$), let 
$\0$ be the minimal nilpotent orbit in $\g$ and let $\0_0\subset\g_0$ be the orbit of a short root element.
Then the restriction of the quotient morphism to $\overline\0$ factors through the projection to $\g_0$, hence induces an isomorphism of $\overline\0/\mathfrak{S}_2$ with $\overline{\0_0}$.
\end{lemma}

\begin{proof}
Let $\{ v_1,\ldots 
,v_{2n}\}$ be the standard basis for $V$.
Observe that $\dim\0_0=\dim\0$.
We first find 
an element of $\0$ of the form $e+u$ where $e\in\0_0$ and $u\in\Lambda^2 V$.
Denote by $e_{ij}$ the matrix with $1$ in the $(i,j)$ position and $0$ elsewhere.
We choose $e=e_{1,2n-1}+e_{2,2n}$, which belongs to $\0_0$ and corresponds to $v_1\otimes v_2\in S^2 V$.
Furthermore, $u=e_{1,2n-1}-e_{2,2n}$ is clearly a highest weight vector in 
$\Lambda^2 V$ corresponding to $v_1\wedge v_2$, and $e+u=2e_{1,2n-1}\in\0$.

By dimensions, $G_0\cdot (e+u)$ is dense in $\0$.
Thus it only remains to show that the restriction of the map $x+y\mapsto 
(x,y\otimes y)$ to $G_0\cdot (e+u)$ factors through the projection to $\g_0$.
To see this, we note that every element of $G_0\cdot (e+u)$ is of the form 
$(gv_1\otimes gv_2,gv_1\wedge gv_2)=(v\otimes w, v\wedge w)$ as an element of $S^2 V\oplus\Lambda^2 V$.
It will therefore suffice to show that the map $$V\times V\rightarrow S^2(\Lambda^2 V), \quad (v,w)\mapsto (v\wedge w)\otimes (v\wedge w)$$ factors through the bilinear map $V\times V\rightarrow S^2 V$, $(v,w)\rightarrow v\otimes w$.
We embed $S^2(S^2 V)$ and $S^2(\Lambda^2 V)$ in $T^4 V$ in the obvious way, where to avoid confusion we omit tensor signs in $T^4 V$.
Then $(v\otimes w)\otimes (v\otimes w)$ (resp. $(v\wedge w)\otimes (v\wedge w)$) is equal to $\frac{1}{4}\left( vwvw+vwwv+wvvw+wvwv\right)$ (resp. $\frac{1}{4}\left( vwvw+wvwv-vwwv-wvvw\right)$).
Let ${\rm pr}_{S^4 V}$ denote the canonical projection $T^4V\rightarrow S^4 V\subset T^4V$.
Then $$(1-{\rm pr}_{S^4 V})\left((v\otimes w)\otimes 
(v\otimes w)\right)=\frac{1}{12}\left( 
vwvw+vwwv+wvvw+wvwv-2vvww-2wwvv\right)$$
which we denote $\rho(v\otimes w)$.
Permuting the first three indices, we obtain $$(1\; 2\; 
3)\cdot\rho(v\otimes w)=\frac{1}{12}\left( 
vvww+wvwv+vwvw+wwvv-2wvvw-2vwwv\right).$$
Then it is easily seen that $\rho(v\otimes w)+2(1\; 2\; 
3)\cdot\rho(v\otimes w)=(v\wedge w)\otimes (v\wedge w)$, which 
establishes the required result.
\end{proof}

The main step in the second paragraph of the proof above applies to arbitrary elements of $\mathfrak{gl}_{2n}$ of the form $vw^TJ'$.
(Such an element is nilpotent if and only if $w^TJ'v=0$.
We right multiply by $J'$ to ensure that $\theta(vw^TJ')=wv^TJ'$.)
This allows us to enlarge Lemma \ref{splem} somewhat.

\begin{corollary}[of proof]\label{sheetcor}
Let $S$ be the set of elements of $\mathfrak{gl}_{2n}$ of rank at most 1, and let $S_0$ be the set of elements of $\mathfrak{sp}_{2n}$ of rank at most 2.
Then $S_0\cong S/\mathfrak{S}_2$.
\end{corollary}

\begin{proof}
Since $S$ can be identified with the set of elements of $\mathfrak{gl}_{2n}$ of the form $vw^TJ'$ and $S_0$ is the set of elements of $\g$ of the form $(vw^T+wv^T)J'$, where $v,w\in V$, then it is enough to show that the morphism: $$V\otimes V\rightarrow S^2(\Lambda^2 V),\;\; v\otimes w \mapsto (v\wedge w)\otimes (v\wedge w)$$
factors through the canonical morphism to $S^2 V$.
This was proved in the second paragraph of the proof of Lemma \ref{splem}.
\end{proof}

Note that the sets $S$ and $S_0$ are determinantal varieties.
A direct proof of Cor. \ref{sheetcor} can be given by establishing isomorphisms of $S$ with $(V\times V)/{\mathbb G}_m$ (acting with weights $1$ and $-1$) and of $S_0$ with $(V\times V)/\Gamma$, where $\Gamma$ is the group generated by ${\mathbb G}_m$ and the involution $(x,y)\mapsto (y,x)$.

Recall that a {\it sheet} in a Lie algebra ${\g}$ is an irreducible component 
of the (locally closed) subset $\{ x\in\g : \dim G\cdot x = r\}$ for some $r$.
The minimum dimension of the orbit of a non-zero element of $\mathfrak{gl}_{2n}$ is 
$4n-2$, and the set of all elements of $\mathfrak{gl}_{2n}$ with $\dim G\cdot 
x=4n-2$ is irreducible of dimension $4n$, equal to the set of elements of the form $A+\lambda I$ where $\lambda\in k$ and $A\in S$.
If ${\rm char}\, k\nmid 2n$, it follows that $S':=\{ x\in\mathfrak{sl}_{2n} : \dim G\cdot x = 4n-2\}$ is the image under the projection $\g=\mathfrak{sl}_{2n}\oplus kI\rightarrow\mathfrak{sl}_{2n}$ of $S$.
We call $S'$ the {\it minimal sheet} of $\mathfrak{sl}_{2n}$.
Then (with this assumption on the characteristic), it follows that $\overline{S'}/\mathfrak{S}_2\cong S_0$.\footnote{If ${\rm char}\, k\neq 2$ divides $n$ then the same statement is true on replacing $\mathfrak{sl}_{2n}$ by $\mathfrak{pgl}_{2n}$.}
Note that $S_0$ is the closure of a sheet in $\g_0$.

\subsection{Other cases}\label{otherssubsec}

The remaining cases in \cite{Brylinski-Kostant:JAMS} can be dealt with using similar arguments to Lemma 
\ref{splem}.
The idea for the involution cases is to construct the map $\varepsilon:\g_0\rightarrow S^2 U$ such that, for any $x=(x_0,u)$ belonging to a dense open subset of $\0$, we have $\varphi(x)=(x_0,\varepsilon(x_0))$.
It then follows that $\varphi(x)=(x_0,\varepsilon(x_0))$ for all $x\in\overline\0$.
This approach can be easily adapted to the other cases.

In positive characteristic, the situation is somewhat complicated by questions in the representation theory of $G_0$.
To avoid getting bogged down in technical details, we will briefly 
sketch the arguments required for the remaining shared orbit pairs.
We first assume that $k$ has characteristic zero; for a discussion of the adaptations required in positive characteristic, see Remark \ref{poscharremark}(b).
In all cases, $\0_0\subset\g_0$ and $\0\subset\g$ are nilpotent orbits of the same dimension.
We denote by $\varpi_i$ the $i$-th fundamental weight for $\g_0$ and by $V(\varpi)$ the irreducible $\g_0$-representation with highest weight $\varpi$.

\subsubsection{$(\g_0,\g)=(\mathfrak{sp}_{2n_1}\oplus\cdots \oplus\mathfrak{sp}_{2n_m},\mathfrak{sp}_{2n})$ with $n_1+\cdots + n_m=n$.}\label{spsubsubsec}

This is the only case with $\g$ simple and $\g_0$ non-simple, and it is discussed in \cite[Example 4.9]{Brylinski-Kostant:JAMS}.
A direct proof is extremely straightforward: the minimal nilpotent orbit closure $c_n$ of $\mathfrak{sp}_{2n}$ is isomorphic to $k^{2n}/\{ \pm I\}$, which has a finite cover map
$$
c_n =k^{2n}/\{ \pm I\} \to c_{n_1} \times \cdots \times c_{n_m} = k^{2n_1}/\{ \pm I\} \times \cdots \times k^{2n_m}/\{ \pm I\}.
$$ 
If there is another nilpotent orbit in $\g$ which is shared with a nilpotent orbit in $\g_0$, then by the closure relation, the minimal special orbit $\0=[2^2,1^{2n-4}]$ must be shared with some nilpotent orbit $\0_1\times\ldots \times \0_m$ in $\g_0$.
Note that $\dim \0=4n-2$.
However, the projection of $\0$ to each subalgebra $\mathfrak{sp}_{2n_i}$ consists of elements of rank $\leq 2$, and hence each orbit $\0_i$ is contained in the orbit with partition $[2^2,1^{2n_i-4}]$ (or $[2]$ if $n_i=1$).
In particular, $\sum_1^m \dim \0_i\leq \sum_1^m (4n_i-2)<4n-2$, so this is impossible.

\subsubsection{$(\g_0,\g)=(\mathfrak{so}_9,\mathfrak{f}_4)$, 
$(\0_0,\0)=(\0_{[2^4,1]},\0_{\mini})$.}\label{f4subsubsec}
Here $\g=\g_0\oplus V(\varpi_4)$ and $\g_0$ is the fixed point subalgebra for the involution $\theta$ of $\g$ which satisfies $\theta(e_{\pm\alpha_4})=-e_{\pm\alpha_4}$ and fixes all other simple root elements.
Let $V$ be the natural module for $G_0$ and let $U=V(\varpi_4)$.
We find a representative $e+u\in\0$ where (expressed in terms of roots in $\g$) 
$e=e_{0120}-2e_{2342}\in\0_0$ and $u=\pm 2 e_{1231}$, a highest weight vector in $U$.
(This is an easy matrix calculation if one identifies $e_{0120}$, $e_{1111}$, $e_{1231}$, $e_{2342}$ with the positive root elements of a subalgebra isomorphic to $\mathfrak{so}_5$.)
We have $S^2 U=V(2\varpi_4)\oplus V\oplus k$ and $u\otimes 
u\in V(2\varpi_4)$.
The map $S^2\g_0\rightarrow V(2\varpi_4)$ is given by the identification 
$V(2\varpi_4)\cong\Lambda^4 V$ and the natural map $S^2(\Lambda^2 V)\rightarrow\Lambda^4 V$; the projection map $S^2(\Lambda^2 
V)\rightarrow\Lambda^4 V$ sends $e\otimes e$ to $-4u\otimes u$.
(One can deduce this from Lemma \ref{squarezero} applied to the $\mathfrak{so}_5$-algebra.)
One therefore obtains $\overline\0/\langle\theta\rangle\cong\overline{\0_0}$.

The case $(\g_0,\g)=(\mathfrak{so}_8,\mathfrak{f}_4)$, 
$(\0_0,\0)=(\0_{[3,2^2,1]},\0_{\mini})$ now follows by Prop. \ref{mincor}.
In this case we have 
$\overline{\0_0}\cong\overline{\0}/(\mathfrak{S}_2\mathfrak{S}_2)$.

\subsubsection{$(\g_0,\g)=(\mathfrak{f}_4,\mathfrak{e}_6)$, 
$(\0_0,\0)=(\0_{\mathrm{min\text{-}sp}},\0_{\mini})$.}\label{e6subsubsec}
In this case $\g_0$ is the fixed point subalgebra for the graph involution $\theta$ of $\g$ and $\g=\g_0\oplus U$ where $U=V(\varpi_4)$ is the $(-1)$ eigenspace for $\theta$.
There is a representative $e+u\in\0$ where $e=e_{1232}\in\0_0$ and $u$ is a highest weight vector in $U$.
We have $S^2 U=V(2\varpi_4)\oplus U\oplus k$ and $S^2\g_0=V(2\varpi_1)\oplus V(2\varpi_4)\oplus k$; composing the projection from $S^2\g_0$ to $V(2\varpi_4)$ with the inclusion $V(2\varpi_4)\hookrightarrow S^2 U$, $e\otimes e$ is sent to a multiple of $u\otimes u$.
Hence we have $\overline\0/\langle\theta\rangle\cong\overline{\0_0}$.

\subsubsection{$(\g_0,\g)=(\mathfrak{sl}_3,\g_2)$, 
$(\0_0,\0)=(\0_{\mathrm{reg}},\0_{\mini})$.}\label{g2subsubsec}
Fix a primitive cube root of unity $\omega\in k$, and let $V$ denote the 
natural module for $\g_0$.
We have $\g=\g_0\oplus V\oplus V^*$ and there is an automorphism $\theta$ of $\g$ which fixes $\g_0$ and acts as $\omega$ (resp, $\omega^2$) on $V$ (resp. $V^*$).
In this case, generalizing the argument of Lemma \ref{genlem}, we see that the quotient $\g/\langle\theta\rangle$ is obtained as the image of the morphism $$\varphi:\g_0\oplus V\oplus V^*\rightarrow \g_0\times (V\otimes V^*)\times S^3 V\times S^3 V^*,\; (x,v,\chi)\mapsto (x,v\otimes\chi, v\otimes v\otimes v, \chi\otimes \chi\otimes \chi).$$
We have $e+v_1+\chi_3\in\0$ where $e=e_{12}+e_{23}\in\0_0$ and $v_1,\chi_3$ are highest weight vectors in $V,V^*$.
(The numbering comes from dual bases for $V$ and $V^*$.)
We need to show that the restriction of $\varphi$ to $G_0\cdot (e+v_1+\chi_3)$ factors through the projection to $\g_0$.
With the identification $V\otimes V^*=\mathfrak{gl}_3\supset\g_0$, we have $v_1\otimes\chi_3=e^2$.
Furthermore, the composed map $$\g_0\otimes\g_0\subset 
\mathfrak{gl}_3\otimes\mathfrak{gl}_3=(V\otimes V^*)\otimes (V\otimes 
V^*)\rightarrow V\otimes V\otimes\Lambda^2 V^*\rightarrow S^3 V$$ sends 
$e\otimes e^2$ to $v_1\otimes v_1\otimes v_1$, and similarly for the component involving $\chi_3$.

\subsubsection{$(\g_0,\g)=(\g_2,\mathfrak{so}_8)$, 
$(\0_0,\0)=(\0_{\mathrm{subreg}},\0_{\mini})$.}\label{G2case}
Let $\omega$ be as above.
We have $\g=\g_0\oplus V\oplus V$ where $V=V(\varpi_1)$ is the minimal faithful module for $\g_0$.
Taking appropriate summands, we can specify the action of $\mathfrak{S}_3$ on $\g$ by diagram automorphisms as follows: $\g_0=\g^{\mathfrak{S}_3}$, the $3$-cycle $(1\; 2\; 3)$ acts as multiplication by $\omega$ (resp. 
$\omega^2$) on the first (resp. the second) copy of $V$, and $(1\; 3)$ acts by sending $(v,w)\in V\oplus V$ to $(w,v)$.

We want to find generators of the invariant ring $k[\g]^{\mathfrak{S}_3}$.
This can be done directly, but for later use we will recall a useful result of Briand.
Let $U=k^n$ be the permutation representation for $\mathfrak{S}_n$ and let $W=U_1\oplus\ldots \oplus U_r$ with each $U_i$ isomorphic to $U$.
Let $f_1,\ldots ,f_n\in k[U_1]^{\mathfrak{S}_n}$ be the elementary symmetric polynomials.
The polarisations of $f_i$ are the coefficients of each monomial $t_1^{i_1}t_2^{i_2}\ldots t_n^{i_r}$ in $f_i(t_1u_1+\ldots + t_ru_r)$; they are $\mathfrak{S}_n$-invariant polynomials on $W$.
According to a classical result of Weyl, the polarisations generate $k[W]^{{\mathfrak S}_n}$ when $k={\mathbb C}$.
In positive characteristic, we use the following \cite{Briand}: the polarisations generate $k[W]^{{\mathfrak S}_n}$ whenever $n!$ is invertible in $k$.\footnote{In fact, when $(n,r)=(3,2)$ one only requires ${\rm char}\, k\neq 3$.}
We deduce that $\g/\mathfrak{S}_3$ can be described as the image of the morphism $$\varphi:\g_0\oplus V\oplus V\rightarrow \g_0\times S^2 V\times S^3 V,\; x+v+w\mapsto (x,v\otimes w, v\otimes v\otimes v+w\otimes w\otimes w).$$

There is a representative $(e,v_1,v_2)\in\g_0\oplus V\oplus V$ for $\0$ where $e=e_{\alpha_2}+e_{3\alpha_1+\alpha_2}\in\0_0$ and $\{ v_1,\ldots ,v_7\}$ is a standard basis for $V$.
To show that the restriction of $\varphi$ to $\0$ factors through the projection to $\g_0$, we therefore have to express $v_1\otimes v_2$ and $v_1^{\otimes 3}+v_2^{\otimes 3}$ in terms of $e$.
The former follows from considering $e$ as an element of $\mathfrak{so}_7$, so that its matrix square belongs to $\Sigma_7\cong S^2 V$ and (as can be easily seen) identifies with a multiple of $v_1\otimes v_2$.
For the latter we note that $S^3 V=V(3\varpi_1)\oplus V$ and that there is an isomorphism of $G_0$-modules$$S^3\g_0\cong V(3\varpi_2)\oplus V(2\varpi_1+\varpi_2)\oplus V(3\varpi_1)\oplus \g_0\oplus V\oplus k.$$
In particular, there is a $G_0$-equivariant linear map $S^3\g_0\rightarrow S^3 V$,
which by observation sends $e\otimes e\otimes e$ to a non-zero multiple of 
 $v_1^{\otimes 3}+v_2^{\otimes 3}\in V(3\varpi_1)\subset S^3 V$.

The case $(\g_0,\g)=(\g_2,\mathfrak{so}_7)$, $(\0_0,\0)=(\0_{\mathrm{subreg}},\0_{\mathrm{min\text{-}sp}})$ now follows by Cor. \ref{mincor} and the sequence of inclusions $\g_2\subset\mathfrak{so}_7\subset\mathfrak{so}_8$.
Furthermore, the closure of the minimal orbit $\0_{[2^2,1^3]}$ in $\mathfrak{so}_7$ maps onto the closure of the orbit $\0_{\tilde{A}_1}$ of the same dimension in $\g_2$.
It is easy to check (cf. \cite[Cor. 4.4]{Levasseur-Smith}) that the map $\varphi:\overline{\0_{[2^2,1^3]}}\rightarrow \overline{\0_{\tilde{A}_1}}$ is a bijection, and moreover is \'etale at any point of the open subset $\varphi^{-1}(\0_{\tilde{A}_1})$.
It follows that $\varphi$ is birational.
Since $\overline{\0_{[2^2,1^3]}}$ is normal and $\overline{\0_{\tilde{A}_1}}$ is not, $\varphi$ is therefore a (non-isomorphic) normalization map.
(The map $\varphi$ was studied in detail by Levasseur and Smith \cite{Levasseur-Smith}.
In characteristic $>3$, the failure of normality of $\overline{\0_{\tilde{A}_1}}$ follows from a straightforward computation of the slice from an element of $A_1$.)

\begin{remark}\label{poscharremark}
a) The cases in \S \ref{sonsubsec}, \S \ref{spsubsec}, \S \ref{spsubsubsec} (with $m=2$), \S \ref{f4subsubsec} and \S \ref{e6subsubsec} are those arising from an involution $\theta$, i.e. an action of $\mathfrak{S}_2$ on $\g$.
Panyushev's proof of these results comes down to the fact \cite[Prop. 5.10]{Panyushev:shared-orbits} that the projection to ${\mathfrak g}_0$ induces the quotient morphism $\overline{\0_{\mini}}\rightarrow\overline{\0_{\mini}}/\mathfrak{S}_2$ precisely when $\0_{\mini}$ intersects trivially with the $-1$ space for $\theta$.
(These are exactly the involutions of simple Lie algebras satisfying this condition.)

b) In the above discussions we frequently used the representation theory of the Lie algebra $\g_0$ to decompose symmetric and alternating powers of $\g_0$ and $U$.
In characteristic $p$, we can repeat these arguments verbatim under the condition that all corresponding Weyl modules $\Delta(\varpi)$ are irreducible.
(This is a sufficient but probably not necessary criterion for the existence of the required $G_0$-equivariant map.)
One just has to be careful that the desired projection remains non-zero when reducing modulo $p$.

(i) $(\mathfrak{so}_9,{\mathfrak f}_4)$: $U=\Delta(\varpi_4)$, $\Delta(\varpi_1)$ and $\Delta(2\varpi_4)$ are irreducible whenever $p>2$ \cite[Remark 3.4]{McNinch}.
It is easy to see that this also ensures that the natural map $S^2(\Lambda^2 V)\rightarrow \Lambda^4 V$ sends $e\otimes e$ to a (non-zero) highest weight vector.
(As elements of $\Lambda^2 V$, $e_{2342}$ corresponds to $v_1\wedge v_2$ and $e_{0120}$ corresponds to $v_3\wedge v_4$.)

(ii) $({\mathfrak f}_4,{\mathfrak e}_6)$: $U=\Delta(\varpi_4)$, $\Delta(2\varpi_4)$ and $\Delta(2\varpi_1)$ are irreducible under the assumption $p\neq 3,7,13$ \cite{Luebeck}.
The condition that $e\otimes e$ project to a {\it non-zero} multiple of $u\otimes u$ is then easily seen to be satisfied, by considering the scalar product with the $-2\varpi_4$-weight space in $S^2\g_0$.
The lowest weight vector in $\Delta(2\varpi_4)$ is (up to signs and structure constants $\in \{ 2,3\}$) of the form 
$$v_{-2\varpi_4}=4f_{0122}\otimes f_{2342}-4f_{1122}\otimes f_{1342}+4f_{1222}\otimes f_{1242}-f_{1232}\otimes f_{1232}$$
and hence $(e\otimes e| v_{-2\varpi_4})\neq 0$.
This argument can also be adapted to characteristic 13, as follows: in this case, $\Delta(2\varpi_1)$ is irreducible, but $\Delta(2\varpi_4)$ has a simple quotient $L(2\varpi_4)$ by a one-dimensional trivial submodule \cite{Luebeck}.
It follows that $S^2 U$ is the direct sum of $U$ and the indecomposable tilting module $T(2\varpi_4)$ (with socle and head isomorphic to $k$).
A short calculation establishes that the Casimir element $\Omega$ acts on $\Delta(2\varpi_1)$, resp. $\Delta(2\varpi_4)$, with weight $20$, resp. $13$.
Since $S^2\g_0$ is tilting, it follows that we have a direct sum decomposition $S^2\g_0\cong\Delta(2\varpi_1)\oplus T(2\varpi_4)$.
We can therefore project from $S^2 \g_0$ to $T(2\varpi_4)\hookrightarrow S^2 U$, rather than to $\Delta(2\varpi_4)$.
(This approach rather breaks down in characteristic 7, where neither $\Delta(2\varpi_4)$ nor $\Delta(2\varpi_1)$ is simple.)

(iii) $(\mathfrak{sl}_3,\g_2)$: it is easy to see that this goes through word for word if ${\rm char}\, k>3$.
In fact one can even adapt this example to characteristic $3$, in which case we identify the ${\mathbb Z}/(3)$-grading of $\g$ with an action of an infinitesimal group scheme $\mu_3\subset{\mathbb G}_m$.
In this context, the finite map $\overline\0\rightarrow\overline{\0_0}$ restricts to a purely inseparable bijection over $\0_0$.
(We remark that this statement in characteristic 3 has a Langlands dual counterpart, related to the existence of an exceptional isogeny, see Rk. \ref{classificationremark}.
Note that in characteristic 3, $\g$ has two nilpotent orbits of minimal dimension 6.)

(iv) $(\g_2,\mathfrak{so}_8)$: here the map $S^2\g_0\rightarrow S^2 V$ is explicit and maps $e\otimes e$ to a non-zero multiple of $v_1\otimes v_2$ in characteristic $>3$ (even for $p=7$ when $S^2 V$ is not completely reducible).
For the map $S^3\g_0\rightarrow S^3 V$, we note that $\Delta(3\varpi_1)$ is irreducible in arbitrary characteristic \cite{Luebeck} and that the Casimir element $\Omega$ acts on $\Delta(3\varpi_2)$, $\Delta(2\varpi_1+\varpi_2)$, $\Delta(3\varpi_1)$, $\g_0=\Delta(\varpi_2)$, $V=\Delta(\varpi_1)$, $k$ with respective weights $27$, $16$, $12$, $6$, $3$ and $0$.
The assumption $p>5$ ensures that $12$ is not congruent to $27$, $16$, $6$, $3$ or $0$ modulo $p$.
In particular, the map $\pi:S^3\g_0\rightarrow \Delta(3\varpi_1)$ is well-defined for $p>5$ and is given by the projection onto the $12$-eigenspace for $\Omega$.
To see that $e\otimes e\otimes e$ is mapped onto a non-zero multiple of $u\otimes u\otimes u$, write $e$ as $e_{01}+e_{31}$ and note that the simple reflection $s_{\alpha_1}$ has a representative $s$ in $G_0^e$ such that ${\Ad\, s}(e_{01})=e_{31}$, ${\Ad\, s}(e_{31})=e_{01}$, $sv_1=v_2$ and $sv_2=v_1$.
We note first of all that $\pi(e\otimes e\otimes e)=3\pi(e_{01}\otimes e_{31}\otimes e_{31})+3\pi(e_{01}\otimes e_{01}\otimes e_{31})$ and that $s\cdot v_1\otimes v_1\otimes v_1=v_2\otimes v_2\otimes v_2$.
Thus it will suffice to show that $\pi(e_{01}\otimes e_{31}\otimes e_{31})$ is a non-zero multiple of $v_1\otimes v_1\otimes v_1$.
Using an argument similar to (ii), we note that this has non-zero scalar product with the following lowest weight vector in $V_{-3\varpi_1}$ (up to scaling individual terms by $\{ \pm 1,\pm 2,\pm 3\}$):
$$2f_{21}\otimes f_{21}\otimes f_{21}-9f_{11}\otimes f_{21}\otimes f_{31}+27\alpha_2^\vee\otimes f_{31}\otimes f_{32}-27e_{01}\otimes f_{32}\otimes f_{32}+27 f_{01}\otimes f_{31}\otimes f_{31}-9f_{10}\otimes f_{21}\otimes f_{32}.$$
\end{remark}

\subsection{Classification of shared nilpotent orbits}

Brylinski-Kostant's work can be seen as a classification of those pairs 
$\g_0\subset\g$ of reductive Lie algebras with $\g$ simple for which there is a non-zero nilpotent $G$-orbit $\0\subset\g$ containing an open $G_0$-orbit $\widetilde\0_0$; then there is a finite covering map from $\widetilde\0_0$ to a nilpotent orbit $\0_0$ in $\g_0$.
The nilpotent orbits $\0$ and $\0_0$ are called shared nilpotent orbits.

In most cases, $\g_0$ is the isotropy subalgebra for a finite group $\Gamma$ of automorphisms, and $\overline\0/\Gamma$ is isomorphic to $\overline{\0_0}$. 

\begin{proposition}\label{p.sharedorbits}
The following table gives all connected shared nilpotent orbits in complex Lie algebras $\g_0\subset\g$, with $\g$ simple.
These are also shared orbits in positive characteristic (but this is not a complete list).
In the final column we list restrictions on positive characteristic $p$.
Denote by $\sigma$ a triality automorphism of $\mathfrak{so}_8$.
\end{proposition}

\begin{center}
\begin{tabular}{|c | c|  c|  c|  c| c| c|c|c|}
\hline
 & $\fg$  &  $\0$  &  $\pi_1(\0)$  &  ${\rm deg}$  &  $\fg_0$  & $\0_0$  & $\pi_1(\0_0)$ & \text{Char. $p$} \\
\hline
1 & $\mathfrak{sp}_{2(n_1+\ldots + n_{m})}$ & $A_1$ & $\mathfrak{S}_2$ & $2^{m-1}$ & $\oplus_{i=1}^{m} \mathfrak{sp}_{2n_i}$ & $\prod_{i=1}^m\0_{\mini}$ & $\mathfrak{S}_2\ldots \mathfrak{S}_2$ & $p\neq 2$ \\
\hline
2 & $\fg_2$ & $A_1$  &  $1$ &  $3$  & $\mathfrak{sl}_3$  & $[3]$  & $\mathbb{Z}/(3)$ & $p\neq 2$ \\
\hline
3 & $\mathfrak{f}_4$  & $A_1$ & $1$  & $2$ &  $\mathfrak{so}_9$  & $[2^4,1]$ & $\mathfrak{S}_2$ & $p>3$ \\
\hline
4 & $\mathfrak{f}_4$  & $A_1$ & $1$  & $4$  & $\mathfrak{so}_8$  & $[3,2^2,1]$ & $\mathfrak{S}_2 \mathfrak{S}_2$ &  $p>3$ \\
\hline
5 & $\mathfrak{e}_6$  &  $A_1$  & $1$ & $2$ & $\mathfrak{f}_4$ & $\tilde{A}_1$  &  $\mathfrak{S}_2$ & $p\neq 2,3, 7$ \\
\hline
6 & $\mathfrak{so}_7$  &  $[2^2,1^3]$  & $1$ & $1$ & $\mathfrak{g}_2$ & $\tilde{A}_1$  &  $1$ & $p>5$ \\
\hline
7 & $\mathfrak{so}_7$  &  $[3,1^4]$  & $\mathfrak{S}_2$ & $3$ & $\mathfrak{g}_2$ & $G_2(a_1)$  & $\mathfrak{S}_3$ & $p>5$ \\
\hline
8 & $\mathfrak{so}_8$  &  $[2^2,1^4]$  & $1$ & $6$ & $\mathfrak{g}_2$ & $G_2(a_1)$  &  $\mathfrak{S}_3$ & $p>5$ \\
\hline
9a  & $\mathfrak{so}_{8}$  &  $[2^{4}]^I, [2^{4}]^{II}$  & $\mathfrak{S}_2$ & $1$ & $\mathfrak{so}_{7}$ & $[3,2^{2}]$  &  $\mathfrak{S}_2$ & $p>2$ \\
\hline
9b  & $\mathfrak{so}_{8}$  &  $[3,1^5]$  & $\mathfrak{S}_2$ & $1$ & $\sigma(\mathfrak{so}_{7})$ & $[3,2^{2}]$  &  $\mathfrak{S}_2$ & $p>2$ \\
\hline
10  & $\mathfrak{so}_{9}$  &  $[2^2,1^5]$  & $1$ & $2$ & $\sigma(\mathfrak{so}_{7})$ & $[3,2^{2}]$  &  $\mathfrak{S}_2$ & $p>2$ \\
\hline
11  & $\mathfrak{so}_{4r}$, $r>2$  &  $[2^{2r}]^I, [2^{2r}]^{II}$  & $\mathfrak{S}_2$ & $1$ & $\mathfrak{so}_{4r-1}$ & $[3,2^{2r-2}]$  &  $\mathfrak{S}_2$ & $p>2$ \\
\hline
12  & $\mathfrak{so}_{4r+1}$  &  $[2^{2r},1]$  & $\mathfrak{S}_2$ & $2$ & $\mathfrak{so}_{4r}$ & $[3,2^{2r-2},1]$  &  $\mathfrak{S}_2\mathfrak{S}_2$ & $p>2$ \\
\hline
13  & $\mathfrak{so}_{n}$, $n>4r+1$  &  $[2^{2r},1^{n-4r}]$  & $1$ & $2$ & $\mathfrak{so}_{n-1}$ & $[3,2^{2r-2},1^{n-4r}]$  &  $\mathfrak{S}_2$ &  $p>2$ \\
\hline
14  & $\mathfrak{sl}_{2n}$  &  $[2,1^{2n-2}]$  & $1$ & $2$ & $\mathfrak{sp}_{2n}$ & $[2^2,1^{2n-4}]$  &  $\mathfrak{S}_2$ & $p>2$ \\
\hline
\end{tabular}
\end{center}

\begin{proof}
By \cite[Theorem 6.10]{Brylinski-Kostant:JAMS} and our Example \ref{sonexamples}(a), shared Lie algebras $\fg_0 \subset \fg$ are given by the above table. 
(As explained in Example \ref{sonexamples}(a), the extra example arises from a minor gap in the reasoning in \cite{Brylinski-Kostant:JAMS}.)
It remains to find all shared orbits for each of these Lie algebras.
For (1), we already proved in \S \ref{spsubsubsec} that all shared orbit pairs are of the form given in the table.

For (2)-(5), there are no more shared orbits, by inspecting dimensions.
(Note that if $\0\subset \g$ appears in a shared orbit pair then so does any smaller orbit $\0' \subset \overline{\0}$.)
For the pair $\g_2\subset\mathfrak{so}_8$, there are three $\SO_8$-orbits lying immediately above the minimal orbit in the partial order (see (9)), and these are conjugate via the outer automorphism group $\mathfrak{S}_3$.
Hence it only suffices to check that $[3,1^5]$ is not shared with the regular orbit in $\g_2$.
If it were then the minimal orbit in $\mathfrak{so}_9$ would be shared with the regular orbit in $\g_2$ by (12), so this is ruled out.
Now the unique orbit in $\mathfrak{so}_7$ lying immediately above $[3,1^4]$ in the partial order is $[3,2^2]$, which is shared with $[2^4]$ in $\mathfrak{so}_8$, and so cannot be shared with the regular orbit in $\g_2$.

For the shared Lie algebras $\mathfrak{so}_7\subset\mathfrak{so}_8$, we note that $[3,2^2,1]$ is the unique orbit lying immediately above each of $[2^4]^I,[2^4]^{II}$, $[3,1^5]$ and it is of dimension 16.
Clearly, any element of $[3,1^5]$ projects onto an element of $\mathfrak{so}_7$ of rank $\leq 2$, hence (by dimensions) it is not shared with an orbit in (the standard) $\mathfrak{so}_7$.
The same argument shows that, for $n\geq 10$, $\0_{[3,1^{n-3}]}\subset\mathfrak{so}_n$ is never shared with an orbit in $\mathfrak{so}_{n-1}$, hence the rows (10)-(13) exhaust all remaining possibilities when $\g=\mathfrak{so}_n$.

Finally, for (14), we only need to show that $\0_{[2^2,1^{2n-4}]}\subset\mathfrak{sl}_{2n}$ does not appear in a shared orbit pair.  One computes that 
this orbit has dimension $8n-8$ while there is no nilpotent orbit in $\mathfrak{sp}_{2n}$ with this dimension.
\end{proof}

\begin{remark}\label{classificationremark}
The list in Prop. \ref{p.sharedorbits} is not exhaustive in positive characteristic.
For example, in characteristic 2 the simple Lie algebra $\g$ of type $G_2$ is isomorphic to $\mathfrak{psl}_4\cong [\mathfrak{pgl}_4,\mathfrak{pgl}_4]$.
It follows that there is a purely inseparable bijection ${\mathcal N}(\mathfrak{sl}_4)\rightarrow{\mathcal N}(\g)$.

Similarly, a simple Lie algebra of type $F_4$ in characteristic 2 has an ideal isomorphic to $\mathfrak{pso}_8$, hence there is a purely inseparable map ${\mathcal N}(\mathfrak{spin}_8)\rightarrow{\mathcal N}(\mathfrak{f}_4)$.
This example is related to the existence of an exceptional isogeny, hence has analogues in types $B$ and $C$ in characteristic 2 and in type $G_2$ in characteristic 3 (where $\g$ has an ideal isomorphic to $\mathfrak{psl}_3$).
We do not know whether there exist shared orbits in positive characteristic not appearing in Prop. \ref{p.sharedorbits} and not related to exceptional isogenies or to type $G_2$ in characteristic 2.
Note also that in small positive characteristic, the unipotent variety is not isomorphic to the nilpotent cone, hence there may exist distinct examples of shared unipotent orbits.
\end{remark}

\subsection{A doubled version of $(\mathfrak{g}_2,\mathfrak{so}_8)$}\label{doubled}

There are seven special orbits (all in exceptional types) for which the group appearing in Lusztig's special pieces conjecture is $\mathfrak{S}_3$.
As we will see later, all but two of these sit atop a Slodowy slice singularity isomorphic to $g_2^{\special}$.
The exceptions are the orbits $D_4(a_1)+A_1$ and $D_4(a_1)$ in type $E_8$; in the latter case, the singularity down to $2A_2$ is a union of two copies of $g_2^{\special}$.
In this section we will explain both anomalies, by considering the whole degeneration from $\0=\0_{D_4(a_1)+A_1}$ to $2A_2$.
Let ${\mathcal S}$ be the Slodowy slice to an element of $\0_{2A_2}$.
We will use arguments similar to \S \ref{G2case} to show that ${\mathcal S}\cap\overline\0$ is isomorphic to $(d_4\times d_4)/\mathfrak{S}_3$.
We obtain the isomorphism of ${\mathcal S}\cap\overline{\0_{D_4(a_1)}}$ with $2g_2^{\special}$ as a by-product.

First of all, we clarify some notation.
Let $\mathfrak{S}_3$ act diagonally by triality automorphisms of $\mathfrak{so}_8\oplus\mathfrak{so}_8$.
The fixed point subalgebra is a direct sum $\g_0\oplus\g_0$ of two Lie algebras of type $G_2$; there is an $\mathfrak{S}_3$-equivariant action of $G_0\times G_0$, where ${\rm Lie}(G_0)=\g_0$.
Similarly to the description in \S \ref{G2case}, there is a decomposition $$\mathfrak{so}_8\oplus \mathfrak{so}_8=(\g_0\oplus \g_0) \oplus (V\oplus V) \oplus (V\oplus V),$$ where $V$ is the minimal faithful representation and $(1\; 2\; 3)\in\mathfrak{S}_3$ acts as $\omega$ (resp. $\omega^2$) on the first two (resp. the last two) copies of $V$.
To help to keep track of terms, we denote by $\g'_0$ (resp. $V'$) the second copy of $\g_0$ (resp. $V$) in each pair, and similarly for the elements.
Then we can repeat the argument from \S \ref{G2case} to see that $(\mathfrak{so}_8\oplus\mathfrak{so}_8)/\mathfrak{S}_3$ is isomorphic to the image of the morphism $$\varphi:\g_0\oplus\g'_0\oplus V\oplus V'\oplus V\oplus V'\rightarrow \g_0\oplus \g'_0\oplus S^2(V\oplus V')\oplus S^3(V\oplus V'),$$
$$(x,x',v, v', w,w')\mapsto (x,x',(v+v')\otimes (w+w'), (v+v')^{\otimes 3}+(w+w')^{\otimes 3}).$$
We want to restrict the map $\varphi$ to the subset $Y=\overline{\0_{\mini}(\mathfrak{so}_8)}\times \overline{\0_{\mini}(\mathfrak{so}_8)}$.
We are going to show that $\varphi|_Y$ factors through the map $$\g_0\oplus \g'_0\oplus V\oplus V'\oplus V\oplus V'\rightarrow \g_0\oplus \g'_0\oplus V\otimes V',$$
where $(v,v',w,w')\mapsto v\otimes w'+w\otimes v'$ in the final component.

Recall the representative $(e,v_1,v_2)$ of $\0_{\mini}(\mathfrak{so}_8)$ which was described in \S \ref{G2case}.
(In particular, $v_1$ is a highest weight vector.)
By dimensions, the closure of the $G_0\times G_0$-orbit of $(e, e', v_1, v'_1, v_2, v'_2)$ equals $Y$.
Note that $$(v_1+v'_1)\otimes (v_2+v'_2)=v_1\otimes v_2+(v_1\otimes v'_2+v_2\otimes v'_1)+v'_1\otimes v'_2\in S^2 V+V\otimes V'+S^2 V' .$$
As in \S \ref{G2case}, $v_1\otimes v_2$ (resp. $v'_1\otimes v'_2$) identifies with the matrix square of $e$ (resp. $e'$).
Thus the restriction of $\varphi$ to $Y$ factors through the projection from $S^2(V\oplus V')$ to $V\otimes V'$.
Hence we have only to show that the component $S^3(V\oplus V')$ is redundant when restricting $\varphi$ to $Y$.

We consider the decomposition $S^3(V\oplus V')=S^3 V\oplus (S^2 V\otimes V')\oplus (V\otimes S^2 V')\oplus S^3 V'$.
The argument from \S \ref{G2case} establishes that $\varphi$ factors through the projection from $S^3(V\oplus V')$ to $(S^2 V\otimes V')\oplus (S^2 V'\otimes V)$.
Hence it will suffice to show that there is a $G_0\otimes G'_0$-equivariant map from $\g_0\oplus \g'_0\oplus (V\otimes V')$ to $S^2 V\otimes V'$ which sends $(e,e',v_1\otimes v'_2+v_2\otimes v'_1)$ to $v_1\otimes v_1\otimes v'_1+v_2\otimes v_2\otimes v'_2$.
(Symmetry will account for the remaining term.)
To see this, we first show that there is a $G_0$-equivariant map $\g_0\otimes V\rightarrow S^2 V$.
By a straightforward calculation, $\g_0\otimes V$ is isomorphic as a $G_0$-module to $V(\varpi_1+\varpi_2)\oplus V(2\varpi_1)\oplus V$, and after scaling, the $G_0$-equivariant projection $\g_0\otimes V\rightarrow V(2\varpi_1)\subset S^2 V$ sends $e\otimes v_1$ to $v_2\otimes v_2$.
By applying the action of the reductive centralizer of $e$ in $G_0$ (which is isomorphic to $\mathfrak{S}_3$) we deduce that $e\otimes v_2$ is sent to $v_1\otimes v_1$.
Hence the induced map from $\g_0\otimes V\otimes V'$ to $S^2 V\otimes V'$ sends:
$$e_0\otimes (v_1\otimes v'_2+v_2\otimes v'_1)=(e\otimes v_1)\otimes v'_2+(e\otimes v_2)\otimes v'_1\mapsto v_1\otimes v_1\otimes v'_1+v_2\otimes v_2\otimes v'_2.$$
We have therefore proved the characteristic zero statement of:

\begin{lemma}\label{d4d4lem}
Assume the ground field has characteristic zero or greater than 7.
Denote by $d_4$ the minimal nilpotent orbit closure in $\mathfrak{so}_8$.
Let $e$, $v_1$, $G_0$ be as above.
Then $(d_4\times d_4)/\mathfrak{S}_3$ is isomorphic to the closure of the $G_0\times G_0$-orbit of $(e,e,v_1\otimes v_2+v_2\otimes v_1)$ in $\g_0\oplus \g_0\oplus V\otimes V$.
\end{lemma}

\begin{proof}
We established this in characteristic zero above.
Similarly to Remark \ref{poscharremark}(ii), to apply these arguments in positive characteristic we need to ensure that the relevant Weyl modules are irreducible.
It can be checked \cite{Luebeck} that this holds if the characteristic is greater than $7$.
It only remains to check that the map $\g_0\otimes V\rightarrow S^2 V$ projecting onto $\Delta(2\varpi_1)$ sends $e\otimes v_2$ to a non-zero multiple of $v_1\otimes v_1$, which is a straightforward computation.
\end{proof}

This is the preparation we require to prove our claim about the Slodowy slice singularity from $\0'=\0_{2A_2}$ to $\0=\0_{D_4(a_1)+A_1}$ in type $E_8$.
Note that the Slodowy slice ${\mathcal S}$ at $f\in\0'$ is transverse if the characteristic of $k$ is zero or greater than 5, in particular, if ${\rm char}\, k$ is (very) good.

\begin{proposition}\label{d4d4Prop}
Let ${\rm char}\, k=0$ or $>7$.
Let $\0,\0'$ be as above and let ${\mathcal S}$ be the Slodowy slice to an element of $\0'$.
Then the intersection ${\mathcal S}\cap\overline{\0}$ is isomorphic to $(d_4\times d_4)/\mathfrak{S}_3$.
\end{proposition}

\begin{proof}
Let $\g$ be a simple Lie algebra of type $E_8$ and let $\{ h_0, e_0, f_0\}\subset\g$ be an $\mathfrak{sl}_2$-triple with $e_0, f_0\in \0_{2A_2}$.
Then ${\mathcal S}=f_0+\g^{e_0}$.
The reductive part $\g^{e_0}\cap \g^{h_0}$ of the centralizer is isomorphic to $\g_0\oplus\g'_0$ (distinguishing the copies with a prime, as above); the module structure of the positive part is as follows:
$$\g^{e_0}(2) = V\otimes V'\oplus \C e_0; \quad \g^{e_0}(4) = V\oplus V'.$$
We verified in GAP that there exists a representative of $\0$ of the form $f_0 + e+e'+x_2$ where $e$ and $e'$ are standard representatives (as above) of the subregular orbits in $\g_0$ and $\g'_0$, and $x_2$ is supported only in $V\otimes V'\subset \g^{e_0}(2)$; furthermore, it is straightforward to check that $x_2$ identifies with $v_1\otimes v'_2+v_2\otimes v'_1$ after appropriate scaling. 
\end{proof}

It follows from Prop. \ref{d4d4Prop} that the poset of orbits in $E_8$ lying between $\0'$ and $\0$ is in bijection with the set of symplectic leaves (with the closure ordering) in $(d_4\times d_4)/\mathfrak{S}_3$.
This can be seen as follows.
For $\Gamma=1,\mathfrak{S}_2,\mathfrak{S}_3$ let $Y_\Gamma$ denote the image in $(d_4\times d_4)/\mathfrak{S}_3$ of the set of elements of $d_4\times d_4$ with $\mathfrak{S}_3$-stabilizer isomorphic to $\Gamma$, and let $Y^0_{\Gamma}$ denote the subset of elements of $Y_\Gamma$ such that one of the components is zero.
Then we have the following isomorphic diagrams:

$$
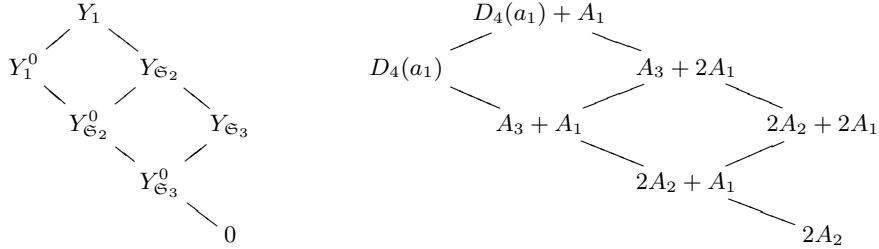
\small{\xymatrix@=.2cm{
{} & Y_{1}  \ar@{-}[dl] \ar@{-}[dr] & {} & {} &&&& {} & D_4(a_1)+A_1 \ar@{-}[dl] \ar@{-}[dr] & {} & {} \\
Y^0_{1}  \ar@{-}[dr] & {} & Y_{{\mathfrak S}_2}  \ar@{-}[dl] \ar@{-}[dr] & {} &&&& D_4(a_1) \ar@{-}[dr] & {} & A_3+2A_1  \ar@{-}[dl] \ar@{-}[dr] & {} \\
{} & Y^0_{\mathfrak{S}_2} \ar@{-}[dr] & {} & Y_{\mathfrak{S}_3}  \ar@{-}[dl]  &&&& {} & A_3+A_1  \ar@{-}[dr] & {} & 2A_2+2A_1 \ar@{-}[dl] \\
{} & {} & Y^0_{\mathfrak{S}_3}  \ar@{-}[dr] & {} &&&& {} & {} & 2A_2+A_1  \ar@{-}[dr] & {} \\
{} & {} & {} & 0 &&&& {} & {} & {} & 2A_2
}}
$$

We therefore obtain, almost immediately:

\begin{corollary}\label{2g2spcor}
Suppose ${\rm char}\, k=0$, and let $\0''$ be the orbit in $E_8$ with label $D_4(a_1)$.
The intersection ${\mathcal S}\cap \overline{\0''}$ is a union of two copies of $g_2^{\special}$, meeting transversely at their unique point of intersection.
\end{corollary}

\begin{proof}
This is the set $Y_{1}^0$ (which is isomorphic to $d_4/\mathfrak{S}_3\cong g_2^{sp}$ by \S \ref{G2case}).
\end{proof}

\section{Almost shared orbits}\label{almostsharedorbits} 

In the previous section we were concerned with pairs $\g_0\subset\g$ of simple Lie algebras for which there is a nilpotent $G$-orbit $\0\subset\g$ containing an open $G_0$-orbit $\widetilde\0_0$.
In most cases, $\g_0$ is the isotropy subalgebra for a finite group $\Gamma$ of automorphisms, and $\overline\0/\Gamma$ is isomorphic to $\overline{\0_0}$.
In this section we consider a weaker condition: suppose there exists a $G_0$-orbit $\widetilde\0_0$ of codimension 1 in $\0$.
We shall call $\0$ an {\it almost shared orbit}.
A straightforward modification of the arguments of the previous section allows us to concretely describe the quotient $\overline\0/\Gamma$ as the closure of a $(G_0\times{\mathbb G}_m)$-orbit.
From now on we assume for simplicity that the base field is ${\mathbb C}$ (although see Remark \ref{modularremark}).

\subsection{Outer involution of $\mathfrak{sl}_n$}\label{outersln}

In type $A_{n-1}$ with $n\geq 3$ odd, all outer involutions are conjugate, and act freely on the minimal nilpotent orbit.
The standard choice of outer involution is $x\mapsto -x^T$, with fixed points $\mathfrak{so}_n$.
In fact, our analysis will be easier if we take a slightly different involution (in the same conjugacy class); equivalently, we want to consider a different symmetric bilinear form.
Let $J$ denote the $n\times n$ matrix with $1$ on the anti-diagonal and $0$ elsewhere.
To avoid confusion, denote by ${\rm O}_{n,J}$ (resp. $\SO_{n,J}$, $\mathfrak{so}_{n,J}$) the group (resp. connected group, Lie algebra) of $n\times n$ matrices preserving the bilinear form determined by $J$.
Denote by $\Sigma^\circ_{n,J}$ the space of traceless $n\times n$ matrices $A$ such that $JA^TJ=A$.

In what follows we do not need to assume that $n$ is odd.
With this notation, let $\gamma$ be the involutive automorphism of $\SL_n$ given by $g\mapsto 
J(g^T)^{-1}J$.
By abuse of notation we also write $\gamma$ for the differential 
$\mathfrak{sl}_n={\mathfrak g}\rightarrow{\mathfrak g}$, $x\mapsto -Jx^TJ$.
We first observe that ${\mathfrak g}=\mathfrak{so}_{n,J}\oplus\Sigma^\circ_{n,J}$ is 
the eigenspace decomposition for $\gamma$, so Lemma \ref{genlem} yields:

\begin{lemma}
The map $\varphi:{\mathfrak g}\rightarrow \mathfrak{so}_{n,J}\times 
S^2(\Sigma^\circ_{n,J})$, $x\mapsto (\frac{1}{2}(x+\gamma(x)),\frac{1}{4} 
(x-\gamma(x))\otimes (x-\gamma(x)))$ induces an isomorphism of ${\mathfrak 
g}/\langle\gamma\rangle$ with $\mathfrak{so}_{n,J}\times X'$, where $X'$ is the 
space of pure symmetric 2-tensors in $S^2(\Sigma^\circ_{n,J})$.

Furthermore, if $Y$ is any closed $\gamma$-stable subset of ${\mathfrak g}$ 
then $Y/\langle\gamma\rangle\cong\varphi(Y)$.
\end{lemma}

We want to describe $\overline\0/\langle\gamma\rangle$ where $\0={\mathcal 
O}_{\mini}(\mathfrak{g})$.
Let $E=(e_{11}+e_{1n}-e_{n1}-e_{nn})\in\0$.
Then clearly $\varphi(E)=(e_{11}-e_{nn},(e_{1n}-e_{n1})\otimes (e_{1n}-e_{n1}))$.
The squared scaling action of ${\mathbb G}_m$ on $\mathfrak{sl}_n$ induces the weighted action on $\mathfrak{so}_{n,J}\times S^2(\Sigma^\circ_{n,J}))$, where $t\cdot (x,y)=(t^2 x,t^4 y)$.

\begin{lemma}\label{slnouterlem}
Let ${\mathcal O}$ be the minimal nilpotent orbit in $\mathfrak{sl}_n$.
Then the quotient $\overline{\mathcal O}/\langle\gamma\rangle$ is isomorphic to 
the closure of the $\SO_{n,J}\times{\mathbb G}_m$-orbit of 
$(e_{11}-e_{nn},(e_{1n}-e_{n1})\otimes (e_{1n}-e_{n1}))$ in 
$\mathfrak{so}_{n,J}\times S^2(\Sigma^\circ_{n,J})$.
\end{lemma}

\begin{proof}
Let $M=e_{11}-e_{nn}$, $N=(e_{1n}-e_{n1})\otimes (e_{1n}-e_{n1})$.
Since $M$ is semisimple, the connected stabilizer of $(M,N)$ in $\SO_{n,J}\times{\mathbb G}_m$ is contained in $Z_{\SO_{n,J}}(M)\cong \SO_{n-2,J}\times\SO_{2,J}$.
Clearly, the simple component $\SO_{n-2,J}$ acts trivially on $N$, and $s\in\SO_{2,J}\cong {\mathbb G}_m$ acts via $$s\cdot N=(s^2 e_{1n}-s^{-2}e_{n1})\otimes (s^2e_{1n}-s^{-2} e_{n1})$$ which is equal to $N$ if and only if $s^4=1$.
It follows that the connected component of the stabilizer of $(M,N)$ is isomorphic to $\SO_{n-2,J}$ and hence the dimension of the orbit is $\frac{1}{2}n(n-1)+1-\frac{1}{2}(n-2)(n-3)=2n-2$.
Since this is equal to the dimension of ${\mathcal O}$ (and since 
$(M,N)\in\varphi({\mathcal O})$) it follows that $\overline{\mathcal 
O}/\langle\gamma\rangle$ is isomorphic to the closure of the orbit of $(M,N)$.
\end{proof}

We are interested in the case $n=3$, where we have $\g_0\cong\mathfrak{sl}_2$.
Denote by $V$ the natural module for $\SL_2$.
Then $\Sigma^\circ_{3,J}\cong S^4 V$, hence $S^2(\Sigma^\circ_{3,J})\cong S^8 V\oplus S^4 V\oplus S^0 V$.
(Of course $S^0 V\cong {\mathbb C}$, but it will be helpful to retain the notation $S^0 V$ to refer to this specific submodule of $S^2(S^4 V)$.)
We want to show that the quotient in this case factors through the projection to $\mathfrak{sl}_2\times S^8 V$.

\begin{lemma}\label{constslem}
Let $\{ v_{2n},v_{2n-2},\ldots ,v_{-2n}\}$ be a standard basis of weight vectors for $S^{2n} V$, i.e. satisfying $e\cdot v_{2(i-n)}=(i+1)v_{2(i+1-n)}$ and $f\cdot v_{2(n-i)}=(i+1)v_{2(n-i-1)}$.
There is a non-degenerate symmetric $\SL_2$-invariant bilinear form $\langle .\, , .\rangle$ on $S^{2n} V$ which is unique up to scalar multiplication and satisfies $\langle v_0,v_0\rangle = (-1)^n \begin{pmatrix} 2n \\ n\end{pmatrix}\langle v_{2n},v_{-2n}\rangle$.
\end{lemma}

\begin{proof}
The existence of a non-degenerate $\SL_2$-invariant form of sign $(-1)^m$ on $S^m V$ can be deduced from the existence of a symplectic form on $V$.
The uniqueness is essentially Schur's lemma.
The relation follows from the fact that $\langle e\cdot v_{2(i-1)},v_{-2i}\rangle = -\langle v_{2(i-1)},e\cdot v_{-2i}\rangle$ for $1\leq i\leq n$.
\end{proof}

We want to determine the projection of $(e_{13}-e_{31})\otimes (e_{13}-e_{31})$ to the factor $S^8 V$.
Fix the $\mathfrak{sl}_2$-triple $e=e_{12}-e_{23}$, $h=2(e_{11}-e_{33})$, $f=2(e_{21}-e_{32})$ in $\mathfrak{so}_{3,J}$ (hence the isomorphism with $\mathfrak{sl}_2$).
Clearly $e_{13}\otimes e_{13}$ (resp. $e_{31}\otimes e_{31}$) is a highest (resp. lowest) weight vector in $S^8 V\subset S^2(S^4 V)$.

\begin{lemma}\label{basislem}
If $v_8=e_{13}\otimes e_{13}$ is fixed as a highest weight vector in a standard basis for $S^8 V\subset S^2 (S^4 V)$, then we have $$v_0=8e_{31}\otimes e_{13}-16(e_{23}+e_{12})\otimes (e_{32}+e_{21})+4(e_{11}-2e_{22}+e_{33})\otimes (e_{11}-2e_{22}+e_{33})$$ and $v_{-8}=16 e_{31}\otimes e_{31}$.
\end{lemma}

\begin{proof}
This is a routine (but somewhat tedious) calculation, which can be seen by applying $\frac{1}{4!}(\ad f)^4$ and $\frac{1}{8!}(\ad f)^8$ respectively.
We observe that the bilinear form on $S^8 V$ induced by the trace form on $\Sigma^\circ_{3,J}$ satisfies $\langle e_{13}\otimes e_{13},e_{31}\otimes e_{31}\rangle = 2$ and $\langle v_0,v_0\rangle = 8^2+32^2+32.36=2240$, consistently with Lemma \ref{constslem}.
\end{proof}

\begin{lemma}\label{projectionlem}
The projection of $(e_{13}-e_{31})\otimes (e_{13}-e_{31})$ to the irreducible summand $S^8 V\subset S^2(S^4 V)$ is $$v_8-\frac{1}{140}v_0 + \frac{1}{16}v_{-8}$$
where $v_8,v_0,v_{-8}$ are as in Lemma \ref{basislem}.
The remainder $(e_{13}-e_{31})\otimes (e_{13}-e_{31})-v_8+\frac{1}{140}v_0-\frac{1}{16}v_{-8}$ is a sum of vectors in the zero $h$-weight spaces in $S^4 V$ and $S^0 V$.
\end{lemma}

\begin{proof}
The coefficients of $v_8$ and $v_{-8}$ are clear.
To see that the coefficient of $v_0$ is as claimed, we simply divide $$\langle (e_{13}-e_{31})\otimes (e_{13}-e_{31}),v_0\rangle = -2\langle e_{13}\otimes e_{31},v_0\rangle = -16$$
by $\langle v_0,v_0\rangle=2240$.
The final assertion is now clear, since $(e_{13}-e_{31})\otimes (e_{13}-e_{31})-v_8-\frac{1}{16} v_{-8}$ is already in the zero $h$-weight space.
\end{proof}

\begin{lemma}\label{sl3lem}
Let $v_8,v_0,v_{-8}$ be weight vectors in a standard weight basis for $S^8 V$.
Let ${\mathbb G}_m$ act on $\mathfrak{sl}_2\oplus S^8 V$ with respective weights 2, 4.
Then the following varieties are isomorphic:

(i) The quotient $\overline{{\mathcal O}_{\mini}(\mathfrak{sl}_3)}/\langle 
\gamma\rangle$;

(ii) the closure of the $\SL_2\times{\mathbb G}_m$-orbit of $(h,v_8-\frac{1}{140}v_0+\frac{1}{16}v_{-8})$ in $\mathfrak{sl}_2\oplus S^8 V$;

(iii) the closure of the $\SL_2\times{\mathbb G}_m$-orbit of $(h,av_8+bv_0+cv_{-8})$ for any non-zero $a,b,c$ such that $ac=1225 b^2$.
\end{lemma}

\begin{proof}
By the previous discussion, the quotient in (i) is isomorphic to the closure of the $\SL_2\times{\mathbb G}_m$-orbit of $(h,v_8-\frac{1}{140}v_0+\frac{1}{16}v_{-8},w_0,z_0)$ where $w_0\in S^4 V$ and $z_0\in S^0 V$ are in the zero $h$-weight spaces, and where ${\mathbb G}_m$ acts on $S^4 V$ and $S^0 V$ with weight 4.
Now $S^2\mathfrak{sl}_2\cong S^4 V\oplus S^0 V$ and $h\otimes h$ has non-zero projection to each component.
It follows that the restriction of the quotient map $\mathfrak{so}_{3,J}\times\Sigma^\circ_{3,J}\rightarrow\mathfrak{sl}_2 \times S^8 V\times S^4 V\times S^0 V$ to $\0_{\mini}(\mathfrak{sl}_3)$ factors through the projection to $\mathfrak{sl}_2\times S^8 V$.
Hence the varieties in (i) and (ii) are isomorphic.

For the equivalence of (ii) and (iii), the standard maximal torus of $\SL_2$ acts via $${\rm diag} (t,t^{-1})\cdot (h,v_8-\frac{1}{140}v_0+\frac{1}{16}v_{-8}) = (h,t^8 v_8-\frac{1}{140}v_0+\frac{1}{16} t^{-8}v_{-8})$$
so if we choose $t$ such that $-140 t^8=\frac{a}{b}$ then we have $t^8 v_8-\frac{1}{140} v_0+t^{-8} v_{-8}=-\frac{1}{140b} (av_8+bv_0+cv_{-8})$.
We can apply a final suitable scaling to $S^8 V$ to obtain the desired isomorphism.
\end{proof}

The main goal of the above deliberations is the following result.
Let $\g$ be a simple Lie algebra of type $E_7$ and let $\0'$, $\0$ be the nilpotent orbits in $\g$ with respective Bala-Carter labels $A_3+A_2+A_1$, $A_4+A_1$.
Then $\0'\subset\overline{\0}$ is of codimension 4.
Let $\{ h,e,f\}\subset\g$ be an $\mathfrak{sl}_2$-triple with $e,f\in\0'$ and let ${\mathcal S}=f+\g^e$.

\begin{proposition}
The intersection ${\mathcal S}\cap\overline{\0}$ is isomorphic to $\overline{\0_{\mini}(\mathfrak{sl}_3)}/\langle\gamma\rangle$.
\end{proposition}

\begin{proof}
The reductive part $C=G^e\cap G^h$ (resp. ${\mathfrak c}=\g^e\cap\g^h$) of the centralizer of $e$ is isomorphic to $\SL_2$ (resp. $\mathfrak{sl}_2$), and the other irreducible ${\mathfrak c}$-submodules of ${\mathfrak g}^e$ are as follows:
$${\mathfrak g}^e(2)=S^8 V\oplus S^4 V\oplus {\mathbb C},\;\; {\mathfrak g}^e(4)=S^6 V\oplus S^2 V,\;\;{\mathfrak g}^e(6)=S^4 V.$$
We constructed standard bases in GAP for ${\mathfrak c}$ and these various submodules.
Denote by $\{e_0,h_0,f_0\}$ the basis of ${\mathfrak c}$ constructed in this way and by $v_i$ (resp. $y_i$) the weight vectors in $S^8 V$ (resp. $S^4 V\subset\g^e(2)$).
We used GAP to find a representative for ${\mathcal O}$ in ${\mathcal S}$ of the following form:
$$x=f+h_0+v-\frac{2}{7}y_0-\frac{8}{5}e+\frac{5}{648}[v,[h_0,v]]$$
where $v=18v_8+\frac{9}{70}v_0+\frac{9}{8}v_{-8}$.
It follows by dimensions and unibranchness that ${\mathcal S}\cap\overline{\0}$ is isomorphic to the closure of the $C\times{\mathbb G}_m$-orbit of $x$, hence is isomorphic to the closure in $f+{\mathfrak c}\oplus\g^e(2)$ of the $C\times{\mathbb G}_m$-orbit of $f+h_0+v-\frac{2}{70}y_0-\frac{8}{5}e$.
The result therefore follows by Lemma \ref{sl3lem}.
\end{proof}

\subsection{Slodowy slice singularities isomorphic to ${\mathbb C}^{4}/\mathfrak{S}_{3}$ or ${\mathbb C}^8/\mathfrak{S}_3$}\label{S3subsec}

In this subsection, denote by $V$ the natural representation for $\Sp_{2n}$, with basis $\{ v_1,\ldots ,v_{2n}\}$, and let ${\mathfrak h}_2$ be the 2-dimensional reflection representation for $\mathfrak{S}_3$.
The $(\Sp_{2n}\times\mathfrak{S}_3)$-module $W:=V\otimes{\mathfrak h}_2$ is isomorphic over $\mathfrak{S}_3$ to a sum of $n$ copies of ${\mathfrak h}_2\oplus{\mathfrak h}_2^*$.
Fix a primitive cube root $\omega$ of unity in ${\mathbb C}$.
We write elements of $W$ in the form $(u,v)$ where $u,v\in V$, such that $u$ (resp. $v$) is in the $\omega$- (resp. $\omega^2$-)eigenspace for $(1\; 2\; 3)$, and such that $(1\; 3)\cdot (u,v)=(v,u)$.
By Briand's Theorem \cite{Briand} (see \S \ref{G2case}), the quotient morphism $W\rightarrow W/\mathfrak{S}_3$ is given by the map $$\varphi:W\rightarrow S^2 V\times S^3 V,\;\;(u,v)\mapsto\left(u\otimes v, u\otimes u\otimes u+v\otimes v\otimes v\right).$$

There is a natural isomorphism of $S^2 V$ with the adjoint representation for ${\rm Sp}_{2n}(\C)$, and the set of elements of the form $u\otimes v$ is mapped to the set of elements of $\mathfrak{sp}_{2n}$ of rank at most 2.
This is the closure of a sheet (see \S \ref{spsubsec}) which is irreducible of dimension $4n-1$ and contains a dense open $\Sp_{2n}\times{\mathbb G}_m$-orbit.
It is easy to see that $e_{11}-e_{2n,2n}\in\mathfrak{sp}_{2n}$ is a representative of the open orbit, and it corresponds to $v_1\otimes v_{2n}\in S^2 V$.

\begin{lemma}\label{S3lem}
Let $\Sp_{2n}\times{\mathbb G}_m$ act on $\mathfrak{sp}_{2n}\oplus S^3 V$, where the weight of the ${\mathbb G}_m$-action is $2$ on $\mathfrak{sp}_{2n}$ and $3$ on $S^3 V$.
Then the quotient $({\mathfrak h}_2\oplus{\mathfrak h}_2^*)^{\oplus n}/\mathfrak{S}_3$ is isomorphic to:

(i) the closure of the ${\rm Sp}_{2n}\times{\mathbb G}_m$-orbit of $(e_{11}-e_{2n,2n},\; v_1\otimes v_1\otimes v_1+v_{2n}\otimes v_{2n}\otimes v_{2n})\in\mathfrak{sp}_{2n}\oplus S^3 V,$

(ii) the closure of the ${\rm Sp}_{2n}\times{\mathbb G}_m$-orbit of $(e_{11}-e_{2n,2n},\; av_1\otimes v_1\otimes v_1+bv_{2n}\otimes v_{2n}\otimes v_{2n})$ for any $a,b\in {\mathbb C}^\times$.
\end{lemma}

\begin{proof}
The centralizer of $s=e_{11}-e_{2n,2n}$ in $\Sp_{2n}$ is easily seen to be isomorphic to $\Sp_{2n-2}\times {\mathbb C}^\times$, with the $1$-dimensional central torus acting on $v_1$ (resp. $v_{2n}$) with weight $1$ (resp. $-1$).
It follows that $(s,v_1\otimes v_1\otimes v_1+v_{2n}\otimes v_{2n}\otimes v_{2n})$ is $\Sp_{2n}$-conjugate to $(s,t^3 v_1^{\otimes 3}+t^{-3} v_{2n}^{\otimes 3})$ and hence that the $\Sp_{2n}\times{\mathbb G}_m$-orbit of $(s,v_1^{\otimes 3}+v_{2n}^{\otimes 3})$ is of dimension $4n$, so is dense in the closure of the image of $\varphi$.
This proves part (i).
Part (ii) now easily follows from (i) by first choosing $t$ such that $at^3=bt^{-3}$, and then scaling $S^3 V$ so that we have $a=b=1$.
\end{proof}

\subsubsection{The singularity of $(G_2(a_1),A_1)$ in type $G_2$}\label{S3smallsing}

With the aid of Lemma \ref{S3lem} we can easily describe the singularity associated to this pair.
This can also be deduced from the Brylinski-Kostant result that $\overline{\0_{\mini}(\mathfrak{so}_8)}/\mathfrak{S}_3\cong \overline{\0_{G_2(a_1)}}$ (\cite[Thm. 6.6]{Brylinski-Kostant:JAMS}, see also \S \ref{G2case}).

\begin{lemma}\label{G2minlem}
Let $\0'$ (resp. $\0$) be the minimal (resp. subregular) nilpotent orbit in a Lie algebra $\g$ of type $G_2$ and let ${\mathcal S}$ be the Slodowy slice at an element of $\0'$.
Then ${\mathcal S}\cap\overline{\0}$ is isomorphic to $({\mathfrak h}_2\oplus{\mathfrak h}_2^*)/\mathfrak{S}_3$.
\end{lemma}

\begin{proof}
Let $\hat\alpha$ be the highest root in $\g$.
We choose the $\mathfrak{sl}_2$-triple $\{ h=\hat\alpha^\vee,e=e_{\hat\alpha},f=e_{-\hat\alpha}\}$.
Then ${\mathcal S}=f+\g^e$ is a Slodowy slice at $f\in\0'$.
The structure of the centralizer is as follows: $\g^e=\mathfrak{c}\oplus \mathfrak{u}_{\alpha_2}$, where $\mathfrak{c}=\langle e_{\alpha_1},f_{\alpha_1}\rangle\cong\mathfrak{sl}_2$ and ${\mathfrak u}_{\alpha_2}$ is the span of all root subspaces $\g_\alpha$ for $\alpha\in\Phi^+\setminus\{\alpha_1\}$.
Similarly, $C=G^e\cap G^f$ is (connected and is) the simple subgroup of $G$ generated by the root subgroups $U_{\pm\alpha_1}$.
As a $C$-module, $\mathfrak{u}_{\alpha_2}\cong S^3V\oplus {\mathbb C}$ where the trivial submodule is spanned by $e$ and $S^3 V$ is spanned by all root subspaces $\g_{m\alpha_1+\alpha_2}$.
The ${\mathbb G}_m$-action on $f+\g^e$ has weight $2$ (resp. $3$, $4$) on ${\mathfrak c}$ (resp. $S^3 V$, ${\mathbb C}e$).
For the standard inclusion of $\g$ into $\mathfrak{so}_7$, the element $f+h+ae_{\alpha_2}+be_{3\alpha_1+\alpha_2}+ce_{\hat\alpha}$ of the Slodowy slice is mapped to the matrix: $$\begin{pmatrix} 1 & 0 & 0 & 0 & b & c & 0 \\ 0 & -1 & a & 0 & 0 & 0 & -c \\ 0 & 0 & 2 & 0 & 0 & 0 & -b \\ 0 & 0 & 0 & 0 & 0 & 0 & 0 \\ 0 & 0 & 0 & 0 & -2 & -a & 0 \\ 1 & 0 & 0 & 0 & 0 & 1 & 0 \\ 0 & -1 & 0 & 0 & 0 & 0 & -1 \end{pmatrix}$$
Now an easy computation verifies that $f+h+e_{\alpha_2}-8e_{3\alpha_1+\alpha_2}-3e_{\hat\alpha}$ belongs to the subregular nilpotent orbit.
By dimensions and unibranchness, ${\mathcal S}\cap\overline{\0}$ equals the closure of the $(C\times{\mathbb G}_m)$-orbit of this element.
We note that the ${\mathbb G}_m$-action sends $e_{\hat\alpha}$ to $t^4 e_{\hat\alpha}=\frac{1}{2}{\rm Tr} ((t^2 h)^2)e_{\hat\alpha}$, so that the intersection ${\mathcal S}\cap\overline{\0'}$ is isomorphic to the closure of the $(C\times{\mathbb G}_m)$-orbit of $h+e_{\alpha_2}-8e_{3\alpha_1+\alpha_2}$ in ${\mathfrak c}\times S^3 V$.
By Lemma \ref{S3lem}(ii), ${\mathcal S}\cap\overline{\0}$ is isomorphic to $({\mathfrak h}_2\oplus{\mathfrak h}_2^*)/\mathfrak{S}_3$.
\end{proof}

\begin{remark}
There is another way to understand the appearance of this singularity, in relation to the shared orbit pair $(\g_2,\mathfrak{so}_8)$ (see \S \ref{G2case}).
Let $\hat\alpha$ be the highest root in $\g=\mathfrak{so}_8$ and consider the $\mathfrak{sl}_2$-triple $\{ \hat\alpha^\vee, e_{\hat\alpha}, f_{\hat\alpha}\}$.
It is easy to see that $f_{\hat\alpha}$ is fixed by the outer automorphism group $\mathfrak{S}_3$.
Let $\g=\sum_{i\in{\mathbb Z}} \g(i)$ be the grading arising from the action of $\hat\alpha^\vee$.
Then the $\mathfrak{S}_3$-representation structure of $\g(1)$ is ${\mathbb C}^2\oplus {\mathfrak h}_2\oplus {\mathfrak h}_2^*$ (where the pairing of ${\mathfrak h}_2$ and ${\mathfrak h}_2^*$ arises from the Lie bracket).
The map $U={\mathfrak h}_2\oplus{\mathfrak h}_2^*\rightarrow \g$, $x\mapsto \exp(\ad x)(f_{\hat\alpha})$ is $\mathfrak{S}_3$- and ${\mathbb G}_m$-equivariant, and on taking quotients, induces an isomorphism of $U/\mathfrak{S}_3$ with the slice.
\end{remark}

Lemma \ref{G2minlem} provides a new proof of the result in \cite[\S 4.4.3]{FJLS}.

\begin{corollary}\label{G2mincor}
Let $\0''$ be the orbit with label $\tilde{A}_1$.
Then ${\mathcal S}\cap\overline{\0''}$ is isomorphic to the non-normal singularity $m={\rm Spec}\, {\mathbb C}[s^2,st,t^2,s^3,s^2t,st^2,t^3]$.
\end{corollary}

\begin{proof}
By Lemma \ref{G2minlem}, the intersection of the Slodowy slice with $\0_{\tilde{A}_1}$ is isomorphic to the singular locus of $({\mathfrak h}_2\oplus{\mathfrak h}_2^*)/\mathfrak{S}_3$.
On the other hand, the singular points are the images in the quotient of the points with non-trivial stabilizer; up to conjugacy these are the points with $u=v$ (in the notation at the beginning of the subsection).
Set $u=\begin{pmatrix} s & t \end{pmatrix}^T\in V$; then the image of $(u,u)$ in the quotient is:
$$(u\otimes u, u\otimes u\otimes u+v\otimes v\otimes v) = \left(\begin{pmatrix} s^2 & st \\ st & t^2 \end{pmatrix}, \begin{pmatrix} 2s^3 \\ 2s^2 t \\ 2st^2 \\ 2t^3\end{pmatrix}\right),$$
where we make the appropriate identifcations of $V\otimes V$ and $S^3 V$ with ${\rm Mat}_2$ and ${\mathbb C}^4$ respectively.
The result follows.
\end{proof}

\subsubsection{The singularity of $(2A_2+2A_1,D_4(a_1)+A_1)$ in type $E_8$}\label{S3bigsing}

Let $G$ be of type $E_8$, let $\g=\Lie(G)$ and let $\0'$ (resp. $\0$) be the nilpotent orbit with Bala-Carter label $2A_2+2A_1$ (resp. $D_4(a_1)+A_1$).
Let $\{ h,e,f\}\subset\g$ be an $\mathfrak{sl}_2$-triple with $e,f\in\0'$, and let ${\mathcal S}:=f+\g^e$.

\begin{lemma}\label{2A2+2A1lem}
The intersection ${\mathcal S}\cap\overline{\0}$ is isomorphic to $({\mathfrak h}_2\oplus{\mathfrak h}_2^*)^{\oplus 2}/\mathfrak{S}_3$.
\end{lemma}

\begin{proof}
This is another application of Lemma \ref{S3lem} together with some computations in GAP, using results of \cite{Lawther-Testerman} on the structure of $\g^e$.
We choose: $$e=e_{\alpha_1}+e_{\alpha_3}+e_{\alpha_5}+e_{\alpha_6}+e_{\alpha_2}+e_{\alpha_8},\;\; f=2f_{\alpha_1}+2f_{\alpha_3}+2f_{\alpha_5}+2f_{\alpha_6}+f_{\alpha_2}+f_{\alpha_8}, \;\; h=[e,f].$$
The reductive part of the centralizer $C=G^e\cap G^f$ is isomorphic to ${\rm Sp}_4$.
Let $V$ be the natural 4-dimensional module for $C$.
Then $\g^e(1)\cong S^3 V$.
There are several other non-trivial subspaces $\g^e(i)$, but in fact this is all the information we need to prove our result.
In GAP, we constructed a basis for the centralizer, with $u_1=e_{\hat\alpha}$ (resp. $u_{20}=f_{\hat\alpha-\alpha_8}$) a highest (resp. lowest) weight vector in $\g^e(1)$.
Denote by $\beta_1^\vee,\beta_2^\vee$ the simple coroots in ${\mathfrak c}=\mathfrak{sp}_{4}$.
Then it turns out that: $$f+\beta_1^\vee+\beta_2^\vee+u_1-8u_{20}+3[u_1,u_{20}]\in\0.$$
It follows that the projection onto ${\mathfrak c}\oplus \g^e(1)$ induces an isomorphism of $(f+\g^e)\cap\overline{\0}$ with the closure of the $(\Sp_4\times{\mathbb G}_m)$-orbit of $(\beta_1^\vee+\beta_2^\vee, u_1-8u_{20})$, which is isomorphic to $({\mathfrak h}_2\oplus{\mathfrak h}_2^*)^{\oplus 2}/\mathfrak{S}_3$ by Lemma \ref{S3lem}.
\end{proof}

The attentive reader may notice the similarity of the constants $-8$ and $3$ appearing in the previous two Lemmas.
This appears to be related to Prop. \ref{d4d4Prop}.

\begin{corollary}\label{2A2+2A1cor}
Let $\0''$ be the orbit with label $A_3+2A_1$.
Then ${\mathcal S}\cap\overline{\0''}$ is isomorphic to the four-dimensional non-normal singularity $m'={\rm Spec}\, R$, where $R$ is the subring of ${\mathbb C}[s,t,u,v]$ consisting of polynomials with zero linear term.
\end{corollary}

\begin{proof}
This is a direct analogue of Cor. \ref{G2mincor}, and the proof is the same.
\end{proof}

\begin{remark}\label{modularremark}
The results of this section can be extended to positive characteristic $p$, under mild restrictions on the characteristic.
One first of all requires $p$ to be good for $G$ so that the classification of nilpotent orbits is unchanged from characteristic zero.
As well as checking the irreducibility of all the relevant Weyl modules, one also has to be careful that the Slodowy slice is transverse (which holds if $p$ is greater than the height of the orbit in question), that representatives found in the Slodowy slice are well-defined modulo $p$ and that the conditions used to test membership of $\overline\0$ are also sufficient modulo $p$.

(i) For the outer involution of $\mathfrak{sl}_3$, Lemma \ref{slnouterlem} holds in characteristic $>2$, but we require $p>7$ in order for $S^8 V$ and $S^4V$ to be irreducible $\SL_2$-modules.
Under the latter assumption, $S^2(S^4 V)$ equals the direct sum $S^8 V\oplus S^4 V\oplus k$, the form defined in Lemma \ref{constslem} is non-degenerate, the constants appearing in the denominator in Lemma \ref{projectionlem} are non-zero, the Slodowy slice ${\mathcal S}$ is transverse to $\0_{A_3+A_2+A_1}$ and the representative of $\0$ in ${\mathcal S}$ is well defined modulo $p$.
The condition we used to test membership of $\overline{\0}$ was: $\rho(x)^9=0$, ${\rm rank}(\rho(x)^8)\leq 1$, where $\rho$ is the minimal representation.
It can be checked from the tables in \cite{Stewart} that this condition is sufficient in characteristic $>7$.

(ii) It is straightforward to see (using Briand's theorem \cite{Briand} or a direct argument) that the description of $k^{2n}/\mathfrak{S}_3$ in Lemma \ref{S3lem} holds in characteristic $>3$.
This assumption also ensures that the Slodowy slice to the minimal orbit in $G_2$ is transverse. It follows almost immediately that Lemma \ref{G2minlem} holds in characteristic $>3$.
The stronger condition $p>5$ is required for the Slodowy slice to $\0_{2A_2+2A_1}$ in $E_8$ to be transverse (and for $p$ to be good).
The test for membership of $\overline\0$ was: $(\ad x)^7=0$, ${\rm rank}((\ad x)^6)\leq 2$, which is easily seen (e.g from \cite{Lawther-Testerman}) to suffice in characteristic $>5$.
Thus the statement of Lemma \ref{2A2+2A1lem} holds in characteristic $>5$.
\end{remark}

\section{Exceptional degeneration of special orbits in type $F_4$}\label{F4sec}

\subsection{An action of $\mathfrak{S}_4$ on $\mathfrak{so}_8$}\label{D4S4subsec}

We now consider the action of $\mathfrak{S}_4$ by affine diagram automorphisms of $\g=\mathfrak{so}_8$.
After conjugating we may assume:

 - that $(1\; 2)$ acts via the graph involution $\gamma_{13}$ which swaps $\alpha_1$ and $\alpha_3$, but fixes $\alpha_2$ and $\alpha_4$;

 - that $(2\; 3)$ acts via the graph involution $\gamma_{34}$ swapping $\alpha_3$ and $\alpha_4$;

 - that $(3\; 4)$ acts via $\gamma_{13}\circ\Ad\, (\alpha_1^\vee(\sqrt{-1})\alpha_3^\vee(-\sqrt{-1}))$.

It is straightforward to check that the Coxeter relations are satisfied and that therefore this gives a well-defined action of $\mathfrak{S}_4$ on $\g$.
In particular, $\mathfrak{S}_3\subset\mathfrak{S}_4$ acts via diagram automorphisms as in \S \ref{G2case}, and the normal Klein 4-subgroup $\Gamma_4$ of $\mathfrak{S}_4$ acts via the subgroup generated by the inner involutions $\Ad\, t$, where $t\in T$ satisfies $\alpha_2(t)=1$ and exactly two of $\alpha_1(t)$, $\alpha_3(t)$ and $\alpha_4(t)$ equal $-1$.
We define an action of $\mathfrak{S}_4$ on ${\rm Spin}_8$ similarly.
The fixed point subgroup (resp. subalgebra) $G_0\cong\SL_3$ (resp. 
$\g_0\cong\mathfrak{sl}_3$) has basis of simple roots $\{ 
\alpha_2,\alpha_1+\alpha_2+\alpha_3+\alpha_4\}$.
(We will fix this order for the simple roots of $G_0$, hence the isomorphism with $\SL_3$.
The action of $\mathfrak{S}_4$ on ${\mathfrak g}$ also induces an action on the adjoint group of type $D_4$; the fixed point subgroup in that case is isomorphic to 
$\SL_3\rtimes\mathfrak{S}_2$, with $\mathfrak{S}_2$ acting via an outer 
involution and representing the longest element of the Weyl group of $G$.)
The remaining direct summands in $\g$ are as follows: $\g_0$ is the derived subalgebra of a Levi subalgebra ${\mathfrak l}$ of type $A_2$, and $\mathfrak{z}(\mathfrak{l})$ is isomorphic to the 2-dimensional irreducible $\mathfrak{S}_4$-module; the span of all the root subspaces $\g_\alpha$, with $\alpha$ not fixed by the induced action of $\mathfrak{S}_3$ on the root system, is isomorphic over $G_0\times\mathfrak{S}_4$ to $(V\oplus V^*)\otimes {\mathfrak h}_3$, where $V$ is the natural representation for $\SL_3$ and ${\mathfrak h}_3$ is the 3-dimensional reflection representation for $\mathfrak{S}_4$.
Let $W=V\oplus V^*$.
Thus we have $$\g=\g_0\oplus{\mathfrak z}({\mathfrak l})\oplus (W\otimes {\mathfrak h}_3).$$

Let $\0_{\mini}$ denote the minimal nilpotent orbit in $\g$.
Our goal in this section is to prove that the singularity associated to the degeneration $(\0_{F_4(a_3)},\0_{A_2})$ in type $F_4$ is $\overline{\0_{\mini}}/\mathfrak{S}_4$ (Thm. \ref{d4S4thm}).
The methods of the previous section remain frustratingly out of reach: the minimum codimension of a $G_0$-orbit in $\0_{\mini}$ is two.
We can however describe $\overline{\0_{\mini}}/\mathfrak{S}_4$ as the closure of the image of a certain morphism (Lemma \ref{d4S4lem}).
Since $\g$ is not a sum of trivial and reflection representations for $\mathfrak{S}_4$, it is not straightforward to describe the invariants on $\g$.
We instead use the fact that $\mathfrak{S}_4$ is solvable: we first consider the quotient by $\Gamma_4$, which will reduce us to the problem of constructing a certain quotient for $\mathfrak{S}_3$.

Observe that $\g^{\Gamma_4}={\mathfrak l}$ and each non-identity element of $\Gamma_4$ acts on two of the three copies of $W$ via multiplication by $(-1)$.
It is therefore straightforward to describe the quotient morphism $\g\rightarrow\g/\Gamma_4$.
As earlier, to avoid confusion we denote elements of $T^3 W$ in the form $uvw$.

\begin{lemma}\label{Gamma4quot}
Identify $\g$ as a vector space with $\mathfrak{g}_0\oplus {\mathfrak z}({\mathfrak l})\oplus W^{\oplus 3}$, therefore denoting an element of $\g$ via a tuple $(x,z,u,v,w)$ where $x\in\mathfrak{g}_0$, $z\in {\mathfrak z}({\mathfrak l})$, and $u,v,w\in W$.
Then the map $\varphi:\g\rightarrow \mathfrak{g}_0\times {\mathfrak 
z}({\mathfrak l})\times (S^2W)^{\oplus 3}\times T^3 W$, $$(x,z,u,v,w)\mapsto 
(x,z,u\otimes u,v\otimes v,w\otimes w,uvw)$$ embeds $\g/\Gamma_4$ into 
$\mathfrak{l}\times X^{\times 3}\times Y$, where $X$ is the set of elements of $S^2W\cong\Sigma_6$ of rank at most 1, and $Y$ is the set of pure 3-tensors in $T^3 W$.
\end{lemma}

\subsection{The quotient $\overline{\0_{\mini}}/\Gamma_4$}\label{Gamma4quotsec}

In order to make practical use of Lemma \ref{Gamma4quot} to describe the quotient $\overline{\0_{\mini}}/\Gamma_4$, at this stage we appeal to some computations in GAP.
Let $\{ v_1,v_2,v_3\}$ be the standard basis for $V$ and let $\{ \chi_1,\chi_2,\chi_3\}$ be the dual basis for $V^*$ (so that $\chi_3$ is a highest weight vector).
Identify ${\mathfrak z}({\mathfrak l})$ with the set of triples $(\lambda_1,\lambda_2,\lambda_3)\in {\mathbb C}^3$ with $\sum_1^3 \lambda_i=0$, where $(\lambda_1,\lambda_2,\lambda_3)$ denotes $\lambda_1\alpha_1^\vee+\lambda_2\alpha_3^\vee+\lambda_3\alpha_4^\vee$.
Denote by ${\mathfrak z}({\mathfrak l})_{\reg}$ the subset with $\lambda_1,\lambda_2,\lambda_3$ distinct.

Let $\beta=\alpha_1+\alpha_2+\alpha_3+\alpha_4$.
We first fix the elements $e_{\alpha_2},e_\beta,f_{\alpha_2},f_\beta$ in a Chevalley basis, generating $\g_0$ as a Lie algebra.
Then we pick highest (resp. lowest) weight vectors $e_{\alpha_2+\alpha_i}$ (resp. $f_{\alpha_2+\alpha_i}$) in each copy of $V$ (resp. $V^*$), where $i=1,3,4$.
We construct the rest of the Chevalley basis for $\g$ in such a way that elements of the standard basis for each copy of $V$ are root elements, i.e. $e_{\alpha_i}=[f_{\alpha_2},e_{\alpha_2+\alpha_i}]$ and $f_{\beta-\alpha_i}=[f_\beta,e_{\alpha_i}]$ for $i=1,3,4$.
This ensures that the elements of the dual basis for each copy of $V^*$ are also root elements.

We note that $\mathfrak{S}_4/\Gamma_4\cong\mathfrak{S}_3$ acts by simultaneously permuting the indices of the $\lambda_i$ and the summands $W$ in $\g$.
The action of $g\in G_0=\SL_3$ on $\g_0$ (resp. $V$, $V^*$) is via conjugation (resp. left multiplication, right multiplication $g\cdot\chi =\chi g^{-1}$).

\begin{lemma}\label{replem}
For each $z=(\lambda_1,\lambda_2,\lambda_3)\in {\mathfrak z}({\mathfrak l})$, let $\Delta(z)={\rm diag}\, (\lambda_1,\lambda_2,\lambda_3)\in\g_0$ and let $\xi_1,\eta_1,\xi_2,\ldots ,\eta_3$ be non-zero scalars.
Then the following element of $\g_0\oplus{\mathfrak z}({\mathfrak l})_{\reg}\oplus (V\oplus V^*)^3$ belongs to $\overline{\0_{\mini}}$:
$$(\Delta(z),z,\xi_1v_1+\eta_1\chi_1, \xi_2 v_2+\eta_2\chi_2, \xi_3v_3+\eta_3\chi_3)$$
if the following conditions are satisfied:

 $\bullet \quad\xi_i\eta_i=-\prod_{j\neq i}(\lambda_i-\lambda_j)$ for $1\leq i\leq 3$;

 $\bullet \quad\eta_1\eta_2\eta_3 = \xi_1\xi_2\xi_3 = -\prod_{i<j}(\lambda_i-\lambda_j)$.
\end{lemma}

\begin{proof}
Let $T_0$ be the diagonal matrices in $G_0$ and let $\mathcal{Y}_{\reg}$ be the set of elements satisfying the conditions stated in the lemma.
Each fibre of the projection ${\mathcal Y}_{\reg}\rightarrow {\mathfrak z}({\mathfrak l})_{\reg}$ is a single $T_0$-orbit of dimension 2.
By dimensions, it will suffice to show that there is at least one element of $\overline\0_{\mini}$ in each fibre.
We checked in GAP that for any distinct $s,t\in {\mathbb C}$ and any non-zero $a,b\in {\mathbb C}$ there is an element of $\overline\0_{\mini}$ of the form:
$$t(\alpha_1^\vee+\alpha_2^\vee)+s\alpha_3^\vee + ae_{\alpha_1+\alpha_2}+\frac{t(s-t)}{a}f_{\alpha_1+\alpha_2}+be_{\alpha_3}+\frac{s(t-s)}{b}f_{\alpha_3}+\frac{st(s-t)}{ab}f_{\alpha_1+\alpha_2+\alpha_3}+\frac{ab}{t-s}e_{\alpha_1+\alpha_2+\alpha_3}.$$
Setting $\lambda_1=\frac{2t-s}{3}$ and $\lambda_2=\frac{2s-t}{3}$, it is straightforward to see that $t(\alpha_1^\vee+\alpha_2^\vee)+s\alpha_3^\vee$ corresponds to $\Delta(z)+z\in\g_0\oplus{\mathfrak z}({\mathfrak l})$.
Note that the highest (resp. lowest) weight vectors in the three copies of $V$ (resp. $V^*$) are $e_{\alpha_i+\alpha_2}$ (resp. $f_{\alpha_i+\alpha_2}$) where $i=1,3,4$.
Hence we obtain the desired result by setting $\xi_1=a$ and $\xi_2=b$.
\end{proof}

If the $\lambda_i$ are distinct in the lemma above  (i.e. if $s,t$ are non-zero) then the set of all such 
representatives of $\overline{\mathcal O}_{\mini}$ (ranging over all non-zero values of 
$a,b$) is a single $T\cap G_0$-orbit of dimension 2, with trivial stabilizer.
By dimensions, and by the fact that $\Gamma_4$ acts trivially on ${\mathfrak z}({\mathfrak l})$, we obtain the following:

\begin{corollary}\label{repcor}
The closure of the image of the morphism $\rho:G_0\times {\mathfrak z}({\mathfrak l})\rightarrow {\mathfrak g}_0\oplus{\mathfrak z}({\mathfrak l})\oplus W^{\oplus 3}=\g$, $$(g,z=(\lambda_1,\lambda_2,\lambda_3))\mapsto (g\Delta(z)g^{-1},z, g\cdot((\lambda_i-\lambda_{i+1})v_i - (\lambda_i-\lambda_{i-1})\chi_i)_{i\in{\mathbb Z}/(3)})$$
 is $\overline{\0_{\mini}}$.
\end{corollary}

\begin{proof}
This follows immediately from considering the restriction of $\rho$ to the open subset $G_0\times{\mathfrak z}({\mathfrak l})_{\reg}$, since the generic stabilizer is trivial and $G_0$ acts trivially on ${\mathfrak z}({\mathfrak l})$.
\end{proof}

Using Cor. \ref{repcor}, we can drastically simplify the description of the quotient $\overline{\0_{\mini}}/\Gamma_4$.

\begin{lemma}\label{OminGamma4quot}
Let $\varphi:\g\rightarrow\g/\Gamma_4$ be the quotient morphism as described in Lemma \ref{Gamma4quot}.
The restriction of $\varphi$ to $\overline{\0_{\mini}}$ factors through the projection
$\pi:{\mathfrak l}\times (S^2 W)^{\oplus 3}\times T^3 W\rightarrow {\mathfrak 
l}\times (S^2 V)^{\oplus 3}\times (S^2 V^*)^{\oplus 3}$.

Hence $\overline{\0_{\mini}}/\Gamma_4$ is isomorphic to the closure of the image of the morphism 
$$\Upsilon:G_0\times {\mathfrak z}({\mathfrak l})\rightarrow \mathfrak{g}_0\oplus {\mathfrak z}({\mathfrak l})\oplus (S^2V\oplus S^2V^*)^3,$$
$$(g, z=(\lambda_1,\lambda_2,\lambda_3))\mapsto g\cdot (\Delta(z),z , ((\lambda_i-\lambda_{i+1})^2 v_i\otimes v_i, (\lambda_i-\lambda_{i-1})^2 \chi_i\otimes \chi_i)_{i\in{\mathbb Z}/(3)}).$$
\end{lemma}

\begin{proof}
Let $\Psi=\varphi\circ\rho$, where $\rho$ is the map defined in Cor. \ref{repcor}.
By the Corollary, it is enough to show that $\Psi$ factors through $\pi$.
We observe:
$$S^2 W=S^2 V\oplus S^2 V^*\oplus (V\otimes V^*),\;\; T^3 W\cong T^3V\oplus (V\otimes V\otimes V^*)^{\oplus 3}\oplus (V\otimes V^*\otimes V^*)^{\oplus 3}\oplus T^3{V^*}.$$
We view $\Psi$ component-wise, as consisting of various morphisms $\Psi^i_{U}:G_0\times{\mathfrak z}({\mathfrak l})\rightarrow U$ where $U=S^2 V$, $S^2 V^*$ etc.
By Lemma \ref{Gamma4quot}, we have to show that the maps $$\Psi_{V\otimes V^*}^i:(g,z)\mapsto (\lambda_i-\lambda_{i+1})(\lambda_i-\lambda_{i-1}) g\cdot (v_i\otimes\chi_i)$$ and the analogous maps $\Psi_{T^3 V}$, $\Psi_{T^2 V\otimes V^*}$, $\Psi_{V\otimes T^2 V^*}$, $\Psi_{T^3 V^*}$ factor through $\pi$, i.e. can be expressed in terms of $\Psi^i_{S^2 V}$, $\Psi^i_{S^2 V^*}$ and the projection to $\g_0\oplus{\mathfrak z}({\mathfrak l})$.
In each case this comes down to identifying a $G_0$-equivariant polynomial map from the codomain of $\pi$ to $U$ which sends $\Upsilon(1,z)$ to $\Psi_U^i(1,z)$.

Let $\xi_i=(\lambda_i-\lambda_{i+1})$, $\eta_i=(\lambda_i-\lambda_{i-1})$ (considered as functions of $z$).
For ease of notation, we denote $\Delta(z)$ by $\Delta$.

\noindent
{\it Step 1.}
Consider the three projections to $V\otimes V^*\subset S^2 W$.
We identify $V\otimes V^*$ with $\mathfrak{gl}_3$ via the natural isomorphism (considering $V$ and $V^*$ as column and row vectors respectively).
We have $\Psi_{V\otimes V^*}^i (g,z) = \xi_i\eta_i gv_i\otimes\chi_ig^{-1}$.
Note that $\xi_i\eta_i v_i\otimes\chi_i$, when considered as an element of $\mathfrak{gl}_3$, is a diagonal matrix with a single non-zero entry $-\prod_{j\neq i}(\lambda_i-\lambda_j)$.
It is therefore easy to see that $\Psi_{V\otimes V^*}^i(1,z) = -(\Delta -\lambda_{i+1})(\Delta -\lambda_{i-1})$ as an element of $\mathfrak{gl}_3$, so this step is clear.

\noindent
{\it Step 2.}
Next we consider $\Psi_{T^3 V}(g,z)=-g\cdot\left(\prod_{i<j}(\lambda_i-\lambda_j)v_1v_2v_3\right)\in T^3 V$.
There is a natural $G_0$-equivariant map $T^2\mathfrak{gl}_3\rightarrow T^3 V$ which arises from the sequence: $$(V\otimes V^*)\otimes (V\otimes V^*)\cong (V\otimes V)\otimes (V^*\otimes V^*)\rightarrow V\otimes V\otimes\Lambda^2 V^*\cong V\otimes V\otimes V.$$
It is easy to verify that $$\begin{pmatrix} \lambda_1-\lambda_3 & 0 & 0 \\ 0 & \lambda_2-\lambda_3 & 0 \\ 0 & 0 & 0 \end{pmatrix} \otimes \begin{pmatrix} 0 & 0 & 0 \\ 0 & (\lambda_2-\lambda_1)(\lambda_2-\lambda_3) & 0 \\ 0 & 0 & 0 \end{pmatrix}\mapsto -\prod_{i<j}(\lambda_i-\lambda_j) v_1v_2v_3$$ and hence, by $G_0$-equivariance and the fact that the left hand side is $(\Delta-\lambda_3)\otimes ((\Delta-\lambda_1)(\Delta-\lambda_3))$, the map $\Psi_{T^3 V}$ factors through the projection to ${\mathfrak l}$.
A similar argument will clearly account for the component $\Psi_{T^3 V^*}$.

\noindent
{\it Step 3.}
Finally, we have to show that each of the components $\Psi_{V\otimes V\otimes V^*}^i$ factors through $\pi$.
By the description in Cor. \ref{repcor}, $$\Psi_{V\otimes V\otimes V^*}^i(1,z) = -(\lambda_{i-1}-\lambda_i)^2(\lambda_{i-1}-\lambda_{i+1})v_{i-1}v_{i+1}\chi_i.$$
We use the natural $G_0$-equivariant map which arises from the sequence:
$$\mathfrak{gl}_3\otimes S^2 V^*\hookrightarrow (V\otimes V^*)\otimes (V^*\otimes V^*) = V\otimes (V^*\otimes V^*)\otimes V^*\rightarrow V\otimes\Lambda^2 V^*\otimes V^*\cong V\otimes V\otimes V^*,$$
where we retain the order of the individual factors at the second step.
It is straightforward to verify that $(\Delta-\lambda_{i-1})\otimes (\chi_{i}\otimes\chi_{i})\mapsto (\lambda_{i-1}-\lambda_{i+1})v_{i+1}v_{i-1}\chi_{i}$.
Since $\Psi^i_{S^2 V^*}(1,z) = (\lambda_i-\lambda_{i-1})^2 \chi_i\otimes\chi_i$, this is what we require.
The same argument covers the components $\Psi_{V\otimes V^*\otimes V^*}^i$.
\end{proof}

\subsection{The quotient $\overline{\0_{\mini}}/\mathfrak{S}_4$}

The above preparation makes it possible to describe $\overline{\0_{\mini}}/\mathfrak{S}_4$, by considering the action of the quotient group $\mathfrak{S}_4/\Gamma_4\cong\mathfrak{S}_3$.
Let ${\mathfrak v}={\mathfrak g}_0\oplus {\mathfrak z}({\mathfrak l})\oplus (S^2 V\oplus S^2 V^*)^{\oplus 3}$.
Then Lemma \ref{OminGamma4quot} identifies $\overline{\0_{\mini}}/\Gamma_4$ with a closed subset $Z$ of ${\mathfrak v}$.
We first note that $\mathfrak{S}_3$ acts on ${\mathfrak v}$ by permuting the three copies of $S^2 V\oplus S^2 V^*$ and via the reflection representation on ${\mathfrak z}({\mathfrak l})$.
Hence the isotypical component of the reflection representation is isomorphic to ${\mathfrak w}\otimes{\mathfrak h}_2$ where ${\mathfrak w}={\mathbb C}\oplus S^2 V\oplus S^2 V^*$.
Fix a primitive cube root $\omega$ of $1$ in ${\mathbb C}$.

\begin{lemma}\label{vS3quotlem}
(a) With the above notation, there is an isomorphism of $G_0\times \mathfrak{S}_3$-modules
$${\mathfrak v}\rightarrow \left(\g_0\oplus S^2 V\oplus S^2 V^*\right)\oplus {\mathfrak w}\otimes {\mathfrak h}_2.$$

(b) Using the isomorphism in (a), let ${\mathfrak w}_1$ (resp. ${\mathfrak w}_2$) be the $\omega$- (resp. $\omega^2$-)eigenspace for $(1\; 2\; 3)$ on ${\mathfrak w}\oplus {\mathfrak h}_2$, and let $(1\; 3)$ act on ${\mathfrak w}_1\oplus {\mathfrak w}_2$ via $(w_1,w_2)\mapsto (w_2,w_1)$.
Then the quotient ${\mathfrak v}/{\mathfrak S}_3$ can be identified with the image of the morphism
$$\Theta:{\mathfrak v}\rightarrow \left(\g_0\oplus S^2 V\oplus S^2 V^*\right)\times S^2{\mathfrak w}\times S^3{\mathfrak w},\quad (x, w_1, w_2)\mapsto (x, w_1\otimes w_2, w_1^{\oplus 3}+w_2^{\oplus 3}).$$
\end{lemma}

\begin{proof}
Part (a) is just a summary of the above observation on the $\mathfrak{S}_3$-action.
Part (b) follows from the same argument as in \S \ref{G2case}.
\end{proof}

The rather cumbersome notation of Lemma \ref{vS3quotlem} is only a temporary necessity: our main result in this section is that, when restricted to $Z$, the quotient factors through the projection to $\g_0\oplus S^2 V\oplus S^2 V^*$.
Using the notation of Lemma \ref{OminGamma4quot}, this is summarized in the following:

\begin{theorem}\label{d4S4lem}
The quotient variety $\overline{\0_{\mini}}/\mathfrak{S}_4$ is isomorphic to the closure of the image of the morphism $\psi:\SL_3\times {\mathfrak z}({\mathfrak l})\rightarrow \mathfrak{sl}_3\oplus S^2 V\oplus S^2 V^*$, $$(g,z=(\lambda_1,\lambda_2,\lambda_3))\mapsto \left( g\Delta(z)g^{-1}, \sum_i (\lambda_i-\lambda_{i+1})^2 gv_i\otimes gv_i, \sum_i (\lambda_i-\lambda_{i-1})^2 \chi_ig^{-1}\otimes \chi_i g^{-1}\right).$$
\end{theorem}

\begin{proof}
The proof uses similar arguments to Lemma \ref{OminGamma4quot}, but is significantly more involved.
The details will therefore be dealt with in a sequence of Lemmas.
Denote the projection ${\mathfrak v}\rightarrow{\mathfrak g}_0\oplus S^2 V\oplus S^2V^*$ by $\Theta_0$.
Let $A=g\Delta(z) g^{-1}\in\g_0$ and for $i\in{\mathbb Z}/(3)$ let $C_i=(\lambda_i-\lambda_{i+1})^2gv_i\otimes gv_i$, $D_i=(\lambda_i-\lambda_{i-1})^2\chi_i g^{-1}\otimes\chi_i g^{-1}$.
Depending on the context, we consider $C_i$ to be an element of $S^2 V$ or a morphism $G_0\times{\mathfrak z}({\mathfrak l})\rightarrow S^2 V$, and similarly for $A$, $D_i$.
In the notation of Lemma \ref{OminGamma4quot}, $\pi(g,z)=(A,z,(C_i,D_i)_{i\in{\mathbb Z}/(3)})$.

We often think of $C_i$ and $D_i$ in terms of symmetric $3\times 3$ matrices, although the action of $g\in \SL_3$ on the former is via $g\cdot C_i= gC_ig^T$ and on the latter is $g\cdot D_i = (g^T)^{-1}D_ig^{-1}$.
Since $\mathfrak{S}_3$ acts by simultaneously permuting the indices in $\lambda_i$, $C_i$, $D_i$, the invariants for this action can be easily described via polarisations.
This is best illustrated by considering the action of $\mathfrak{S}_3$, which permutes the columns of: $$\begin{pmatrix} \lambda_1 & \lambda_2 & \lambda_3 \\ C_1 & C_2 & C_3 \\ D_1 & D_2 & D_3\end{pmatrix}.$$
Then $\Theta_0(\pi(g,z))$ therefore equals $(A,C_1+C_2+C_3,D_1+D_2+D_3)$.
Let $C=C_1+C_2+C_3$, $D=D_1+D_2+D_3$.
The remaining components in $\Theta$ are:

(i) quadratic term $\sum_i \lambda_i^2$ and its polarisations $\sum_i C_i\otimes C_i$, $\sum_i 
D_i\otimes D_i$, $\sum_i \lambda_i C_i$, $\sum_i \lambda_i D_i$, $\sum_i C_i\otimes 
D_i$;

(ii) cubic term $\sum_i \lambda_i^3$ and its polarisations $\sum_i C_i\otimes C_i\otimes C_i$, $\sum_i 
D_i\otimes D_i\otimes D_i$, $\sum_i \lambda_i^2 C_i$, etc.

\vspace{0.1cm}
In order to prove Thm. \ref{d4S4lem} we have to show that the various terms $\sum_i C_i\otimes C_i$, $\sum_i  C_i\otimes D_i$ etc. (which we now think of as morphisms $G_0\times{\mathfrak z}({\mathfrak l})\rightarrow S^2(S^2 V)$, $S^2 V\otimes S^2 V^*$ etc.) factor through $\Theta_0\circ\pi$.
In other words, we want to express them in terms of $A,C,D$.
We will say that each of the terms $\sum_i C_i\otimes C_i$ 
etc. {\it depends algebraically} on $A$, $C$, and $D$.
We first make the trivial observations that $\sum_i \lambda_i^2={\rm Tr}\, A^2$ and $\sum_i \lambda_i^3 = 3\lambda_1\lambda_2\lambda_3=3\,{\rm det}\, A$ depend algebraically on $A$.
Next, we note that the action of $\mathfrak{gl}_3$ on each factor in the tensor product induces an action by derivations of $\mathfrak{sl}_3$ on the symmetric square $S^2 V$.
Then $A\cdot C = 2\, \sum_i \lambda_i C_i$ depends algebraically on $A$ and $C$, and similarly for $\sum_i \lambda_i^2 C_i$, $\sum_i \lambda_i D_i$, $\sum_i \lambda_i^2 D_i$.
Using the same argument to account for $\sum_i \lambda_i C_i\otimes C_i$, $\sum_i \lambda_i D_i\otimes D_i$ and (letting $A$ act on $S^2 V$ only) $\sum_i \lambda_i C_i\otimes D_i$, it therefore suffices to prove that the quadratic and cubic terms only involving $C_i$ and $D_i$ depend algebraically on $A$, $C$, $D$.
This is established in Lemmas \ref{tlem1}, \ref{tlem2}, \ref{tlem3} and \ref{tlem4}, which will therefore complete the proof.
\end{proof}

\begin{remark}
The singularity $d_4/\mathfrak{S}_4$ has been studied in the recent mathematical physics literature in the context of (discrete quotients of) Coulomb branches of quiver gauge theories, most recently in \cite[\S 7.4]{hanany2023actions}. 
The Hasse diagram may then be predicted by the quiver subtraction algorithm with ``decorated quivers''  \cite{Bourget-Grimminger-Hanany-Zhong, Bourget-Grimminger}, or by the fission and decay algorithm \cite{Bourget-Sperling-Zhong-letter, Bourget-Sperling-Zhong}.
In fact, the whole $F_4(a_3)$ orbit closure admits an orthosymplectic magnetic quiver description \cite[(5.16)]{Hanany-Sperling}.
\end{remark}


In the following lemmas, we consider various $G_0$-equivariant maps from $\g_0\oplus S^2 V\oplus S^2 V^*$ to $S^2 V\otimes S^2 V^*$, $S^2(S^2 V)$ etc.
We want to check in each case that the map sends $\Theta_0(\pi(g,z))$ to (a scalar multiple of) $\sum_i C_i\otimes D_i$, $\sum_i C_i\otimes C_i$ etc.
Because of the $G_0$-invariance, it is enough to check this fact when $g=1$.
{\it For ease of notation, we will therefore only consider the values of $A, C_i, D_i$ when $g=1$.}

\begin{lemma}\label{tlem1}
The mixed term $\sum_i C_i\otimes D_i\in S^2 V\otimes S^2 V^*$ depends algebraically on $A$.
\end{lemma}

\begin{proof}
We will construct a $G_0$-equivariant quadratic map from $\mathfrak{sl}_3$ to $S^2 V\otimes S^2 V^*$ which sends $A$ to $\sum C_i\otimes D_i$.
In terms of symmetric tensors, we have $$\sum_i C_i\otimes D_i = \sum_i \prod_{j\neq i}(\lambda_i-\lambda_j)^2 (v_i\otimes v_i)\otimes (\chi_i\otimes \chi_i).$$
There is a natural $\GL_3$-equivariant linear map $\mathfrak{gl}_3\otimes\mathfrak{gl}_3\rightarrow S^2V\otimes S^2V^*$, by identifying $\mathfrak{gl}_3$ with $V\otimes V^*$ and projecting from $V\otimes V^*\otimes V\otimes V^*\cong (V\otimes V)\otimes (V^*\otimes V^*)$ to $S^2V\otimes S^2V^*$.
Clearly this induces a linear map $\rho:S^2(\mathfrak{gl}_3)\rightarrow S^2 V\otimes S^2 V^*$.
We have $\left(\prod_{j\neq i}(\lambda_i-\lambda_j)\right) v_i\chi_i=\prod_{j\neq i}(A-\lambda_j)$ for each $i$ and hence
 $$\sum_i C_i\otimes D_i = \rho\left( \sum_i \left(\prod_{j\neq i}(A-\lambda_j) \right)\otimes \left(\prod_{j\neq i}(A-\lambda_j)\right)\right).$$
The expression in the parentheses is a symmetric tensor which can be considered symbolically as an element of $R[\lambda_1,\lambda_2,\lambda_3]^{\mathfrak{S}_3}$, where $R={\mathbb C}[A\otimes 1 ,1\otimes A]$ and with the action given by permuting the $\lambda_i$.
Hence $R[\lambda_1,\lambda_2,\lambda_3]^{\mathfrak{S}_3}$ is generated over $R$ by $\sum\lambda_i=0$, $\sum\lambda_i^2={\rm Tr}\, A^2$ and $\lambda_1\lambda_2\lambda_3=\det A$.
Therefore the sum in the 
parentheses depends algebraically on $A$.
(This can also be checked directly.)
\end{proof}

Before we prove our next result, we describe an $\SL_3$-linear inclusion $\sigma:S^2 V\rightarrow S^2(S^2 V^*)$ as follows.
There is a non-zero linear map $S^2 V\otimes S^2 V\otimes S^2 V\rightarrow {\mathbb C}$ 
which sends $(a_1\otimes a_2)\otimes (b_1\otimes b_2)\otimes (c_1\otimes c_2)$ 
to $$\sum_{1\leq i,j\leq 2}\det ( a_1 \; b_i\; c_j)\det (a_2\; b_{3-i}\; 
c_{3-j}),$$
where the $a_i,b_i,c_i$ are here considered as column vectors.
Clearly this induces a map $S^2V\otimes S^2 (S^2 V)\rightarrow {\mathbb C}$, and hence induces a map $\sigma:S^2 V\rightarrow (S^2(S^2 V))^*\cong S^2(S^2 V^*)$.
Similarly (by means of the dual construction), we obtain a linear map $\tau:S^2(S^2 V^*)\rightarrow S^2 V$, which for convenience we normalize such that $\tau\circ\sigma=3\,{\rm Id}_{S^2 V}$.

It will be useful for us to record the effect on the elements of the standard basis.
For ease of notation, denote $v_{i_1}\otimes\ldots \otimes v_{i_r}\in S^r V$ as $v_{i_1\ldots i_r}$, and similarly for elements of $S^r V^*$.
For the map $S^2 V\rightarrow S^2(S^2 V^*)$, after multiplying by a scalar if necessary we have, for $i,j\in{\mathbb Z}/(3)$ (where all tensors are symmetric):
$$\sigma:v_{ij}\mapsto \chi_{i+1,j+1}\otimes \chi_{i-1,j-1}-\chi_{i+1,j-1}\otimes \chi_{j+1,i-1}$$
and:
$$\tau:\chi_{ij}\otimes \chi_{ll}\mapsto \left\{ \begin{array}{rl} 0 & \mbox{if $i=l$ or $j=l$,} \\ 2v_{l+1,l+1} & \mbox{if $i=j=l-1$,} \\
2v_{l-1,l-1} & \mbox{if $i=j=l+1$,} \\
-2v_{l-1,l+1} & \mbox{if $\{ i,j\} = \{ l-1,l+1\}$,} \end{array}\right.$$
while $\tau(\chi_{il}\otimes\chi_{jl})=-\frac{1}{2}\tau(\chi_{ij}\otimes\chi_{ll})$.

\begin{lemma}\label{tlem2}
The term $\sum_i C_i\otimes C_i$ (resp. $\sum_i D_i\otimes D_i$) depends algebraically on $A$, $C$ and $D$.
\end{lemma}

\begin{proof}
Using the dual maps to $\sigma$ and $\tau$, it is easy to see that there is an isomorphism of $\SL_3$-modules: $S^2(S^2 V)\cong S^4 V\oplus S^2 V^*$.
Clearly $\sum_i C_i\otimes C_i\in S^4 V$, and hence it will be enough to show that 
the projection of $C\otimes C - \sum_i C_i\otimes C_i$ to $S^4 V$ depends 
algebraically on $A$ and $D$.

To prove our assertion we consider an $\SL_3$-equivariant linear map $S^2(\mathfrak{gl}_3)\otimes S^2 V^*\rightarrow S^4 V$ as follows.
With the aid of the map $\rho:S^2(\mathfrak{gl}_3)\rightarrow S^2 V\otimes S^2 V^*$ constructed in the proof of Lemma \ref{tlem1}, there is a sequence of linear maps: $$S^2(\mathfrak{gl}_3)\otimes S^2 V^*\rightarrow (S^2 V\otimes S^2 V^*)\otimes S^2 V^*\rightarrow S^2 V\otimes S^2(S^2 V^*)\rightarrow S^2 V\otimes S^2 V\rightarrow S^4 V$$
where the penultimate map is $1\otimes \tau$.
One now observes that the composed linear map sends $\frac{1}{2}(A\otimes A)\otimes D$ to
$$ \sum_i (\lambda_{i-1}-\lambda_{i+1})^2(\lambda_i-\lambda_{i-1})^2 v_{i-1,i-1}\otimes v_{i+1,i+1}$$
which is nothing but the projection of $\frac{1}{2}\left( C\otimes C-\sum_i C_i\otimes C_i\right) = C_1\otimes C_2+C_1\otimes C_3+C_2\otimes C_3$ to $S^4 V$.
This establishes the required statement for $\sum_i C_i\otimes C_i$.
The same argument proves our assertion regarding $\sum_i D_i\otimes D_i$.
\end{proof}

\begin{lemma}\label{tlem3}
The term $\sum_i C_i\otimes C_i\otimes D_i$ (resp. $\sum_i C_i\otimes D_i\otimes D_i$) depends algebraically on $A$ and $C$ (resp. $A$ and $D$).
\end{lemma}

\begin{proof}
Let $\hat\rho:S^2 V\otimes S^2(\mathfrak{gl}_3)\rightarrow S^2(S^2 V)\otimes S^2 V^*$ be induced by the linear map $\rho:S^2(\mathfrak{gl}_3)\rightarrow S^2 V\otimes S^2 V^*$ and the canonical map $S^2 V\otimes S^2 V\rightarrow S^2(S^2 V)$.
Since $C_i\otimes D_i=\rho\left(\left(\prod_{j\neq i}(A-\lambda_j)\right)\otimes \left(\prod_{j\neq i}(A-\lambda_j)\right)\right)$ we have:
$$\hat\rho\left(\sum_i C_i\otimes\left(\prod_{j\neq i}(A-\lambda_j)\right)\otimes \left(\prod_{j\neq i}(A-\lambda_j)\right)\right) = \sum_i C_i\otimes C_i\otimes D_i.$$
In this case, the expression to which we apply $\hat\rho$ is symmetric in the second and third terms and can be considered symbolically as an element of $R[C_1,C_2,C_3,\lambda_1,\lambda_2,\lambda_3]^{\mathfrak{S}_3}$, where $R={\mathbb C}[1\otimes A\otimes 1,1\otimes 1\otimes A]$ and $\mathfrak{S}_3$ acts by simultaneously permuting the $C_i$ and the $\lambda_i$.
So our assertion in this case follows from the same fact about $\mathfrak{S}_3$-invariants used earlier.
(Alternatively we can directly compute the expression in terms of $\sum_i C_i$, $\sum_i \lambda_i C_i$, $\sum_i \lambda_i^2 C_i$.)
Replacing $V$ by $V^*$ establishes the corresponding result for $\sum_i C_i\otimes D_i\otimes D_i$.
\end{proof}

\begin{lemma}\label{tlem4}
The tensors $\sum_i C_i\otimes C_i\otimes C_i$ and $\sum_i D_i\otimes D_i\otimes D_i$ depend algebraically on $A$, $C$, $D$.
\end{lemma}

\begin{proof}
Similarly to Lemma \ref{tlem2}, we observe that $\sum_i C_i\otimes C_i\otimes 
C_i\in S^6 V$ and that there is a direct sum decomposition $S^3(S^2 V)\cong S^6 
V\oplus U$ where $U$ is the kernel of the projection $S^3(S^2 V)\rightarrow S^6 V$.
Noting further that $$C\otimes C\otimes C-3\left(\sum_i C_i\otimes C_i\right) \otimes C+2\sum_i C_i\otimes C_i\otimes C_i = 6\, C_1\otimes C_2\otimes C_3$$
(where all tensors are symmetric), it will therefore suffice to show that 
the projection of $C_1\otimes C_2\otimes C_3$ to $S^6 V$ depends 
algebraically on $A$.
We have $${\rm pr}_{S^6 V}\left(C_1\otimes C_2\otimes C_3\right) = \left( \prod_{i<j} (\lambda_i-\lambda_j)^2\right) v_{112233}.$$
To show that the right-hand side can be expressed in terms of $A$, we consider the linear map $\mu:S^2(S^2\mathfrak{gl}_3)\rightarrow S^6 V$ obtained via the composition $$S^2\mathfrak{gl}_3\otimes S^2\mathfrak{gl}_3\rightarrow \left(S^2 V\otimes S^2 V^*\right)\otimes \left( S^2 V\otimes S^2 V^*\right)\rightarrow S^2(S^2 V)\otimes S^2(S^2 V^*)\rightarrow S^2(S^2 V)\otimes S^2 V\rightarrow S^6 V$$
where the penultimate map is $1\otimes \tau$.
We note that the map $S^2(\mathfrak{gl}_3)\rightarrow S^2 V\otimes S^2 V^*$ sends $e_{ii}\otimes e_{jj}$ to $v_{ij}\otimes \chi_{ij}$.
It follows, similarly to the description of $\tau(\chi_{il}\otimes \chi_{jl})$, that $$\mu:(e_{ii}\otimes e_{jj})\otimes (e_{ll}\otimes e_{ll})\mapsto \left\{ \begin{array}{rl} 0 & \mbox{if $i=l$ or $j=l$,} \\ 2v_{112233} & \mbox{if $i=j=l\pm 1$,} \\ -2v_{112233} & \{ i,j\} = \{ l-1,l+1\},\end{array}\right.$$
and $\mu((e_{ii}\otimes e_{ll})\otimes (e_{jj}\otimes e_{ll}))=-\frac{1}{2}\mu((e_{ii}\otimes e_{jj})\otimes (e_{ll}\otimes e_{ll}))$ for all $i,j,l$.
A routine calculation therefore shows that $$\mu\left( (\Delta\otimes \Delta)\otimes (\Delta^2\otimes \Delta^2)\right) = -2(\lambda_1-\lambda_2)^2(\lambda_1-\lambda_3)^2(\lambda_2-\lambda_3)^2 v_{112233}$$
and hence $\mu\left( -\frac{1}{2} (A\otimes A)\otimes (A^2\otimes A^2)\right) = {\rm pr}_{S^6 V} C_1\otimes C_2\otimes C_3$.
Symmetry establishes the remaining case $\sum_i D_i\otimes D_i\otimes D_i$.
\end{proof}

\subsection{A transverse slice in type $F_4$}

With the benefit of Thm. \ref{d4S4lem}, we can appeal to some straightforward computations in GAP to describe the singularity associated to the last remaining minimal degeneration of special orbits in type $F_4$.
Now that Thm. \ref{d4S4lem} has been established, we can switch notation.
Thus let $G$ be a simple group of type $F_4$, let $\g=\Lie(G)$ with a fixed Chevalley basis and let $\{ h,e,f\}$ be an $\mathfrak{sl}_2$-triple with $f$ lying in the orbit $\0'=\0_{A_2}$.
Let ${\mathcal S}=f+\g^e$ be the Slodowy slice at $f$ and let $\0$ be the nilpotent orbit with label $F_4(a_3)$.

\begin{theorem}\label{d4S4thm}
With the above notation, ${\mathcal S}\cap\overline\0$ is isomorphic to $\overline{\0_{\mini}(\mathfrak{so}_8)}/\mathfrak{S}_4$ (in shorter form, $d_4/\mathfrak{S}_4$).
\end{theorem}

\begin{proof}
This follows from Thm. \ref{d4S4lem} and a computation in GAP.
With the standard numbering of the roots, we can assume that $f=2f_{\alpha_1}+2f_{\alpha_2}$ and $e=e_{\alpha_1}+e_{\alpha_2}$.
The reductive part $C:=G^e\cap G^h$ of the centralizer of $e$ has connected component $C^\circ$ isomorphic to $\SL_3$.
The standard maximal torus of $G$ contains a maximal torus of $C$, and 
$\{\alpha_4, \alpha_1+2\alpha_2+3\alpha_3+\alpha_4\}=:\{ \beta_1,\beta_2\}$ is 
a basis of simple roots of $C^\circ$.
Let $V$ be the natural module for $C^\circ=\SL_3$.
The nilradical of ${\mathfrak g}^e$ is isomorphic as a $C^\circ$-module to 
$S^2 V\oplus S^2 V^*\oplus {\mathbb C}^2$, where the trivial component is spanned by $e$ 
and $e_{\alpha_1+\alpha_2}$.
By abuse of notation, identify $\{ v_i\otimes v_j : i\leq j\}$ and $\{ 
\chi_i\otimes \chi_j : i\leq j\}$ with bases for these copies of $S^2 V$ and $S^2 
V^*$.
Let $\cg=\g^e\cap\g^h=\Lie(C^\circ)$.
Let $\lambda_1,\lambda_2,\lambda_3\in {\mathbb C}$ be distinct such that $\sum_i\lambda_i=0$, and let $\Delta 
=\lambda_1\beta_1^\vee + (\lambda_1+\lambda_2)\beta_2^\vee$, which identifies with 
$\Delta(\lambda_1,\lambda_2,\lambda_3)$ as an element of $\cg=\mathfrak{sl}_3$.
Now we observe computationally that for any such $\lambda_1,\lambda_2,\lambda_3$, the following element of $\g$:
$$f+\Delta+ \frac{1}{2}\sum_{i=1}^3 (\lambda_i-\lambda_{i+1})^2 v_i\otimes 
v_i+\frac{1}{2}\sum_{i=1}^3 (\lambda_i-\lambda_{i-1})^2 \chi_i\otimes\chi_i - 
\frac{1}{4}\sum_{i=1}^3 \lambda_i^2 e+\frac{5}{2}\lambda_1\lambda_2\lambda_3 
e_{\alpha_1+\alpha_2}$$
belongs to $\overline{\0}$.
It therefore follows, by Thm. \ref{d4S4lem} and projecting onto $\cg\oplus S^2 V\oplus S^2 V^*$, that there is a subvariety of ${\mathcal 
S}\cap\overline{\0}$ which is isomorphic to 
$\overline{\0_{\mini}(\mathfrak{so}_8)}/\mathfrak{S}_4$.
By dimensions and unibranchness, this subvariety must equal the whole intersection, 
and we are done.
\end{proof}

Note that one consequence of Thm. \ref{d4S4thm} is that ${\rm Sing}(\0,\0')$ is equal to the smooth equivalence class of ${\mathbb C}^6/\mathfrak{S}_4$.
In the next section we establish that the corresponding Slodowy slice singularity is {\it isomorphic} to ${\mathbb C}^6/\mathfrak{S}_4$ (via a computational proof).

\begin{remark}
a) All of the above arguments can be adapted to positive characteristic $p$, along similar lines to Remarks \ref{poscharremark}(ii) and \ref{modularremark}.
If $p>3$ then the arguments in \S \ref{Gamma4quotsec} go through verbatim.
For the subsequent section, we once more need to ensure that all relevant Weyl modules are irreducible and that any sums of such modules remain direct in characteristic $p$.
This issue occurs:

$\bullet$ in Lemma \ref{tlem2}, where the projection from $S^2(S^2 V^*)$ to $S^4 V$ is used. The modules $S^4 V$ and $S^2 V^*$ are irreducible in arbitrary characteristic.

$\bullet$ in Lemma \ref{tlem4}, the inclusion of $S^6 V$ as a direct summand in $S^3(S^2 V)$ is used. The complement is $\Delta(2\varpi_1+2\varpi_2)\oplus {\mathbb C}$. The Weyl modules with highest weights $6\varpi_1$ and $2\varpi_1+2\varpi_2$ remain irreducible in characteristic $>5$.

Observing the representatives found in the Slodowy slice in the proof, it therefore follows that the conclusion of Thm. \ref{d4S4thm} holds (at least) in characteristic $>5$.

b) Our description of ${\mathcal S}\cap\overline\0$ in Thm. \ref{d4S4thm} has a parallel in type $E_8$ which we do not fully understand: in that case, a slice from $\0_{A_4}$ to $\0_{E_8(a_7)}$ has an open subset which is a fibre bundle (with fibres of dimension 4) over the regular semisimple elements of $\mathfrak{sl}_5$.
Thm. \ref{F4excslicethm} and Thm. \ref{E8excslicethm} can be deduced (with some effort in the latter case) from these fibre bundle descriptions.

c) The component group of $C$ in Thm. \ref{d4S4thm} is $\mathfrak{S}_2$.
It follows that $C$ acts on $d_4/\mathfrak{S}_4$ as the full automorphism group (i.e. the $\mathfrak{S}_4$-fixed point subgroup in the adjoint group of type $D_4$).
\end{remark}

\section{Computational results}\label{compsec}

In \S\S \ref{sharedorbits}-\ref{F4sec}, we described various quotients $\overline\0/\Gamma$ where $\0$ was (usually) a nilpotent  orbit in a Lie algebra $\g={\rm Lie}(G)$, and $\Gamma$ was a finite group of automorphisms of $G$.
We exploited as much as possible the action of the group $G_0:=G^\Gamma$ on the quotient.
This line of attack is powerless when the $G_0$-orbits have large codimension.
For example, we have no hope of generalizing Lemma \ref{S3lem} to the action of $\mathfrak{S}_r$ on ${\mathfrak h}_{r-1}\oplus{\mathfrak h}_{r-1}^*$ when $r\geq 4$: in this case $\dim G_0=3$ is too small compared to the dimension $2(r-1)\geq 6$ of the quotient.
However, the cases $r=4$ and $r=5$ are of particular interest to us, since we have asserted (and will now prove) that they are the Slodowy slice singularities associated to the degenerations $\0_{F_4(a_3)}> \0_{A_2+\tilde{A}_1}$ in type $F_4$ and $\0_{E_8(a_7)}>\0_{A_4+A_3}$ in type $E_8$.
Note that $\0_{F_4(a_3)}$ (resp. $\0_{E_8(a_7)}$) is special with component group $\mathfrak{S}_4$ (resp. $\mathfrak{S}_5$), and the minimal orbit in the special piece is $\0_{A_2+\tilde{A}_1}$ (resp. $\0_{A_4+A_3}$).
Hence these results, found in \S \ref{F4subsec} and \S \ref{E8subsec}, are a crucial part of our Main Theorem.

In \S \ref{furthersubsec}, we apply similar computational methods to identify two more interesting quotient singularities in exceptional nilpotent cones: we show that there is a Slodowy slice singularity in $E_7$ (resp. $E_6$) which is isomorphic to $S^3({\mathbb C}^2/\mu_2)$ (resp. $S^2({\mathbb C}^2/\mu_3)$), where $\mu_d$ denotes a cyclic group of order $d$.
Our results in \S \ref{F4subsec}, \S \ref{E8subsec} and \S \ref{furthersubsec} depend on some knowledge of generators and relations for the corresponding invariant rings.
We collect this information in \S \ref{invsubsec}, adding some speculation on the numbers of relations for arbitrary symmetric groups $\mathfrak{S}_n$ or wreath products $G(d,1,n)$.
In the final subsection we extend this series of results to prove that the Slodowy slice singularity for the degeneration $\0_{E_8(a_6)}>\0_{E_8(b_6)}$ is $a_2/\mathfrak{S}_4$, for an appropriate action of $\mathfrak{S}_4$.
There is one orbit intermediate between $\0_{E_8(a_6)}$ and $\0_{E_8(b_6)}$, with Bala-Carter label $D_7(a_1)$.
One corollary of this rather surprising result is that the singularity associated to the minimal degeneration $\0_{D_7(a_1)}>\0_{E_8(b_6)}$ is the non-normal surface singularity labelled $\mu$ in \cite{FJLS}.

\subsection{Structure of invariant rings}\label{invsubsec} 

\subsubsection{The quotients $({\mathfrak h}_{n-1}\oplus{\mathfrak h}_{n-1}^*)/\mathfrak{S}_n$}\label{Sninvs}

In the next two subsections we consider the invariants for $\mathfrak{S}_n$ (with $n=4,5$) acting on a sum of two copies of the reflection representation ${\mathfrak h}_{n-1}$.
We note that ${\mathfrak h}_{n-1}^*\cong{\mathfrak h}_{n-1}$, so this is the quotient singularity appearing in the statement of our Main Theorem.
It is well known that the invariant polynomials on a single copy of ${\mathfrak h}_{n-1}$ form a polynomial ring on $(n-1)$ generators, of degrees $2,3,\ldots ,n$.
The generators can be easily understood by identifying ${\mathfrak h}_{n-1}$ with the traceless diagonal $n\times n$ matrices $x$, where $\mathfrak{S}_n$ acts by permuting the entries.
Then the functions $x\mapsto {\rm tr}(x^i)$, $2\leq i\leq n$ form a system of (free) generators.

It is also well known that the invariant polynomials on a finite sum of copies of ${\mathfrak h}_{n-1}$ are generated by polarisations of these invariants, as we explained in \S \ref{G2case}.
For pairs $(x,y)$ of traceless diagonal $n\times n$ matrices, the polarisations are the functions $g_{ij}:(x,y)\mapsto {\rm tr}(x^i y^j)$ for $2\leq i+j\leq n$.
Relations between these generators are known \cite{Domokos-Puskas}, but appear to be rather complicated for $n\geq 4$.
We could not find a result in the literature describing the number of generators of the ideal of relations; indeed, the problem may be open (see \cite[Problem 2.12]{Haiman}).
We are specifically interested in a sum of two copies of ${\mathfrak h}_{n-1}$ when $n=4$ or $n=5$.

\noindent
(i) For $n=4$ there are twelve generators, with three generators of degree $2$, four of degree $3$ and five of degree $4$.
We used GAP to find fifteen relations: three of degree $6$ and six each of degrees $7$ and $8$.
We were able to verify in Magma that this ideal is radical (of height $6$), so this is a complete set.

\noindent
(ii) For $n=5$ we add six more generators, of degree $5$.
We found thirty-five relations: there are $4$ (resp. $9$, $12$, $10$) of degree $7$ (resp. $8$, $9$, $10$).
We have not been able to confirm that this is a complete set.
(The calculation in Magma did not complete.)

\noindent
(iii) For $n=6$, there are twenty-five generators in total.
We used GAP to find seventy relations: there are $5$ (resp. $12$, $18$, $20$, $15$) of degree $8$ (resp. $9$, $10$, $11$, $12$).
We cannot confirm that this is a complete set.


On the basis of these results (and those for $n=2,3$), we conjecture that for general $n$ there are $\begin{pmatrix} n+2 \\ 4 \end{pmatrix}$ relations in total, and that $\tfrac{1}{2}j(j+1)(n-j)$ of them are of degree $n+j+1$, for $1\leq j\leq n$.

\subsubsection{The quotients $S^n({\mathbb C}^2/\mu_d)$}\label{wrprodinvs}

Recall that $W(B_n)$ (with its action on $V={\mathbb C}^n$) can be described as the signed symmetric group, i.e. the group of monomial $n\times n$ matrices with non-zero entries $\pm 1$.
The subgroup $D$ of diagonal matrices is normal and induces an isomorphism of $W(B_n)$ with the wreath product $\mu_2\wr\mathfrak{S}_n$, that is, the semidirect product $\mu_2^n\rtimes\mathfrak{S}_n$.
Replacing $-1$ by a primitive $d$-th root of unity, we obtain the complex reflection group $\mu_d\wr\mathfrak{S}_n$, often denoted $G(d,1,n)$.
In particular, $W(B_n)=G(2,1,n)$.

Let $V$ be the reflection representation for $\Gamma=G(d,1,n)$.
The quotient $(V\oplus V^*)/\Gamma$ is isomorphic to $S^n({\mathbb C}^2/\mu_d)$.
One can obtain an easy upper bound on the number of generators of the invariant ring, by considering the action of $\mathfrak{S}_n$ on the quotient ${\mathbb C}^{2n}/D = {\mathbb C}^2/\mu_d \times {\mathbb C}^2/\mu_d\times\ldots \times{\mathbb C}^2/\mu_d$.
Writing $x_i,y_i$ for the standard coordinates on the $i$-th copy of ${\mathbb C}^2$, we see that ${\mathbb C}[V\oplus V^*]^D$ is generated by $x_i y_i$, $x_i^d$ and $y_i^d$ for $1\leq i\leq n$.
Furthermore, $\Gamma/D\cong\mathfrak{S}_n$ acts by permuting the columns of the following matrix:
$$\begin{pmatrix} x_1 y_1 & x_2 y_2 & \ldots & x_n y_n \\ x_1^d & x_2^d & \ldots & x_n^d \\ y_1^d & y_2^d & \ldots & y_n^d\end{pmatrix}.$$
This leads to a set of homogeneous generators for ${\mathbb C}[V\oplus V^*]^\Gamma$, as polarisations (using the columns of the matrix above) of the elementary invariants $\sum_1^n x_i y_i$, $\sum_1^n x_i^2 y_i^2$, $\ldots$, $\sum_1^n x_i^n y_i^n$.
There are $3$, resp. $6$, $10$ etc. polarisations of these terms, giving a total of $\frac{1}{6}n(n^2+6n+11)$ generators.
In general, this is not a minimal generating set.
We consider the first few values of $n$.


(i) For $n=1$, the generating set $\{ x_1 y_1, x_1^d, y_1^d\}$ is minimal, with a single relation.

(ii) The case $n=2$ is most easily understood using a modified list of $D$-invariants: we note that $t=x_1 y_1+x_2 y_2$, $u=x_1^d+x_2^d$ and $v=y_1^d+y_2^d$ are $\mathfrak{S}_2$-invariant, while $\delta=x_1 y_1-x_2 y_2$, $\xi=x_1^d-x_2^d$, $\eta=y_1^d-y_2^d$ are semi-invariants for $\mathfrak{S}_2$.
We therefore obtain a set of nine generators for ${\mathbb C}[V\oplus V^*]^{\Gamma}$: the ``linear'' invariants $t,u,v$, and the ``quadratic'' invariants $\delta^2$, $\delta\xi$, $\delta \eta$, $\xi^2$, $\xi\eta$, $\eta^2$.
Moreover, $\xi\eta+uv=2(x_1^d y_1^d+x_2^d y_2^d)\in{\mathbb C}[x_1 y_1, x_2 y_2]^{{\mathfrak S}_2} = {\mathbb C}[t,\delta^2]$, so the generator $\xi\eta$ is redundant.
Using the bigrading of the invariant ring discussed below, it can be shown that the resulting set of eight generators is minimal.
There are six obvious relations between these generators:
$$\delta^2.\xi^2 = (\delta\xi)^2,\quad \delta^2.\eta^2=(\delta\eta)^2,\quad \xi^2.\delta\eta = \delta\xi.\xi\eta, \quad \eta^2.\delta\xi = \delta\eta.\xi\eta,\quad \xi^2\eta^2=(\xi\eta)^2, \quad \delta^2.\xi\eta = \delta\xi.\delta\eta, $$
where in the last four relations we should replace $\xi\eta$ by the relevant expression in $-uv+{\mathbb C}[t,\delta^2]$.
It is also straightforward to see that $\xi v+\eta u = 2(x_1^d y_1^d-x_2^dy_2^d)\in \delta{\mathbb C}[t,\delta^2],$ and hence we obtain three further relations:
$$\delta\xi.v+\delta\eta.u\in\delta^2{\mathbb C}[t,\delta^2],\quad \xi^2.v+\xi\eta.u\in\delta\xi{\mathbb C}[t,\delta^2],\quad \xi\eta.v+\eta^2.u\in\delta\eta{\mathbb C}[t,\delta^2].$$
We therefore obtain nine relations satisfied by the set of eight generators.
We have checked in Magma that this is a complete set of relations when $d=2$ or $3$.
The following bigrading on $(\mu_d\times\mu_d)$-invariant ring turns out to be useful: the monomial $x_i^{a+md}y_i^a$ has bidegree $(2a+md,m)$ for any $a,m\in{\mathbb Z}$ such that $a,a+md\geq 0$.
The first component of the bidegree is the ordinary degree; the second is a (suitably normalized) encapsulation of the action by Poisson bracket with $t$.
In the special case of $G(3,1,2)$, the degrees of the generating invariants are as follows:
$$\deg (t)=(2,0),\quad \deg(u)=(3,1),\quad \deg(v)=(3,-1),$$ $$\deg(\delta^2)=(4,0),\quad\deg(\delta\xi)=(5,1),\quad \deg(\delta\eta)=(5,-1),\quad \deg(\xi^2)=(6,2),\quad \deg(\eta^2)=(6,-2).$$

(iii) The case $n=3$ could be tackled by a similar approach, but with rather involved considerations.
When $d=2$, i.e. $\Gamma=W(B_3)$, we have checked (using GAP to obtain the relations and Magma to check that the associated ideal is radical) that there is a minimal set of fifteen generators with thirty six relations.

These results are rather striking, and suggest the following conjecture: that, for $d>1$, the invariant ring ${\mathbb C}[V\oplus V^*]^{G(d,1,n)}$ has the same number of generators and relations as an $a_n$ singularity, i.e. $n(n+2)$ generators and $\frac{1}{4}n^2(n+1)^2$ relations.
(Note in particular that this is independent of $d>1$.
It is interesting to observe that the conjecture requires $\frac{1}{6}n(n^2-1)$ ``redundant'' generators among the polarisations discussed above.)
When $d=2$ (i.e. for $\Gamma=W(B_n)$), it is straightforward to confirm that this is the correct cardinality of a minimal set of generators: there are $2m+1$ generators of degree $2m$ for $m=1,\ldots ,n$, summing to $n(n+2)$ in total.
Furthermore, the three generators of degree 2 span a copy of $\mathfrak{sl}_2$ (for the Poisson bracket), over which the generators of degree $2m$ span a copy of the irreducible module $V(2m)$ (for $1\leq m\leq n$).

\begin{remark}
A heuristic argument in support of this conjecture is as follows.
Let $\Gamma=G(d,1,n)$ and let ${\mathfrak c}$ be the vector space of $\Gamma$-invariant ${\mathbb C}$-valued functions on the set of complex reflections in $\Gamma$.
Denote by $H_{t,c}(\Gamma)$ the rational Cherednik algebra determined by $(t,c)\in{\mathbb C}\oplus{\mathfrak c}$.
It is well known that the assignment $c\mapsto eH_{0,c}(\Gamma)e$ (where $e$ is the average of the group elements of $\Gamma$) defines a flat deformation of $eH_{0,0}e\cong {\mathbb C}[V\oplus V^*]^{\Gamma}$.
Recall that the quotient $(V\oplus V^*)/G(d,1,n)$ has a symplectic resolution, which can be described as the composition ${\rm Hilb}^n(\widetilde{{\mathbb C}^2/\mu_d})\rightarrow S^n(\widetilde{{\mathbb C}^2/\mu_d})\rightarrow S^n({\mathbb C}^2/\mu_d)$, where $\widetilde{{\mathbb C}^2/\mu_d}\rightarrow {\mathbb C}^2/\mu_d$ is the minimal resolution of singularities.
Existence of a symplectic resolution is equivalent to smoothness of ${\rm Spec}(eH_{0,c}e)$ for generic values of the parameter $c$.
The values of $c$ for which ${\rm Spec}(eH_{0,c}e)$ is singular form a union of hyperplanes $H_1,\ldots ,H_r$ in ${\mathfrak c}$.
It is known that there is one such hyperplane, $H_1$ say, such that for any $c\in H_1\setminus \cup_2^r H_i$, ${\rm Spec}(eH_{0,c}e)$ has a unique isolated singular point which is smoothly equivalent to an $a_n$ singularity.\footnote{We thank Gwyn Bellamy for explaining that this can be obtained using a similar calculation to the proof of \cite[Thm.5.5]{Bellamy-Craw}.}
Thus a minimal set of generators for $eH_{0,0}e$ can be deformed and localized to give a generating set (perhaps not minimal) for $a_n$.
Thus ${\mathbb C}[V\oplus V^*]^\Gamma$ has at least $n(n+2)$ generators.
Since the deformation is flat, then a minimal free resolution of the ideal of relations of ${\mathbb C}[V\oplus V^*]^\Gamma$ deforms and localizes to a free resolution of $a_n$.
Thus it also follows that $\frac{1}{4}n^2(n+1)^2$ is a lower bound for the number of relations ${\mathbb C}[V\oplus V^*]^\Gamma$.
\end{remark}

\subsection{The exceptional special slice singularity in $F_4$}\label{F4subsec}

Let $\g$ be a simple Lie algebra of type $F_4$, let $\0'$ (resp. $\0$) be the nilpotent orbit in $\g$ with Bala-Carter label $A_2+\tilde{A}_1$ (resp. $F_4(a_3)$), let $\{ h,e,f\}$ be an $\mathfrak{sl}_2$-triple with $e\in\0'$ and let ${\mathcal S}=f+\g^e$ be the Slodowy slice.
Our goal is to show that the intersection ${\mathcal S}\cap\overline{\0}$ is isomorphic to $({\mathfrak h}_3\oplus{\mathfrak h}_3^*)/\mathfrak{S}_4$.
Let $\rho:{\mathfrak g}\rightarrow\mathfrak{so}_{26}$ be the minimal faithful representation of $\g$.


\begin{theorem}\label{F4excslicethm}
Let $\g,\0',\0,{\mathcal S}$ be as above.
Then ${\mathcal S}\cap\overline{\0}$ is isomorphic to $({\mathfrak h}_3\oplus{\mathfrak h}_3^*)/\mathfrak{S}_4$.
\end{theorem}

\begin{proof}
It is easy to check that $x$ belongs to $\overline\0$ if and only if $\rho(x)^5=0$.
That this is a sufficient condition follows from comparing Jordan block sizes on $\g$ and on the Lie algebra of type $E_6$: the two orbits which are minimal among those not contained in $\overline\0$, with labels $B_3$ and $C_3$, have blocks on $\rho$ of size $7$ and $9$ respectively.

Let $\{ h,e,f\}$ be an $\mathfrak{sl}_2$-triple with $e\in\0'$.
The centralizer ${\mathfrak z} = \g^e$ has dimension $18$, with reductive centralizer ${\mathfrak c}={\mathfrak g}^e\cap{\mathfrak g}^f$ isomorphic to $\mathfrak{sl}_2$. The ${\mathfrak c}$-module structure of ${\mathfrak z}$ is as follows:
${\mathfrak z}=V(2)\oplus V(3)\oplus V(4)\oplus W$, where $V(i)$ denotes the simple $\cg$-module of highest weight $i$, and $W$ is a reducible submodule of dimension six. Note that $\cg$ is the copy of $V(2)$.


Let $f+z\in{\mathcal S}$, where $z\in\g^e$; we first observed that the conditions ${\rm tr}(\rho((f+z)^2))=0$ and $\rho(f+z)^5=0$ imply certain conditions of the form $(\chi+p)(z)=0$, where $\chi\in W^*$ and $p$ is a polynomial in $V(2) \oplus V(3) \oplus V(4)$.
In particular, we were able to construct an explicit subvariety $Y$, which contains ${\mathcal S}\cap\overline\0$ and is isomorphic to $V(2)\oplus V(3)\oplus V(4)$.
We could then check that the remaining conditions $\rho((f+z)^5)$, when restricted to $Y$, give precisely the fifteen relations satisfied by the generators of ${\mathbb C}[{\mathfrak h}_3\oplus{\mathfrak h}_3^*]^{\mathfrak{S}_4}$ (as outlined in \S \ref{Sninvs}).
The statement in the theorem follows.
We note one consequence of this calculation: the conditions ${\rm tr}(\rho(x)^2)=0$ and $\rho(x)^5=0$ cut ${\mathcal S}\cap\overline\0$ out as a subscheme (not just a subset) of ${\mathcal S}$.

Our proof turns out to be more transparent in reverse.
Thus we can use the above to construct a morphism from $({\mathfrak h}_{3}\oplus{\mathfrak h}_{3}^*)/\mathfrak{S}_4$ to the intersection ${\mathcal S}\cap\overline\0$.
Since all of our generating invariants $g_{ij}$ appear, this map is a closed immersion.
By general results on branching and dimension, the image equals the intersection ${\mathcal S}\cap\overline{\0}$.
One can verify that the image lies in $\overline{\0}$ by checking the condition $\rho(x)^5=0$.
\end{proof}

\begin{remark}\label{F4calcremark}
a) It is possible to argue from Thm. \ref{d4S4thm} that ${\mathcal S}\cap\overline{\0}$ is isomorphic to ${\mathbb C}^6/\mathfrak{S}_4$.
(A smooth equivalence of singularities follows more or less immediately; our result establishes an isomorphism as varieties.)
We have included the above computational proof because it helps to illustrate the general approach, which we will also employ in the following subsection.
The precise computation will also be useful in the proof by the second, third and fourth authors of (the global version of) Lusztig's special pieces conjecture \cite{Fu-Juteau-Levy-Sommers:GeomSP}.

b) The symplectic leaves in $({\mathfrak h}_3\oplus{\mathfrak h}_3^*)/\mathfrak{S}_4$ are in one-to-one correspondence with the parabolic subgroups of $\mathfrak{S}_4$ up to conjugacy.
This is summarized in the following diagram:
$$
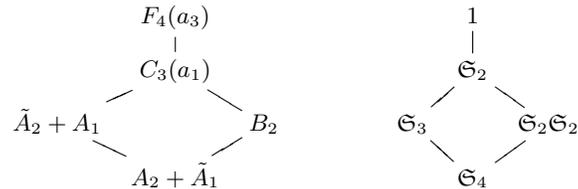
\small{\xymatrix@=.2cm{
{} & F_4(a_3) \ar@{-}[d] & {} &&&& {} & 1 \ar@{-}[d] & {} \\
{} & C_3(a_1) \ar@{-}[dr] \ar@{-}[dl] & {} &&&& {} & \mathfrak{S}_2 \ar@{-}[dr] \ar@{-}[dl] & {} \\
\tilde{A}_2+A_1 \ar@{-}[dr] & {} & B_2 \ar@{-}[dl] &&&& \mathfrak{S}_3 \ar@{-}[dr] & {} & \mathfrak{S}_2\mathfrak{S}_2 \ar@{-}[dl] \\
{} & A_2+\tilde{A}_1 & {} &&&& {} & \mathfrak{S}_4 & {} }}
$$
\end{remark}

\subsection{The exceptional special slice singularity in $E_8$}\label{E8subsec}

The notation is as in the previous subsection, except that $\g$ is simple of type $E_8$, $\0$ has label $E_8(a_7)$ and $\0'$ has label $A_4+A_3$.  
We used a computer to show that ${\mathcal S}\cap\overline{\0}$ 
has the form predicted by the Main Theorem.


\begin{theorem}\label{E8excslicethm}
The slice ${\mathcal S}\cap\overline{\0}$ is isomorphic to $({\mathfrak h}_4\oplus{\mathfrak h}_4^*)/\mathfrak{S}_5$.
\end{theorem}

\begin{proof}
The weighted Dynkin diagram of $\0$ is nonzero only on the fifth simple root, where its value is $2$.  The highest root has coefficient $5$ on this simple root, which implies that $10$ is the largest 
eigenvalue for $\ad(h)$, where $h$ is the semisimple element of the $\sl2$-subalgebra for $\0$.  
It follows that $({\rm ad}\, x)^{11}=0$ for all $x$ in the closure of $\0$.
The two nilpotent orbits which are minimal among those not contained in $\overline\0$, namely $D_5$ and $A_6$, have Jordan blocks of sizes $15$ and $13$ respectively on the adjoint representation.  It follows that the zero set of the equations $({\rm ad}\, x)^{11}$ for $x \in \g$ is exactly $\overline{\0}$.

The centralizer $\mathfrak z$ in $\g$ of an element in $\0'$ has dimension $48$.
The reductive part $\cg$ of $\mathfrak z$ is isomorphic to $\sl2$, as in the $F_4$ case. 
As a $\cg$-module, $\mathfrak z$ decomposes as $V(2) \oplus V(3) \oplus V(4) \oplus V(5) \oplus W$, where $W$ is a $30$-dimensional reducible $\sl2$-representation.
We found in Magma a basis for $\mathfrak z$ that simultaneously respected this $\sl2$-decomposition and the Kazhdan grading.  
(For our calculations in GAP, we used the basis constructed in \cite{Lawther-Testerman}.)
Then by restricting the generic matrix $({\rm ad}\, x)^{11}$ to $\mathcal S$ using the dual basis, we located 
several equations that contained linear terms in the dual basis, allowing us to eliminate those variables.
After eliminating those variables (also using the condition ${\rm tr}(({\rm ad}\, x)^2)=0$), we again found more variables we could eliminate.  Continuing in this way, we were able to eliminate all $30$ variables coming from $W$.  

We thus constructed a subvariety $Y$ of ${\mathcal S}$, which is isomorphic to $V(2)\oplus V(3)\oplus V(4)\oplus V(5)$ and which contains ${\mathcal S}\cap\overline\0$.
Because we do not know that we have obtained a complete set of relations for $({\mathfrak h}_4\oplus {\mathfrak h}_4^*)/\mathfrak{S}_5$, we could not repeat the next step in the proof of Thm. \ref{F4excslicethm}.
Instead, we used some additional equations in $({\rm ad}\, x)^{11}$ to parametrize an $8$-dimensional ${\mathbb G}_m$-stable subvariety $X$ of $Y$ that could potentially lie in the zero set of $({\rm ad}\, x)^{11}$ and be isomorphic to $({\mathfrak h}_4\oplus{\mathfrak h}_4^*)/\mathfrak{S}_5$ and also respect the action of $\cg$.  
However, we could not easily show that $X \subset \overline{\0}$.  We have two different methods to accomplish this fact.  One way was to check the condition $({\rm ad}\, x)^{11}=0$ in GAP, column by column, for $x \in X$. This was very time consuming, requiring a few hours for each column.  A faster way was to pick a small set of the matrix equations and show that they vanish on $X$.  Then we checked that the zero set of this small set of equations on $\mathcal S$ at a regular point of $X$ has dimension $8$.  
(We ported the small set of equations used in the proof to Singular, and found that the ideal they generate has a minimal set of $35$ generators, and that the degrees match those of the known relations for ${\mathbb C}[{\mathfrak h}_4\oplus {\mathfrak h}_4^*]^{{\mathfrak S}_5}$, as discussed in \S \ref{Sninvs}.)
As the full ideal of equations on $\mathcal S$ has zero set of dimension $8$ and is irreducible (since $\overline\0$ is unibranch at $\0'$), we know that the zero set must be exactly $X$, completing the proof.
\end{proof}

\begin{remark}\label{E8calcremark}
a) Theorems \ref{F4excslicethm} and \ref{E8excslicethm} give new proofs of the appearance of the non-normal singularity $m$ for the degenerations $(A_2+\tilde{A}_1,\tilde{A}_2+A_1)$ in $F_4$ and $(A_4+A_3,A_5+A_1)$ and $(A_4+A_3,D_5(a_1)+A_1)$ in $E_8$.
For example, the first singularity is given by the image in the quotient of all points with stabilizer equal to $\mathfrak{S}_3\subset\mathfrak{S}_4$.
This corresponds to a pair of diagonal matrices of the form:
$(x,y)=({\rm diag}\, (s,s,s,-3s), {\rm diag}\, (t,t,t,-3t))$.
It is easy to see that the values of the invariants ${\rm tr}(x^iy^j)$ induce functions generating the ring ${\mathbb C}[s^2,st,t^2,s^3,s^2t,st^2,t^3]$.

b) Similarly to the $F_4$ case, we have a poset equivalence between the orbits in ${\mathcal P}(\0)$ and the parabolic subgroups of $\mathfrak{S}_5$ up to conjugacy:

$$
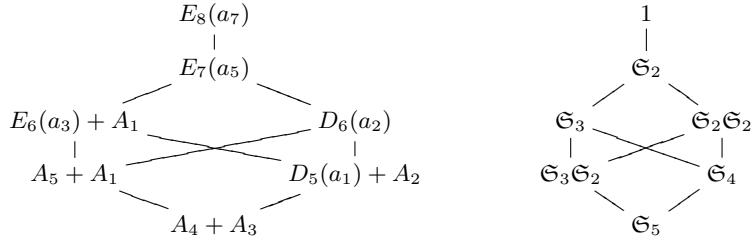
\small{\xymatrix@=.2cm{
{} & E_8(a_7)\ar@{-}[d] & {} &&&& {} & 1 \ar@{-}[d] & {} \\
{} & E_7(a_5) \ar@{-}[dl]\ar@{-}[dr]   & {} &&&& {} & \mathfrak{S}_2 \ar@{-}[dl]\ar@{-}[dr]  & {} \\
E_6(a_3)+A_1 \ar@{-}[d] \ar@{-}[drr]  & {} & D_6(a_2) \ar@{-}[d]\ar@{-}[dll]   &&&& \mathfrak{S}_3 \ar@{-}[d]\ar@{-}[drr]    & {} & \mathfrak{S}_2\mathfrak{S}_2\ar@{-}[d]\ar@{-}[dll]    \\
A_5+A_1\ar@{-}[dr]  & {} & D_5(a_1)+A_2\ar@{-}[dl]  &&&& \mathfrak{S}_3\mathfrak{S}_2\ar@{-}[dr]   & {} & \mathfrak{S}_4\ar@{-}[dl]   \\
{} & A_4+A_3 & {} &&&& {} & \mathfrak{S}_5 & {}  }}
$$

c) We are thankful to Amihay Hanany for the following (conjectural) observation: the larger slice from $A_4+A_2+A_1$ to $E_8(a_7)$ should also be a quotient by $\mathfrak{S}_5$.
One expects this to be the quotient ${\mathcal C}_5/\mathfrak{S}_5$ where ${\mathcal C}_5$ is the Coulomb branch of the quiver gauge theory with a central $U(2)$-node attached to five $U(1)$-nodes.
This is a direct generalisation of the exceptional special slice in type $F_4$.
This observation also connects to the family of symplectic singuilarities that we studied (with Bellamy and Bonnaf\'e) in \cite{BBFJLS}: a `bottom edge' in the Hasse diagram of ${\mathcal C}_k/\mathfrak{S}_k$ is expected to be the singularity there denoted ${\mathcal Y}(k)$, as per the discussion in \cite[\S 4.4]{Bourget-Grimminger}.
Note that ${\mathcal Y}(4)=a_2$ and ${\mathcal Y}(5)$ is denoted $\chi$ in \cite{FJLS} (and is the singularity of the minimal degeneration $A_4+A_2+A_1<A_4+A_3$).
\end{remark}

\subsection{Two further quotient singularities}\label{furthersubsec}

Let $\g$ be a simple Lie algebra of type $E_7$, let $\0$ be the nilpotent orbit in $\g$ with Bala-Carter label $E_7(a_5)$, and let ${\mathcal S}$ be a Slodowy slice to an element of the orbit $\0'=\0_{A_4+A_2}\subset \overline\0$.

\begin{theorem}\label{B3thm}
With ${\mathcal S}$ and $\0$ as above, the slice ${\mathcal S}\cap\overline\0$ is isomorphic to $S^3({\mathbb C}^2/\mu_2)$.
\end{theorem}

\begin{proof}
Let ${\mathfrak h}_{B_3}$ be a Cartan subalgebra of $\mathfrak{so}_7$, considered as a reflection representation for $W(B_3)$.
Recall that the coordinate ring of $S^3({\mathbb C}^2/\mu_2)=({\mathfrak h}_{B_3}\oplus{\mathfrak h}^*_{B_3})/W(B_3)$ has a minimal set of fifteen generators.
Let $\rho:\g\rightarrow{\rm End}(V_{\rm min})$ denote the 56-dimensional faithful representation for $\g$.
For a nilpotent element $x\in\g$, the Jordan block sizes of $\rho(x)$ can be determined by including $\g$ in a simple Lie algebra of type $E_8$, which decomposes as $\g\oplus \mathfrak{sl}_2\oplus V_{\rm min}\oplus V_{\rm min}$.
A quicker approach is to inspect the Tables in \cite{Stewart}.
The condition for membership of $\overline\0$ is $\rho(x)^{10}=0$.
(The nilpotent orbits which are minimal among those not contained in $\overline\0$ are those with labels $A_6$ and $D_5$; elements of both of these orbits have two Jordan blocks of size 11.)
Let $\{ h,e,f\}$ be an $\mathfrak{sl}_2$-triple with $e\in\0'$.
The centralizer ${\mathfrak z}$ of $e$ in $\g$ has dimension 27, with reductive centralizer ${\mathfrak c}={\mathfrak g}^e\cap{\mathfrak g}^f$ isomorphic to $\mathfrak{sl}_2$.
The ${\mathfrak c}$-module structure of ${\mathfrak z}$ is as follows:
${\mathfrak z}=V(2)\oplus V(4)\oplus V(6)\oplus W$, where $W$ is a reducible submodule of dimension 12.
Similarly to Thm. \ref{F4excslicethm} and Thm. \ref{E8excslicethm}, we were able to eliminate variables corresponding to basis elements in $W$, thus showing that the intersection ${\mathcal S}\cap\overline\0$ is contained in a subvariety $Y$ of $f+{\mathfrak z}$ which is isomorphic to $V(2)\oplus V(4)\oplus V(6)$.
Note that this precisely corresponds to the structure of a minimal generating set of ${\mathbb C}[{\mathfrak h}_{B_3}\oplus {\mathfrak h}_{B_3}^*]^{W(B_3)}$.
We can now check that the restriction to $Y$ of the condition $\rho(f+z)^{10}=0$ gives precisely the thirty-six relations generating the ideal of relations for ${\mathbb C}[{\mathfrak h}_{B_3}\oplus {\mathfrak h}_{B_3}^*]^{W(B_3)}$, as outlined in \S \ref{wrprodinvs}.
(This requires a certain amount of reparametrization of $Y$.)
It also turns out to be straightforward to directly parametrize (in terms of points in ${\mathfrak h}_{B_3}\oplus{\mathfrak h}_{B_3}^*$) the intersection $Y\cap\overline\0$, providing a closed immersion $S^3({\mathbb C}^2/\mu_2)\rightarrow Y\cap\overline\0$ which, by the usual arguments, must be an isomorphism onto the image.
(The calculation in GAP of the $10$th power of the $56\times 56$ parametrized matrix takes a few minutes.)
\end{proof}

\begin{remark}\label{WB3remark}
a) Similarly to Remark \ref{F4calcremark} and Remark \ref{E8calcremark}, we note the following isomorphism of posets:
$$
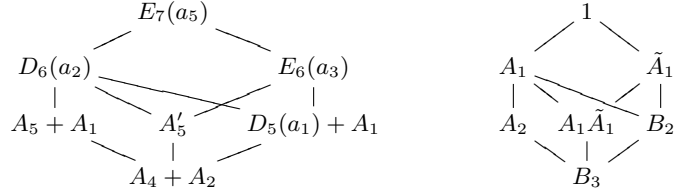
\small{\xymatrix@=.2cm{
{} & E_7(a_5)\ar@{-}[dl]\ar@{-}[dr] & {} &&&& {} & 1 \ar@{-}[dl]\ar@{-}[dr] & {} \\
D_6(a_2)\ar@{-}[d]\ar@{-}[dr]\ar@{-}[drr] & {} & E_6(a_3)\ar@{-}[dl]\ar@{-}[d] &&&& A_1\ar@{-}[d]\ar@{-}[dr]\ar@{-}[drr] & {} & \tilde{A}_1\ar@{-}[dl]\ar@{-}[d] \\
A_5+A_1\ar@{-}[dr]  & A'_5\ar@{-}[d] & D_5(a_1)+A_1\ar@{-}[dl] &&&& A_2 \ar@{-}[dr] & A_1\tilde{A}_1\ar@{-}[d] & B_2\ar@{-}[dl]    \\
{} & A_4+A_2 & {} &&&& {} & B_3 & {}  }}
$$
On the right hand side we denote by $X_n$ a parabolic subgroup of $W(B_n)$; each such parabolic subgroup (up to conjugacy) defines a symplectic leaf, the image in the quotient of the points with stabilizer $W(X_n)$.
The partial order ensures that this isomorphism is unique up to swapping $A_1\tilde{A}_1$ and $B_2$; we verified by computation that the points with stabilizer $W(B_2)$ map to $D_5(a_1)+A_1$.

b) We can (somewhat conjecturally) interpret the Theorem in terms of the affinized ${A}(\0)$-cover $\widetilde\0$ of $\0$.
(Recall that $A(\0)$ denotes the component group of the centralizer of $e$ in the adjoint group of $G$).
Note that the union of the orbits $A_5+A_1$, $D_6(a_2)$ and $E_7(a_5)$ is the special piece containing $\0$; we will prove later that the corresponding Slodowy slice singularity is $({\mathfrak h}_2\oplus {\mathfrak h}_2^*)/\mathfrak{S}_3$.
Lusztig's special pieces conjecture in this case (to be verified in \cite{Fu-Juteau-Levy-Sommers:GeomSP}) says that $\widetilde\0$ is smooth over points of $A_5+A_1$.
Let $\pi:\widetilde\0\rightarrow\overline\0$ be the surjection extending the $A(\0)$-cover of $\0$.
Since $D=\mu_2\times\mu_2\times\mu_2$ is a normal subgroup of $W(B_3)$ with quotient $W(B_3)/D\cong\mathfrak{S}_3$, it is natural to expect $\pi^{-1}({\mathcal S})$ to be isomorphic to $({\mathfrak h}_{B_3}\oplus {\mathfrak h}_{B_3}^*)/D$, i.e. to ${\mathbb C}^2/\mu_2\times {\mathbb C}^2/\mu_2\times {\mathbb C}^2/\mu_2$.

c) A further interesting aspect to this case is that, while $A(\0)=\mathfrak{S}_3$, the {\it fundamental} group of $\0$ is $\mathfrak{S}_3\times\mu_2$.
Thus the affinized {\it universal} cover $\widetilde\0_{\rm uni}$ of $\0$ maps onto $\widetilde\0$ via a $\mu_2$-quotient.
It seems possible that the pre-image of ${\mathcal S}$ in $\widetilde\0_{\rm uni}$ is the quotient of ${\mathbb C}^2\times{\mathbb C}^2\times{\mathbb C}^2$ by the Klein 4-group generated by $(-I,-I,I)$ and $(I,-I,-I)$.
If this is the case then the pre-images of the $E_6(a_3)$ and $A'_5$ orbits become smooth, and the affinized universal cover has only terminal singularities (specifically, a $b_2$ singularity at any point lying above the $D_5(a_1)+A_1$ orbit).
\end{remark}

For our second result of this subsection, we instead let $\g$ be a simple Lie algebra of type $E_6$, with $\0$ the nilpotent orbit with label $E_6(a_3)$ and ${\mathcal S}$ the Slodowy slice at an element of $\0'=\0_{A_4+A_1}$.
Let the cyclic group $\mu_3$ of order 3 act on ${\mathbb C}^2$ via inclusion in the diagonal matrices in $\SL_2$.

\begin{theorem}\label{G(3,1,2)thm}
With ${\mathcal S}$, $\0$ as above, the intersection ${\mathcal S}\cap\overline\0$ is isomorphic to $S^2({\mathbb C}^2/\mu_3)$.
\end{theorem}

\begin{proof}
Let $V$ be the (complex) reflection representation for $\Gamma=G(3,1,2)$.
Recall that the coordinate ring of $S^2({\mathbb C}^2/\mu_3)$ is equal to ${\mathbb C}[V\oplus V^*]^\Gamma$, and it has eight generators.
These generators are of bidegree $(2,0)$, $(3,1)$, $(3,-1)$, $(4,0)$, $(5,1)$, $(5,-1)$, $(6,2)$ and $(6,-2)$ (respectively denoted $t$, $u$, $v$, $\delta^2$, $\delta\xi$, $\delta\eta$, $\xi^2$ and $\eta^2$ in \S \ref{wrprodinvs}).
Let $\rho:\g\rightarrow{\rm End}(V_{27})$ be either of the two irreducible $\g$-representations of dimension 27.
We can determine block sizes $\rho(x)$ for nilpotent $x\in{\mathfrak g}$ by including ${\mathfrak g}$ (in the standard way) as a subalgebra of a simple Lie algebra of type $E_7$, which decomposes as ${\mathfrak g}\oplus V_{27}\oplus V_{27}^*\oplus {\mathbb C}$.
Alternatively, one can inspect the tables in \cite{Stewart} to see that $x\in{\mathfrak g}$ belongs to $\overline\0$ if and only if $\rho(x)^9=0$.

Let $h,e,f$ be an $\mathfrak{sl}_2$-triple with $e\in\0'$.
We outline the structure of the centralizer ${\mathfrak z}$ of $e$, which has dimension 16.
The reductive centralizer $\g^e\cap\g^f={\mathfrak c}={\mathbb C}c$ is a one-dimensional (toral) subalgebra.
The $({\rm ad}\, h)$-action induces a grading on ${\mathfrak z}$, which we denote ${\mathfrak z}=\oplus_{i\geq 0}{\mathfrak z}(i)$.
Since $c\in{\mathfrak c}={\mathfrak z}(0)$, then each graded component ${\mathfrak z}(i)$ has a basis of eigenvectors for ${\rm ad}\, c$.
(We can assume these eigenvalues are integers, and the eigenvalues $\pm 1$ occur.)
Combining the actions of ${\rm ad}\, h$ and ${\rm ad}\, c$ defines a bigrading on ${\mathfrak z}$.
It can be read off from the tables in \cite{Lawther-Testerman} that ${\mathfrak z}=U\oplus W$, where both $U$ and $W$ are $(h,c)$-stable subspaces, $\dim U=\dim W=8$ and $U$ has a basis of $(h,c)$-eigenvectors with eigenvalues: $(0,0), (1,1), (1,-1), (2,0), (3,1), (3,-1), (4,2), (4,-2)$.
Recalling that the canonical contracting ${\mathbb G}_m$-action on the slice is given by $$t\cdot \left(f+\sum_{i\geq 0}z_i\right)=f+\sum_{i\geq 0} t^{i+2} z_i,$$ (where $z_i\in{\mathfrak z}(i)$), we see that a basis for $U$ has precisely the bidegrees corresponding to the above set of generators for ${\mathbb C}[V\oplus V^*]^\Gamma$.

Similarly to the calculations in other cases, we first constructed a basis for ${\mathfrak z}$ in GAP and used the conditions ${\rm tr}(\rho(x)^2)=0$ and $\rho(x)^9=0$ to show that the intersection ${\mathcal S}\cap\overline\0$ is contained in a subvariety $Y$ of ${\mathcal S}$ which is isomorphic to $U$ (in other words, $Y$ is given by a section of the projection $f+U\oplus W\rightarrow U$).
Along the same lines as our calculations in the three previous cases, we next checked that the restriction to $Y$ of the condition $\rho(x)^9=0$ is equivalent to the ideal of relations (with nine generators, as detailed in \S \ref{wrprodinvs}) satisfied by the generators of ${\mathbb C}[V\oplus V^*]^\Gamma$.
It is also (at this stage) a straightforward task to directly parametrize the intersection $Y\cap\overline\0$, thus obtaining an explicit isomorphism $S^2({\mathbb C}^2/\mu_3)\rightarrow Y\cap\overline\0={\mathcal S}\cap\overline\0$.

\end{proof}

\begin{remark}
There are only four symplectic leaves in this case:
$$
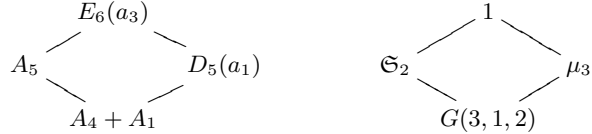
\small{\xymatrix@=.2cm{
{} & E_6(a_3) \ar@{-}[dr] \ar@{-}[dl] & {} &&&& {} & 1 \ar@{-}[dr] \ar@{-}[dl] & {} \\
A_5 \ar@{-}[dr] & {} & D_5(a_1) \ar@{-}[dl] &&&& \mathfrak{S}_2 \ar@{-}[dr] & {} & \mu_3 \ar@{-}[dl] \\
{} & A_4+{A}_1 & {} &&&& {} & G(3,1,2) & {} }}
$$
Note that the singularity attached to each edge is $A_2$, except for the edge $E_6(a_3)>A_5$ (which is an $A_1$ singularity).
Similarly to Remark \ref{WB3remark}, the component group $A(\0)=\mathfrak{S}_2$ in the adjoint group is equal to the quotient $\Gamma/D$, where $D$ is the diagonal subgroup of $\Gamma$ (in this case $\mu_3\times\mu_3$).
The pre-image of ${\mathcal S}$ in the affinized $A(\0)$-cover should therefore be ${\mathbb C}^2/\mu_3\times{\mathbb C}^2/\mu_3$, with the singular locus lying above the closure of the $D_5(a_1)$ orbit.
Another common feature with Remark \ref{WB3remark} is that the fundamental group of $\0$ is $\mathfrak{S}_2\times\mu_3$.
It seems likely that the pre-image of ${\mathcal S}$ in the affinized universal cover of $\0$ should equal ${\mathbb C}^4/\mu_3$ (i.e. points lying above $A_5$ and $D_5(a_1)$ are smooth).
\end{remark}

\subsection{An $\mathfrak{S}_4$-quotient from outer space}\label{outerspace}

In this subsection, let $\g=\mathfrak{sl}_3$.
Recall that $a_2$ denotes the closure of the minimal nilpotent orbit in $\g$.
There are two natural embeddings of $\mathfrak{S}_4$ in the automorphism group of $\g$.
The first of these embeds $\mathfrak{S}_4$ in $\GL_3$, hence in the inner automorphism group of $\mathfrak{sl}_3$, via the $3$-dimensional reflection representation.
(It is easy to see that twisting by the sign representation has no impact on the inclusion in ${\rm Aut}(\mathfrak{sl}_3)$.)
Concretely, we consider the (unique) Klein four subgroup $\Gamma_4$ of $\SL_3$ contained in the diagonal maximal torus.
The Weyl group $\mathfrak{S}_3$ acts on $\Gamma_4$ via the permutation action on the three non-identity elements.
Choosing the permutation matrices as representatives of the Weyl group, we obtain a subgroup of $\GL_3$ isomorphic to $\mathfrak{S}_3\cong \mathfrak{S}_4/\Gamma_4$, which together with $\Gamma_4$ generates a subgroup of the inner automorphism group isomorphic to $\mathfrak{S}_4$.
For short, we call this {\it inner $\mathfrak{S}_4$}.

All elements of $\mathfrak{S}_4$ in the above construction belong to ${\rm O}_3(\C)$ and hence commute with the outer automorphism $\gamma:g\mapsto g^{-T}$.
Replacing each odd permutation of inner $\mathfrak{S}_4$ by its composition with $\gamma$, we obtain a different subgroup of ${\rm Aut}(\g)$ which we call {\it outer $\mathfrak{S}_4$}.
Let us fix this action of $\mathfrak{S}_4$ on $\g$ (and hence on $a_2$).
Let $\0$ (resp. $\0'$) be the nilpotent orbit with label $E_8(a_6)$ (resp. $E_8(b_6)$) and let ${\mathcal S}$ be the Slodowy slice at an element of $\0'$.
We will show that ${\mathcal S}\cap\overline\0$ is isomorphic to $a_2/{\mathfrak S}_4$.

The action on $a_2$ is not free: each $4$-cycle has a 2-dimensional fixed point subset, which is equal to the set of nilpotent elements in a proper Levi subalgebra.
This can be seen by a direct calculation, using the much easier fact that each element of $\Gamma_4$ fixes an $\mathfrak{sl}_2$.
Note however that transpositions act freely on the minimal nilpotent orbit.
(This contrasts with inner $\mathfrak{S}_4$, where the nilpotent elements in a proper Levi subalgebra are fixed by a pair of commuting transpositions, with a $4$-cycle acting by an involution of $\mathfrak{sl}_2$.)

\subsubsection{Generators and relations for $a_2/\mathfrak{S}_4$}

Our first step is to determine the generators and relations of the coordinate ring of $a_2/{\mathfrak S}_4$.
It will be useful to consider the larger set $X$ of elements of $\g$ of rank at most 1, which contains $a_2$ as a divisor.
The elements of $X$ can be written in the form $xy^T$ where $x,y\in{\mathbb C}^3$.
Hence the coordinate ring of $X$ can be identified with $\C[x_i y_j : 1\leq i,j\leq 3]\subset\C[x_1,\ldots ,y_3]$.
With this description, the ideal of $a_2$ is principal, generated by $z:=x_1y_1+x_2y_2+x_3y_3$.
It will be convenient to write the indices $i,j$ modulo $3$.
Note that each non-identity element of $\Gamma_4$ fixes the diagonal entries $x_i y_i$ and one of the pairs $x_i y_{i+1}, x_{i+1} y_i$, and acts as $-1$ on all other $x_i y_j$.
The same argument as Lemma \ref{Gamma4quot} therefore gives us:

\begin{lemma}
a) The $\Gamma_4$-invariants on ${\mathbb C}[X]$ are generated by:
$$\{ b_{i} = x_iy_i, c_i = x_{i+1}^2 y_{i-1}^2, d_i = x_{i-1}^2 y_{i+1}^2 : i\in{\mathbb Z}/(3)\}.$$

b) Let $\lambda_i=b_{i+1}-b_{i-1}$ for $i\in{\mathbb Z}/(3)$.
(In particular, $\sum_i \lambda_i=0$.)
The action of $\mathfrak{S}_4/\Gamma_4\cong\mathfrak{S}_3$ on ${\mathbb C}[X]^{\Gamma_4}$ is as follows:

 $\bullet$ $z=b_1+b_2+b_3$ spans a copy of the sign representation;

 $\bullet$ $\mathfrak{S}_3$ acts by permuting the columns of the matrix:
$$\begin{pmatrix} \lambda_1 & \lambda_2 & \lambda_3 \\ c_1 & c_2 & c_3 \\ d_1 & d_2 & d_3 \end{pmatrix}.$$
\end{lemma}

Note the similarity of the description in (b) with the proof of Thm. \ref{d4S4lem}.
We can now easily list generators for ${\mathbb C}[a_2]^{{\mathfrak S}_4}$: these come from the elementary invariants in the $\lambda_i$ and their polarisations $\sum_i c_i$, $\sum_i \lambda_i c_i$, $\sum_i c_i^2 d_i$ etc.
In fact, several of these turn out to be redundant.
To obtain the relations between them, we exploit the isomorphism ${\mathbb C}[a_2]^{{\mathfrak S}_4}\cong {\mathbb C}[X]^{{\mathfrak S}_4}/{\mathbb C}[X]^{{\mathfrak S}_4}\cap(z)$.
The ideal ${\mathbb C}[X]^{{\mathfrak S}_4}\cap(z)$ is simply the set of all $zh$ with $h\in{\mathbb C}[X]^{{\mathfrak A}_4}$ a semi-invariant for ${\mathfrak S}_4$.

Thus we work in the ring ${\mathbb C}[X]^{{\mathfrak A}_4}$.
It is not difficult to obtain an exhaustive list of generators.
Those which are invariant for ${\mathfrak S}_4$ can be obtained by polarisations as above; there are 18 of them.
Those which are semi-invariant for ${\mathfrak S}_4$ can be obtained similarly, by considering eigenvectors for a $3$-cycle.
Apart from the linear term $z$, there are quadratic semi-invariants such as $\sum \lambda_i (b_{i+1}-b_{i-1})$, and cubic terms $\prod_{i<j} (\lambda_i-\lambda_j)$ along with similar mixed cubic terms in $\lambda_i$, $c_i$, $d_i$.
There are $14$ generating semi-invariants.

Calculations in GAP allowed us to eliminate ten of the invariants and seven of the semi-invariants.
Hence ${\mathbb C}[X]^{{\mathfrak A}_4}$ (resp. ${\mathbb C}[a_2]^{{\mathfrak S}_4}$) has fifteen (resp. eight) generators.
Note that there is an action of a cyclic group of order 3 on $a_2/\mathfrak{S}_4$, coming from the element ${\rm diag}(\omega, 1, \omega^{-1})\in\SL_3$ (which normalizes $\mathfrak{S}_4$ in ${\rm Aut}(\g)$); then we can choose homogeneous generators of ${\mathbb C}[a_2]^{{\mathfrak S}_4}$ which are eigenvectors for this action.
Denote the generators $$f_{0},f_{1},f_{2},g_{0},g_{1},g_{2},h_{1},h_{2}$$
where the $f_i$ (resp. $g_i$, $h_i$) are of degree $4$ (resp. $6$, $8$) and the subscript $i$ indicates a $\omega^i$-eigenvector for the above element of order 3.

We used GAP to find relations among the generators of ${\mathbb C}[X]^{{\mathfrak A}_4}$.
Setting $z=0$, we thus obtained the following nine elements of the ideal of relations:
$$5f_0^3+108f_0f_1f_2-216f_1h_2-216f_2h_1+24g_0^2-216g_1g_2, \quad 8f_0^2f_1-15f_0h_1-18f_2h_2+6g_0g_1+18g_2^2,$$
$$8f_0^2f_2-15f_0h_2-18f_1h_1+6g_0g_2+18g_1^2, \quad f_0^2g_1-10f_0f_1g_0+18f_0f_2g_2+12f_2^2g_0+12g_0h_1-36g_2h_2,$$
$$f_0^2g_0-4f_1f_2g_0-12g_1h_2-12g_2h_1, \quad f_0^2g_2-10f_0f_2g_0+18f_0f_1g_1+12f_1^2g_0+12g_0h_2-36g_1h_1,$$
$$f_0^4-68f_0^2f_1f_2-48f_0g_1g_2+144f_1^2h_1-144f_1g_1^2+144g_2^2h_2-144f_2g_2^2+144h_1h_2,$$
$$9f_0^3f_1+32f_0^2f_2^2-36f_0f_1^2f_2-18f_0^2h_1-120f_0f_2h_2+24f_0g_2^2+72f_1^2h_1+72f_1g_1g_2+72f_2g_1^2+72h_2^2,$$
$$9f_0^3f_2+32f_0^2f_1^2-36f_0f_1f_2^2-18f_0^2h_2-120f_0f_1h_1+24f_0g_1^2+72f_2^2h_1+72f_2g_1g_2+72f_1g_2^2+72h_1^2.$$

We then checked using Magma that the ideal generated by these relations is prime, of height 4.
Hence this is a complete list of relations.
(We note once again the coincidence with ${\mathbb C}[a_2]$, where there are eight generators and nine relations.)
This preparation allowed us to computationally verify the following:

\begin{theorem}\label{a2S4thm}
Let $\0'$ and $\0$ be the nilpotent orbits in $E_8$ with respective Bala-Carter labels $E_8(b_6)$ and $E_8(a_6)$, and let ${\mathcal S}$ be the Slodowy slice to an element $f\in\0'$.
Then ${\mathcal S}\cap\overline{\0}$ is isomorphic to $a_2/\mathfrak{S}_4$.
\end{theorem}

\begin{proof}
To verify this, we used the fact that an element $x\in{\mathcal S}$ belongs to $\overline{\0}$ if and only if $(\ad x)^{19}=0$.
Although the Slodowy slice is of dimension $28$, it was not difficult (in GAP) to find an $8$-dimensional subset, isomorphic to ${\mathbb A}^8$, which must contain ${\mathcal S}\cap\overline{\0}$.
By carefully choosing coordinate functions on the sub-slice and identifying them with $f_0, \ldots , h_2$, we were able to determine that the relations above are {\it necessary} conditions for the $19$-th power to equal zero.
This implies that the relations are also sufficient conditions, since $\0'$ is of codimension $4$ in $\overline{\0}$ (and $\overline{\0}$ is unibranch at $f$).

The phrase `must contain' in the above paragraph is perhaps slightly unsatisfactory.
To be a bit more transparent, we have also verified in GAP that the conditions are sufficient.
View $f_0, \ldots , h_2$ as functions on $a_2$, hence we have a map $\Theta:a_2\rightarrow{\mathcal S}$ which factors through the quotient by $\mathfrak{S}_4$.
We want to check that the image of $\Theta$ is contained in $\overline{\0}$.
Similarly to Thm. \ref{E8excslicethm}, GAP runs out of memory when computing the $19$-th power of this matrix; we have verified column by column that the $19$-th power equals zero.
A much quicker probabilistic proof can be obtained by checking that $\Theta(g)=0$ for a few specific $g$ in $a_2$.
\end{proof}

\begin{remark}
The quotient $a_2/\mathfrak{S}_4$ and similar ``combined ${\mathfrak S}_n$ and ${\mathbb Z}_q$ quotients'' have been studied in the context of Coulomb branches of quiver gauge theories in \cite{hanany2023actions}.
\end{remark}

\subsubsection{A non-normal minimal degeneration}

As in the previous subsection, let $\0$ (resp. $\0'$) be the orbit with Bala-Carter label $E_8(a_6)$ (resp. $E_8(b_6)$).
Recall that $\0''=\0_{D_7(a_1)}$ is the only orbit intermediate between $\0'$ and $\0$.
In \cite{FJLS} we asserted that ${\mathcal S}\cap\overline{\0''}$ is isomorphic to the non-normal singularity $\mu={\rm Spec}\, {\mathbb C}[s^2t^2, s^3t^3, s^4, t^4, s^5t, st^5]$.
Let us show how this follows from Thm. \ref{a2S4thm}.
The nilpotent elements of $a_2$ with non-trivial stabilizer subgroup in $\mathfrak{S}_4$ are precisely those fixed by the 4-cycles, hence are the nilpotent elements lying in the three type $A_1$ Levi subalgebras.
It follows that the singular locus of $a_2/\mathfrak{S}_4$ is isomorphic to the image of the quotient homomorphism applied to (the nilpotent elements in) one of these Levi subalgebras.
We fix the following choice: $$N_{st}=\begin{pmatrix} st & s^2 & 0 \\ -t^2 & -st & 0 \\ 0 & 0 & 0 \end{pmatrix} : s,t\in{\mathbb C}.$$
Recalling the generators of the $\Gamma_4$-invariants, we can see that we have $$\begin{pmatrix} \lambda_1 &\lambda_2 & \lambda_3 \\ c_1 & c_2 & c_3 \\ d_1 & d_2 & d_3 \end{pmatrix} = \begin{pmatrix} -st & -st & 2st \\ 0 & 0 & s^4 \\ 0 & 0 & t^4 \end{pmatrix}.$$
In particular, we obtain the following values of the symmetric functions in the columns of the matrix:
$$\sum_i\lambda_i^2=6s^2t^2, \quad \sum_i\lambda_i^3=6s^3t^3, \quad \sum_i c_i=s^4,\quad\sum_i d_i=t^4, \quad \sum_i\lambda_i c_i=2s^5 t, \quad \sum_i \lambda_i d_i = 2st^5,$$
and it is easy to see that all other symmetric functions, evaluated at $N_{st}$, belong to ${\mathbb C}[s^2t^2,s^3t^3,s^4,t^4,s^5t,st^5]$.

\begin{corollary}
With the notation of Thm. \ref{a2S4thm}, ${\mathcal S}\cap\overline{\0''}$ is isomorphic to $\mu$.
\end{corollary}

\begin{proof}
This follows from the fact that the singular locus of $\overline{\0}$ equals $\overline{\0}\setminus\0=\overline{\0''}$.
\end{proof}

The normalization of $\mu$ is easily seen to be ${\rm Spec}\,{\mathbb C}[st,s^4,t^4]$, i.e. an $A_3$ surface singularity.
This is the only known degeneration of special orbits with a non-normal unibranch singularity.

\subsubsection{Monodromy action}

The component group $\mathfrak{S}_3$ of $\0'$ acts on the Slodowy slice, and therefore on $a_2/\mathfrak{S}_4$.
In fact, $N_{{\rm Aut} {\mathfrak g}}(\mathfrak{S}_4)/\mathfrak{S}_4$ is isomorphic to $\mathfrak{S}_3$, and the monodromy action coincides with the action of the normalizer.
Concretely, ${\mathfrak A}_3$ is generated by conjugation by ${\rm diag}(\omega,1,\omega^2)$, and $g\mapsto -g^T$ is a transposition in $\mathfrak{S}_3$.

Similarly, we have an action of $\mathfrak{S}_3$ on $\mu$: taking the normalization, this is nothing other than the action of $\mathfrak{S}_3$ on a simple singularity of type $A_3$, via the inclusion of a cyclic subgroup of $\SL_2$ of order $4$ in a binary dihedral subgroup of order $24$.
It is easy to check that $\mu/\mathfrak{S}_3$ is normal.


\section{Slodowy slices in special pieces}\label{Section:locallspconj}

We are now ready to prove the local version of Lusztig's special pieces conjecture.
Let ${\mathcal P}(\0)$ be a special piece in a simple Lie algebra, let $\0_{m}$ be the unique minimal orbit in ${\mathcal P}(\0)$, let ${\mathcal S}_m$ be the Slodowy slice at an element of $\0_{m}$ and let $H$ be the subgroup of $\bar{A}(\0)$ defined by Lusztig in \cite{Lusztig:unipotent}.
Denote by ${\mathfrak h}_{n-1}$ the reflection representation for $\mathfrak{S}_n$.



\begin{table}[htp] 
	\caption{The non-trivial special pieces in the exceptional types}\label{exceptional_special_pieces_table}
\begin{center}
\begin{tabular}{|c | c| c| c| c|}
\hline
$(n,k)$ & sing. & $\g$ & pairs $(\0,\0_m)$ \\
\hline
$(2,1)$ & $A_1$ & $E_6$ & $(E_6(a_3),A_5)$, $(A_2,3A_1)$ \\
 & & $E_7$ & $(E_7(a_3), D_6)$, $(E_6(a_3),A'_5)$, $(D_4(a_1)+A_1,A_3+2A_1)$, $(A_2,(3A_1)')$ \\
 & & $E_8$ & $(E_8(a_3), E_7)$, $(E_8(a_5), D_7)$, $(E_8(b_6),A_7)$, $(D_6(a_1),D_5+A_1)$, \\
 & & &  $(E_6(a_3),A_5)$, $(D_4(a_1)+A_2,A_3+A_2+A_1)$, $(2A_2,A_2+3A_1)$, $(A_2,3A_1)$ \\
 \hline
 $(2,2)$ & $c_2$ & $E_7$ & $(D_5(a_1),D_4+A_1)$ \\
 & & $E_8$ & $(E_7(a_3),D_6)$, $(A_4+2A_1,2A_3)$ \\
 \hline
 $(2,3)$ & $c_3$ & $F_4$ & $(\tilde{A}_1,A_1)$ \\
 & & $E_7$ & $(A_2+A_1,4A_1)$ \\
 & & $E_8$ & $(D_5(a_1),D_4+A_1)$ \\
 \hline
 $(2,4)$ & $c_4$ & $E_8$ & $(A_2+A_1,4A_1)$ \\
 \hline
 $(3,1)$ & ${\mathbb C}^4/\mathfrak{S}_3$ & $G_2$ & $(G_2(a_1),A_1)$ \\
 & & $E_6$ & $(D_4(a_1),2A_2+A_1)$ \\
 & & $E_7$ & $(E_7(a_5),A_5+A_1)$, $(D_4(a_1),2A_2+A_1)$ \\
 & & $E_8$ & $(E_8(b_5),E_6+A_1)$, $(D_4(a_1),2A_2+A_1)$ \\
 \hline
 $(3,2)$ & ${\mathbb C}^8/\mathfrak{S}_3$ & $E_8$ & $(D_4(a_1)+A_1,2A_2+2A_1)$ \\
 \hline
 $(4,1)$ & ${\mathbb C}^6/\mathfrak{S}_4$ & $F_4$ & $(F_4(a_3),A_2+\tilde{A}_1)$ \\
 \hline
 $(5,1)$ & ${\mathbb C}^8/\mathfrak{S}_5$ & $E_8$ & $(E_8(a_7),A_4+A_3)$ \\
 \hline
\end{tabular}
\end{center}
\end{table}

\begin{namedtheorem}[Main Theorem]\label{localLSPthm}
For $\g$ exceptional, the intersection ${\mathcal S}_m\cap {\mathcal P}(\0)$ is isomorphic to 
$$({\mathfrak h}_{n-1}\oplus{\mathfrak h}_{n-1}^*)^k/\mathfrak{S}_{n}$$
 for some $k$ and $n$ (and $H=\mathfrak{S}_n$).

For $\g$ classical, let $\0_1, \0_2, \dots \0_r$ be minimal degenerations of $\0$ contained in ${\mathcal P}(\0)$.
Then $H$ is a product of $r$ copies of $\mathfrak{S}_{2}$.  Let $k_i = \frac{1}{2} \codim_{\overline \0} \0_i$.
Then  ${\mathcal S}_m\cap {\mathcal P}(\0)$ is isomorphic to 
$$\prod_{i=1}^r ({\mathfrak h}_{1}\oplus{\mathfrak h}_{1}^*)^{k_i}/\mathfrak{S}_{2}.$$

In both cases, the slice is isomorphic to $\C^d/H$ where $d$ is the codimension of $\0_m$ in $\overline{\0}$ and the action
of $H$ is the one described in the above isomorphisms.
\end{namedtheorem}

\begin{proof}
{\bf Case: $\g$ exceptional.}  When $\g$ is of exceptional type, the group appearing in Lusztig's special pieces conjecture, and hence in our Main Theorem, is $\mathfrak{S}_n$ for some $n\leq 5$.  
There is nothing to prove if ${\mathcal P}(\0)=\0$ (in which case $H$ is trivial).
If $H=\mathfrak{S}_2$ then there are two orbits in ${\mathcal P}(\0)$, hence $\0>\0_m$ is a minimal degeneration and we only need to check that the Slodowy slice singularity is a minimal singularity of type $C$.
This follows from our earlier work, see \cite[\S 13]{FJLS}.
The statement when $H=\mathfrak{S}_4$ (resp. $\mathfrak{S}_5$) is Thm. \ref{F4excslicethm} (resp. Thm. \ref{E8excslicethm}).
If $\0=\0_{D_4(a_1)+A_1}$ in type $E_8$ then $H=\mathfrak{S}_3$ and this was proved in Lemma \ref{2A2+2A1lem}.
(This is the only case with $n>2$ and $k>1$.)
If $\0$ is the subregular orbit in type $G_2$ then this is Lemma \ref{G2minlem}.
This leaves one special piece in $E_6$, two in $E_7$ and two in $E_8$.
We show that all of the remaining Slodowy slice singularities are isomorphic to the slice from the minimal orbit to the subregular orbit closure in $G_2$ (hence to $({\mathfrak h}_2\oplus{\mathfrak h}_2^*)/\mathfrak{S}_3$).
To see this, we note that there are always three orbits in the special piece and (for these cases) there always exists a minimal special degeneration $\0>\0'$ such that $\0_m>\0'$ is a minimal degeneration.
Furthermore, we can assume that $\0'$ is not $A_2+3A_1$ in type $E_7$.
(When $\0=\0_{D_4(a_1)}$ we choose $\0'=\0_{2A_2}$ instead.)
This ensures the following: if ${\mathcal S}=f+\g^e$ is the Slodowy slice at an element $f\in\0'$, then ${\mathcal S}\cap\overline\0=f+\overline{C_0\cdot e_0}$, where $e_0$ is a subregular element of $\cg_0$, the simple component of $\g^e\cap\g^h$ of type $G_2$.
Note in particular that if $\{ e_m,h_m,f_m\}\subset\cg_0$ is an $\mathfrak{sl}_2$-triple with $e_m$ a long root element of $\cg_0$, then $f+f_0\in\0_m$ and, by the argument in \cite[Prop. 4.8]{FJLS}, $f+f_m+\g^e\cap\g^{e_m}$ is a Slodowy slice at $f+f_0$.
It follows that the Slodowy slice singularity from $\0_m$ to $\overline\0$ is isomorphic to the Slodowy slice singularity from the minimal orbit to the subregular orbit closure in $G_2$.

{\bf Case: $\g$ classical.}  
For generalities on the classification of nilpotent orbits via $\epsilon$-partitions, including criteria for specialness, we refer to \cite[Ch. 5]{C-M}.
For a partition $\lambda$ of $n$ and an integer $s$, let $m(s)$ (resp. $h(s)$) denote the multiplicity (resp. height, i.e. $\sum_{j\geq s}m(j)$) of $s$.
Let $\0$ be a special orbit and let $\lambda$ be the partition of $\0$.
Let $\0'$ be a non-special orbit in the piece, with partition $\nu$.  Pick $e \in \0'$.
As will be explained in \cite{JLS:Duality}, it follows from the classification of special orbits (detailed in \cite[\S 6.3]{C-M}) that $\lambda$ is obtained from $\nu$ by replacing occurrences of $[s^{m(s)}]$ in $\nu$ with parts $[s{+}1, s^{m(s)-2},s{-}1]$ in $\lambda$, where $s \equiv \epsilon$ and $h(s) \not \equiv \epsilon'$.
Note that for each such part $[s^{m(s)}]$ there is a simple factor $\mathfrak{sp}_m(s)$ of the reductive centralizer of $e$.
Choose a minimal nilpotent element in each such simple factor, and let $e_0$ be their sum.
Then $e_0$ is of height 2 in $\g$, hence \cite[Cor. 4.9 and Lemma 4.3]{FJLS} apply.
It is also now clear (by considering $\mathfrak{sl}_2$-triples) that $f+e_0$ belongs to $\0$, hence the slice is equal to $\overline{C(e)\cdot (f+e_0)}\cong \overline{C(e)\cdot e_0}$.
The result therefore holds for any orbit in the special piece, including $\0_m$, since the minimal orbit closure in $\mathfrak{sp}(2k_i)$ is isomorphic to $\C^{2k_i}/\mathfrak{S}_2$ with the nontrivial elements acting by $-1$, as noted before.
In the case of $\0_m$, we have $r$ is the number of factors of $\mathfrak{S}_2$ in $H$, completing the proof.
\end{proof}

\begin{remark}\label{orderremark}
It follows from the Main Theorem that there is an order-reversing bijection between the strata in ${\mathcal S}_m\cap{\mathcal P}(\0)$, hence the nilpotent orbits contained in ${\mathcal P}(\0)$, and the parabolic subgroups of $H$ up to conjugacy.
This bijection is identical with the one described by Lusztig in \cite[\S 6]{Lusztig:unipotent}.
Indeed, since both bijections are order-reversing, this is immediate in classical types or when $H=\mathfrak{S}_2$ or $\mathfrak{S}_3$.
Hence, one only has to inspect the special pieces containing $F_4(a_3)$ and $E_8(a_7)$, see Remarks \ref{F4calcremark}(b) and \ref{E8calcremark}(b).
\end{remark}

In Table \ref{exceptional_special_pieces_table}, we list values of $n$ and $k$ for all special pieces in exceptional types with more than one orbit.

\bibliographystyle{alpha}
\bibliography{fjls}

\quad \vspace{0.5cm}

 Baohua Fu (\email{bhfu@math.ac.cn})

AMSS, HLM and MCM, Chinese Academy of Sciences, 5[Thm. 5.5]5 ZhongGuanCun East Road, Beijing, 100190, China and School of Mathematical Sciences, University of Chinese Academy of Sciences, Beijing, China
 \vspace{0.3 cm}

 Daniel Juteau (\email{daniel.juteau@u-picardie.fr} )

LAMFA, Universit\'e de Picardie Jules Verne, CNRS, Amiens, France
\vspace{0.3 cm}

Paul Levy (\email{p.d.levy@lancaster.ac.uk})

Department of Mathematics and Statistics Fylde College,
Lancaster University, Lancaster, LA1 4YF, United Kingdom
\vspace{0.3 cm}

 Eric Sommers(\email{esommers@umass.edu})

Department of Mathematics and Statistics, University of Massachusetts Amherst, Amherst,
MA 01003-4515, USA

\end{document}